\documentclass[11pt]{article}
\usepackage{array}
\usepackage{amsfonts}
\usepackage{amscd}
\usepackage{amssymb}
\usepackage{amsthm}
\usepackage{amsmath}
\usepackage{stmaryrd}
\usepackage{graphicx}
\usepackage{color}
\usepackage{verbatim}
\usepackage{mathrsfs}
\usepackage{tikz}
\usetikzlibrary{calc}
\usepackage{caption}
\usepackage{float}
\usepackage[neveradjust]{paralist}

 \theoremstyle{plain}
 \newtheorem{thm1}{Theorem}
 \newtheorem{cor1}[thm1]{Corollary}
\newtheorem{thm}{Theorem}[section]
\newtheorem{lemma}[thm]{Lemma}
\newtheorem{prop}[thm]{Proposition}
\newtheorem{cor}[thm]{Corollary}

\theoremstyle{definition}
\newtheorem{defn}[thm]{Definition}
\newtheorem{remark}[thm]{Remark}
\newtheorem{example}[thm]{Example}

\numberwithin{equation}{section}

\def\sA{\mathsf{A}}
\def\sB{\mathsf{B}}
\def\sC{\mathsf{C}}
\def\sD{\mathsf{D}}
\def\sE{\mathsf{E}}
\def\sF{\mathsf{F}}
\def\sG{\mathsf{G}}

\def\sX{\mathsf{X}}

\def\cA{\mathcal{A}}

\def\cC{\mathcal{C}}

\def\cS{\mathcal{S}}

\def\FF{\mathbb{F}}

\def\ZZ{\mathbb{Z}}

\def\K{\mathbb{K}}
\def\PG{\mathsf{PG}}

\DeclareMathOperator\Type{\mathrm{Typ}}
\DeclareMathOperator\Res{\mathrm{Res}}

\DeclareMathOperator\Opp{\mathrm{Opp}}
\DeclareMathOperator\Diag{\mathrm{Diag}}
\DeclareMathOperator\disp{\mathrm{disp}}
\DeclareMathOperator\height{\mathrm{ht}}

\def\la{\lambda}

\def\<{\langle}
\def\>{\rangle}

\makeatletter
\renewcommand{\@makefnmark}{\mbox{\textsuperscript{}}}
\makeatother

\makeatletter
\def\keywords{\xdef\@thefnmark{}\@footnotetext}
\makeatother

\makeatletter
\def\MSC{\xdef\@thefnmark{}\@footnotetext}
\makeatother

\title{Automorphisms and opposition in spherical buildings of exceptional type, I}
\author{James Parkinson 
\and
Hendrik Van Maldeghem}
\date{\today}

\begin{document}

\maketitle

\begin{abstract} 
To each automorphism of a spherical building there is naturally associated an \textit{opposition diagram}, which encodes the types of the simplices of the building that are mapped onto opposite simplices. If no chamber (that is, no maximal simplex) of the building is mapped onto an opposite chamber then the automorphism is called \textit{domestic}. In this paper we give the complete classification of domestic automorphisms of split spherical buildings of types $\sE_6$, $\sF_4$, and $\sG_2$. Moreover, for all split spherical buildings of exceptional type we classify (i) the domestic homologies, (ii) the opposition diagrams arising from elements of the standard unipotent subgroup of the Chevalley group, and (iii) the automorphisms with opposition diagrams with at most $2$ distinguished orbits encircled. Our results provide unexpected characterisations of long root elations and products of perpendicular long root elations in long root geometries, and analogues of the density theorem for connected linear algebraic groups in the setting of Chevalley groups over arbitrary fields.
\end{abstract}
\MSC{2010 Mathematics Subject Classification: 20E42, 51E24, 51B25, 20E45}

\section*{Introduction}

The study of the geometry of fixed elements of automorphisms of spherical buildings is a well-established and beautiful topic (see~\cite{PMW:15}). Over the past decade a complementary theory concerning the ``opposite geometry'', consisting of those elements mapped to opposite elements by an automorphism of a spherical building, has been developed. A starting point for this theory is the fundamental result of Abramenko and Brown~\cite[Proposition 4.2]{AB:09}, stating that if $\theta$ is a nontrivial automorphism of a thick spherical building then the opposite geometry $\Opp(\theta)$ is necessarily nonempty. Indeed the generic situation is that $\Opp(\theta)$ is rather large, and typically contains many chambers of the building (\textit{chambers} are the simplices of maximal dimension). The more special situation is when $\Opp(\theta)$ contains no chamber, in which case $\theta$ is called \textit{domestic}.

Domestic automorphisms have recently enjoyed extensive investigation, see \cite{PTM:15,PVM:19a,PVM:19b,TTM:11,TTM:12,TTM:12b,HVM:12,HVM:13}. Cumulatively these papers illuminate an intimate, and as yet not fully understood, connection between domesticity and large, rich fixed subconfigurations. For example, by~\cite{HVM:12} the domestic dualities of large $\sE_6$ buildings are precisely the polarities that fix a split building of type~$\sF_4$, and by \cite{HVM:13} the domestic trialities of $\sD_4$ buildings are precisely the automorphisms fixing a split building of type~$\sG_2$. These remarkable connections underscore the importance of both the opposite geometry and the notion of domesticity in the theory of spherical buildings. 

A systematic study of the opposite geometry was initiated in~\cite{PVM:19a,PVM:19b}, where we developed the notion of an \textit{opposition diagram} of an automorphism, encoding the types of the simplices of the building that are mapped onto opposite simplices by the automorphism. This concept gives a useful framework for the study of the opposite geometry and domesticity. Indeed, a striking consequence of the theory is that there are surprisingly few opposition diagrams possible. The purpose of this paper (along with \cite{PVM:20a,PVM:20c}) is to classify, as much as possible, the class of automorphisms with each diagram, with the focus of this paper being on split spherical buildings of exceptional type.

Let us briefly expand on the above concepts, before summarising our main results. Suppose that $\Delta$ is an irreducible split spherical building with Dynkin diagram~$\Gamma$. The opposition diagram $\Diag(\theta)$ of an automorphism $\theta$ of $\Delta$ is drawn by encircling the nodes of $\Gamma$ corresponding to the types of the minimal simplices of $\Delta$ that are mapped onto opposite simplices by~$\theta$. 

For example, the diagram 
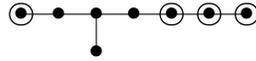
\begin{center}
\begin{tikzpicture}[scale=0.5,baseline=-0.75ex]
\node [inner sep=0.8pt,outer sep=0.8pt] at (-2,0) (1) {$\bullet$};
\node [inner sep=0.8pt,outer sep=0.8pt] at (-1,0) (3) {$\bullet$};
\node [inner sep=0.8pt,outer sep=0.8pt] at (0,0) (4) {$\bullet$};
\node [inner sep=0.8pt,outer sep=0.8pt] at (1,0) (5) {$\bullet$};
\node [inner sep=0.8pt,outer sep=0.8pt] at (2,0) (6) {$\bullet$};
\node [inner sep=0.8pt,outer sep=0.8pt] at (3,0) (7) {$\bullet$};
\node [inner sep=0.8pt,outer sep=0.8pt] at (4,0) (8) {$\bullet$};
\node [inner sep=0.8pt,outer sep=0.8pt] at (0,-1) (2) {$\bullet$};
\draw (-2,0)--(4,0);
\draw (0,0)--(0,-1);
\draw [line width=0.5pt,line cap=round,rounded corners] (1.north west)  rectangle (1.south east);
\draw [line width=0.5pt,line cap=round,rounded corners] (6.north west)  rectangle (6.south east);
\draw [line width=0.5pt,line cap=round,rounded corners] (7.north west)  rectangle (7.south east);
\draw [line width=0.5pt,line cap=round,rounded corners] (8.north west)  rectangle (8.south east);
\end{tikzpicture}
\end{center}
represents an automorphism of an~$\sE_8$ building mapping vertices of types $1,6,7,8$, and no vertices of other types, to opposite vertices (we adopt Bourbaki labelling~\cite{Bou:02}). A priori, there could be $2^8$ possible opposition diagrams for automorphisms of $\sE_8$ buildings, however it is a remarkable fact that there are only $5$ diagrams possible. The idea behind the proof of this fact, from \cite{PVM:19a,PVM:19b}, is as follows. Suppose first that $\Delta$ is a \textit{large} spherical building of rank at least $3$ (meaning that $\Delta$ has no Fano plane residues). In \cite[Theorem~1]{PVM:19a} we showed that every automorphism $\theta$ of $\Delta$ satisfies the following closure property: If there exist type $J_1$ and $J_2$ simplices in $\Opp(\theta)$, then there exists a type $J_1\cup J_2$ simplex in $\Opp(\theta)$. Such automorphisms are called \textit{capped}, and this highly nontrivial property imposes severe constraints on the structure of opposition diagrams. For \textit{small} spherical buildings it turns out that automorphisms are not necessarily capped, however the same constraints on the opposition diagrams exist for other reasons (see~\cite{PVM:19b}).

We call a diagram satisfying the constraints imposed by cappedness an \textit{admissible diagram}. The precise constraints are not required here (see \cite[\S2.1]{PVM:19a} for details), as it is sufficient for our purpose to simply give the complete list of admissible Dynkin diagrams of exceptional type:

\begin{figure}[H]
\begin{center}
${^2}\sE_{6;0}=\begin{tikzpicture}[scale=0.5,baseline=-0.5ex]
\node [inner sep=0.8pt,outer sep=0.8pt] at (-2,0) (2) {$\bullet$};
\node [inner sep=0.8pt,outer sep=0.8pt] at (-1,0) (4) {$\bullet$};
\node [inner sep=0.8pt,outer sep=0.8pt] at (0,-0.5) (5) {$\bullet$};
\node [inner sep=0.8pt,outer sep=0.8pt] at (0,0.5) (3) {$\bullet$};
\node [inner sep=0.8pt,outer sep=0.8pt] at (1,-0.5) (6) {$\bullet$};
\node [inner sep=0.8pt,outer sep=0.8pt] at (1,0.5) (1) {$\bullet$};
\draw (-2,0)--(-1,0);
\draw (-1,0) to [bend left=45] (0,0.5);
\draw (-1,0) to [bend right=45] (0,-0.5);
\draw (0,0.5)--(1,0.5);
\draw (0,-0.5)--(1,-0.5);
\end{tikzpicture}$\qquad
${^2}\sE_{6;1}=\begin{tikzpicture}[scale=0.5,baseline=-0.5ex]
\node at (0,0.8) {};
\node at (0,-0.8) {};
\node [inner sep=0.8pt,outer sep=0.8pt] at (-2,0) (2) {$\bullet$};
\node [inner sep=0.8pt,outer sep=0.8pt] at (-1,0) (4) {$\bullet$};
\node [inner sep=0.8pt,outer sep=0.8pt] at (0,-0.5) (5) {$\bullet$};
\node [inner sep=0.8pt,outer sep=0.8pt] at (0,0.5) (3) {$\bullet$};
\node [inner sep=0.8pt,outer sep=0.8pt] at (1,-0.5) (6) {$\bullet$};
\node [inner sep=0.8pt,outer sep=0.8pt] at (1,0.5) (1) {$\bullet$};
\draw (-2,0)--(-1,0);
\draw (-1,0) to [bend left=45] (0,0.5);
\draw (-1,0) to [bend right=45] (0,-0.5);
\draw (0,0.5)--(1,0.5);
\draw (0,-0.5)--(1,-0.5);
\draw [line width=0.5pt,line cap=round,rounded corners] (2.north west)  rectangle (2.south east);
\end{tikzpicture}$\qquad
${^2}\sE_{6;2}=\begin{tikzpicture}[scale=0.5,baseline=-0.5ex]
\node at (0,0.8) {};
\node at (0,-0.8) {};
\node [inner sep=0.8pt,outer sep=0.8pt] at (-2,0) (2) {$\bullet$};
\node [inner sep=0.8pt,outer sep=0.8pt] at (-1,0) (4) {$\bullet$};
\node [inner sep=0.8pt,outer sep=0.8pt] at (0,-0.5) (5) {$\bullet$};
\node [inner sep=0.8pt,outer sep=0.8pt] at (0,0.5) (3) {$\bullet$};
\node [inner sep=0.8pt,outer sep=0.8pt] at (1,-0.5) (6) {$\bullet$};
\node [inner sep=0.8pt,outer sep=0.8pt] at (1,0.5) (1) {$\bullet$};
\draw (-2,0)--(-1,0);
\draw (-1,0) to [bend left=45] (0,0.5);
\draw (-1,0) to [bend right=45] (0,-0.5);
\draw (0,0.5)--(1,0.5);
\draw (0,-0.5)--(1,-0.5);
\draw [line width=0.5pt,line cap=round,rounded corners] (2.north west)  rectangle (2.south east);
\draw [line width=0.5pt,line cap=round,rounded corners] (1.north west)  rectangle (6.south east);
\end{tikzpicture}$\qquad
${^2}\sE_{6;4}=\begin{tikzpicture}[scale=0.5,baseline=-0.5ex]
\node at (0,0.8) {};
\node at (0,-0.8) {};
\node [inner sep=0.8pt,outer sep=0.8pt] at (-2,0) (2) {$\bullet$};
\node [inner sep=0.8pt,outer sep=0.8pt] at (-1,0) (4) {$\bullet$};
\node [inner sep=0.8pt,outer sep=0.8pt] at (0,-0.5) (5) {$\bullet$};
\node [inner sep=0.8pt,outer sep=0.8pt] at (0,0.5) (3) {$\bullet$};
\node [inner sep=0.8pt,outer sep=0.8pt] at (1,-0.5) (6) {$\bullet$};
\node [inner sep=0.8pt,outer sep=0.8pt] at (1,0.5) (1) {$\bullet$};
\draw (-2,0)--(-1,0);
\draw (-1,0) to [bend left=45] (0,0.5);
\draw (-1,0) to [bend right=45] (0,-0.5);
\draw (0,0.5)--(1,0.5);
\draw (0,-0.5)--(1,-0.5);
\draw [line width=0.5pt,line cap=round,rounded corners] (2.north west)  rectangle (2.south east);
\draw [line width=0.5pt,line cap=round,rounded corners] (4.north west)  rectangle (4.south east);
\draw [line width=0.5pt,line cap=round,rounded corners] (3.north west)  rectangle (5.south east);
\draw [line width=0.5pt,line cap=round,rounded corners] (1.north west)  rectangle (6.south east);
\end{tikzpicture}$\\
$\sE_{6;2}=\begin{tikzpicture}[scale=0.5,baseline=-0.5ex]
\node at (0,0.8) {};
\node at (0,-0.8) {};
\node [inner sep=0.8pt,outer sep=0.8pt] at (-2,0) (1) {$\bullet$};
\node [inner sep=0.8pt,outer sep=0.8pt] at (-1,0) (3) {$\bullet$};
\node [inner sep=0.8pt,outer sep=0.8pt] at (0,0) (4) {$\bullet$};
\node [inner sep=0.8pt,outer sep=0.8pt] at (1,0) (5) {$\bullet$};
\node [inner sep=0.8pt,outer sep=0.8pt] at (2,0) (6) {$\bullet$};
\node [inner sep=0.8pt,outer sep=0.8pt] at (0,-1) (2) {$\bullet$};
\draw (-2,0)--(2,0);
\draw (0,0)--(0,-1);
\draw [line width=0.5pt,line cap=round,rounded corners] (1.north west)  rectangle (1.south east);
\draw [line width=0.5pt,line cap=round,rounded corners] (6.north west)  rectangle (6.south east);
\end{tikzpicture}$\qquad 
$\sE_{6;6}=\begin{tikzpicture}[scale=0.5,baseline=-0.5ex]
\node at (0,0.8) {};
\node at (0,-0.8) {};
\node [inner sep=0.8pt,outer sep=0.8pt] at (-2,0) (1) {$\bullet$};
\node [inner sep=0.8pt,outer sep=0.8pt] at (-1,0) (3) {$\bullet$};
\node [inner sep=0.8pt,outer sep=0.8pt] at (0,0) (4) {$\bullet$};
\node [inner sep=0.8pt,outer sep=0.8pt] at (1,0) (5) {$\bullet$};
\node [inner sep=0.8pt,outer sep=0.8pt] at (2,0) (6) {$\bullet$};
\node [inner sep=0.8pt,outer sep=0.8pt] at (0,-1) (2) {$\bullet$};
\draw (-2,0)--(2,0);
\draw (0,0)--(0,-1);
\draw [line width=0.5pt,line cap=round,rounded corners] (1.north west)  rectangle (1.south east);
\draw [line width=0.5pt,line cap=round,rounded corners] (2.north west)  rectangle (2.south east);
\draw [line width=0.5pt,line cap=round,rounded corners] (3.north west)  rectangle (3.south east);
\draw [line width=0.5pt,line cap=round,rounded corners] (4.north west)  rectangle (4.south east);
\draw [line width=0.5pt,line cap=round,rounded corners] (5.north west)  rectangle (5.south east);
\draw [line width=0.5pt,line cap=round,rounded corners] (6.north west)  rectangle (6.south east);
\end{tikzpicture}$\\
$\sE_{7;0}=\begin{tikzpicture}[scale=0.5,baseline=-0.5ex]
\node at (0,0.8) {};
\node at (0,-0.8) {};
\node [inner sep=0.8pt,outer sep=0.8pt] at (-2,0) (1) {$\bullet$};
\node [inner sep=0.8pt,outer sep=0.8pt] at (-1,0) (3) {$\bullet$};
\node [inner sep=0.8pt,outer sep=0.8pt] at (0,0) (4) {$\bullet$};
\node [inner sep=0.8pt,outer sep=0.8pt] at (1,0) (5) {$\bullet$};
\node [inner sep=0.8pt,outer sep=0.8pt] at (2,0) (6) {$\bullet$};
\node [inner sep=0.8pt,outer sep=0.8pt] at (3,0) (7) {$\bullet$};
\node [inner sep=0.8pt,outer sep=0.8pt] at (0,-1) (2) {$\bullet$};
\draw (-2,0)--(3,0);
\draw (0,0)--(0,-1);
\phantom{\draw [line width=0.5pt,line cap=round,rounded corners] (1.north west)  rectangle (1.south east);}
\phantom{\draw [line width=0.5pt,line cap=round,rounded corners] (7.north west)  rectangle (7.south east);}
\end{tikzpicture}$\qquad 
$\sE_{7;1}=\begin{tikzpicture}[scale=0.5,baseline=-0.5ex]
\node at (0,0.8) {};
\node at (0,-0.8) {};
\node [inner sep=0.8pt,outer sep=0.8pt] at (-2,0) (1) {$\bullet$};
\node [inner sep=0.8pt,outer sep=0.8pt] at (-1,0) (3) {$\bullet$};
\node [inner sep=0.8pt,outer sep=0.8pt] at (0,0) (4) {$\bullet$};
\node [inner sep=0.8pt,outer sep=0.8pt] at (1,0) (5) {$\bullet$};
\node [inner sep=0.8pt,outer sep=0.8pt] at (2,0) (6) {$\bullet$};
\node [inner sep=0.8pt,outer sep=0.8pt] at (3,0) (7) {$\bullet$};
\node [inner sep=0.8pt,outer sep=0.8pt] at (0,-1) (2) {$\bullet$};
\draw (-2,0)--(3,0);
\draw (0,0)--(0,-1);
\draw [line width=0.5pt,line cap=round,rounded corners] (1.north west)  rectangle (1.south east);
\end{tikzpicture}$\qquad
$\sE_{7;2}=\begin{tikzpicture}[scale=0.5,baseline=-0.5ex]
\node at (0,0.8) {};
\node at (0,-0.8) {};
\node [inner sep=0.8pt,outer sep=0.8pt] at (-2,0) (1) {$\bullet$};
\node [inner sep=0.8pt,outer sep=0.8pt] at (-1,0) (3) {$\bullet$};
\node [inner sep=0.8pt,outer sep=0.8pt] at (0,0) (4) {$\bullet$};
\node [inner sep=0.8pt,outer sep=0.8pt] at (1,0) (5) {$\bullet$};
\node [inner sep=0.8pt,outer sep=0.8pt] at (2,0) (6) {$\bullet$};
\node [inner sep=0.8pt,outer sep=0.8pt] at (3,0) (7) {$\bullet$};
\node [inner sep=0.8pt,outer sep=0.8pt] at (0,-1) (2) {$\bullet$};
\draw (-2,0)--(3,0);
\draw (0,0)--(0,-1);
\draw [line width=0.5pt,line cap=round,rounded corners] (1.north west)  rectangle (1.south east);
\draw [line width=0.5pt,line cap=round,rounded corners] (6.north west)  rectangle (6.south east);
\end{tikzpicture}$\\
$\sE_{7;3}=\begin{tikzpicture}[scale=0.5,baseline=-0.5ex]
\node at (0,0.8) {};
\node at (0,-0.8) {};
\node [inner sep=0.8pt,outer sep=0.8pt] at (-2,0) (1) {$\bullet$};
\node [inner sep=0.8pt,outer sep=0.8pt] at (-1,0) (3) {$\bullet$};
\node [inner sep=0.8pt,outer sep=0.8pt] at (0,0) (4) {$\bullet$};
\node [inner sep=0.8pt,outer sep=0.8pt] at (1,0) (5) {$\bullet$};
\node [inner sep=0.8pt,outer sep=0.8pt] at (2,0) (6) {$\bullet$};
\node [inner sep=0.8pt,outer sep=0.8pt] at (3,0) (7) {$\bullet$};
\node [inner sep=0.8pt,outer sep=0.8pt] at (0,-1) (2) {$\bullet$};
\draw (-2,0)--(3,0);
\draw (0,0)--(0,-1);
\draw [line width=0.5pt,line cap=round,rounded corners] (1.north west)  rectangle (1.south east);
\draw [line width=0.5pt,line cap=round,rounded corners] (6.north west)  rectangle (6.south east);
\draw [line width=0.5pt,line cap=round,rounded corners] (7.north west)  rectangle (7.south east);
\end{tikzpicture}$\qquad 
$\sE_{7;4}=\begin{tikzpicture}[scale=0.5,baseline=-0.5ex]
\node at (0,0.8) {};
\node at (0,-0.8) {};
\node [inner sep=0.8pt,outer sep=0.8pt] at (-2,0) (1) {$\bullet$};
\node [inner sep=0.8pt,outer sep=0.8pt] at (-1,0) (3) {$\bullet$};
\node [inner sep=0.8pt,outer sep=0.8pt] at (0,0) (4) {$\bullet$};
\node [inner sep=0.8pt,outer sep=0.8pt] at (1,0) (5) {$\bullet$};
\node [inner sep=0.8pt,outer sep=0.8pt] at (2,0) (6) {$\bullet$};
\node [inner sep=0.8pt,outer sep=0.8pt] at (3,0) (7) {$\bullet$};
\node [inner sep=0.8pt,outer sep=0.8pt] at (0,-1) (2) {$\bullet$};
\draw (-2,0)--(3,0);
\draw (0,0)--(0,-1);
\draw [line width=0.5pt,line cap=round,rounded corners] (1.north west)  rectangle (1.south east);
\draw [line width=0.5pt,line cap=round,rounded corners] (3.north west)  rectangle (3.south east);
\draw [line width=0.5pt,line cap=round,rounded corners] (4.north west)  rectangle (4.south east);
\draw [line width=0.5pt,line cap=round,rounded corners] (6.north west)  rectangle (6.south east);
\end{tikzpicture}$\qquad
$\sE_{7;7}=\begin{tikzpicture}[scale=0.5,baseline=-0.5ex]
\node at (0,0.8) {};
\node at (0,-0.8) {};
\node [inner sep=0.8pt,outer sep=0.8pt] at (-2,0) (1) {$\bullet$};
\node [inner sep=0.8pt,outer sep=0.8pt] at (-1,0) (3) {$\bullet$};
\node [inner sep=0.8pt,outer sep=0.8pt] at (0,0) (4) {$\bullet$};
\node [inner sep=0.8pt,outer sep=0.8pt] at (1,0) (5) {$\bullet$};
\node [inner sep=0.8pt,outer sep=0.8pt] at (2,0) (6) {$\bullet$};
\node [inner sep=0.8pt,outer sep=0.8pt] at (3,0) (7) {$\bullet$};
\node [inner sep=0.8pt,outer sep=0.8pt] at (0,-1) (2) {$\bullet$};
\draw (-2,0)--(3,0);
\draw (0,0)--(0,-1);
\draw [line width=0.5pt,line cap=round,rounded corners] (1.north west)  rectangle (1.south east);
\draw [line width=0.5pt,line cap=round,rounded corners] (3.north west)  rectangle (3.south east);
\draw [line width=0.5pt,line cap=round,rounded corners] (4.north west)  rectangle (4.south east);
\draw [line width=0.5pt,line cap=round,rounded corners] (6.north west)  rectangle (6.south east);
\draw [line width=0.5pt,line cap=round,rounded corners] (2.north west)  rectangle (2.south east);
\draw [line width=0.5pt,line cap=round,rounded corners] (5.north west)  rectangle (5.south east);
\draw [line width=0.5pt,line cap=round,rounded corners] (7.north west)  rectangle (7.south east);
\end{tikzpicture}$\\
$\sE_{8;0}=\begin{tikzpicture}[scale=0.5,baseline=-0.5ex]
\node at (0,0.8) {};
\node at (0,-0.8) {};
\node [inner sep=0.8pt,outer sep=0.8pt] at (-2,0) (1) {$\bullet$};
\node [inner sep=0.8pt,outer sep=0.8pt] at (-1,0) (3) {$\bullet$};
\node [inner sep=0.8pt,outer sep=0.8pt] at (0,0) (4) {$\bullet$};
\node [inner sep=0.8pt,outer sep=0.8pt] at (1,0) (5) {$\bullet$};
\node [inner sep=0.8pt,outer sep=0.8pt] at (2,0) (6) {$\bullet$};
\node [inner sep=0.8pt,outer sep=0.8pt] at (3,0) (7) {$\bullet$};
\node [inner sep=0.8pt,outer sep=0.8pt] at (4,0) (8) {$\bullet$};
\node [inner sep=0.8pt,outer sep=0.8pt] at (0,-1) (2) {$\bullet$};
\draw (-2,0)--(4,0);
\draw (0,0)--(0,-1);
\phantom{\draw [line width=0.5pt,line cap=round,rounded corners] (1.north west)  rectangle (1.south east);}
\phantom{\draw [line width=0.5pt,line cap=round,rounded corners] (8.north west)  rectangle (8.south east);}
\end{tikzpicture}$\qquad
$\sE_{8;1}=\begin{tikzpicture}[scale=0.5,baseline=-0.5ex]
\node at (0,0.8) {};
\node at (0,-0.8) {};
\node [inner sep=0.8pt,outer sep=0.8pt] at (-2,0) (1) {$\bullet$};
\node [inner sep=0.8pt,outer sep=0.8pt] at (-1,0) (3) {$\bullet$};
\node [inner sep=0.8pt,outer sep=0.8pt] at (0,0) (4) {$\bullet$};
\node [inner sep=0.8pt,outer sep=0.8pt] at (1,0) (5) {$\bullet$};
\node [inner sep=0.8pt,outer sep=0.8pt] at (2,0) (6) {$\bullet$};
\node [inner sep=0.8pt,outer sep=0.8pt] at (3,0) (7) {$\bullet$};
\node [inner sep=0.8pt,outer sep=0.8pt] at (4,0) (8) {$\bullet$};
\node [inner sep=0.8pt,outer sep=0.8pt] at (0,-1) (2) {$\bullet$};
\draw (-2,0)--(4,0);
\draw (0,0)--(0,-1);
\draw [line width=0.5pt,line cap=round,rounded corners] (8.north west)  rectangle (8.south east);
\phantom{\draw [line width=0.5pt,line cap=round,rounded corners] (1.north west)  rectangle (1.south east);}
\end{tikzpicture}$\qquad 
$\sE_{8;2}=\begin{tikzpicture}[scale=0.5,baseline=-0.5ex]
\node at (0,0.8) {};
\node at (0,-0.8) {};
\node [inner sep=0.8pt,outer sep=0.8pt] at (-2,0) (1) {$\bullet$};
\node [inner sep=0.8pt,outer sep=0.8pt] at (-1,0) (3) {$\bullet$};
\node [inner sep=0.8pt,outer sep=0.8pt] at (0,0) (4) {$\bullet$};
\node [inner sep=0.8pt,outer sep=0.8pt] at (1,0) (5) {$\bullet$};
\node [inner sep=0.8pt,outer sep=0.8pt] at (2,0) (6) {$\bullet$};
\node [inner sep=0.8pt,outer sep=0.8pt] at (3,0) (7) {$\bullet$};
\node [inner sep=0.8pt,outer sep=0.8pt] at (4,0) (8) {$\bullet$};
\node [inner sep=0.8pt,outer sep=0.8pt] at (0,-1) (2) {$\bullet$};
\draw (-2,0)--(4,0);
\draw (0,0)--(0,-1);
\draw [line width=0.5pt,line cap=round,rounded corners] (1.north west)  rectangle (1.south east);
\draw [line width=0.5pt,line cap=round,rounded corners] (8.north west)  rectangle (8.south east);
\end{tikzpicture}$\\
$\sE_{8;4}=\begin{tikzpicture}[scale=0.5,baseline=-0.5ex]
\node at (0,0.8) {};
\node at (0,-0.8) {};
\node [inner sep=0.8pt,outer sep=0.8pt] at (-2,0) (1) {$\bullet$};
\node [inner sep=0.8pt,outer sep=0.8pt] at (-1,0) (3) {$\bullet$};
\node [inner sep=0.8pt,outer sep=0.8pt] at (0,0) (4) {$\bullet$};
\node [inner sep=0.8pt,outer sep=0.8pt] at (1,0) (5) {$\bullet$};
\node [inner sep=0.8pt,outer sep=0.8pt] at (2,0) (6) {$\bullet$};
\node [inner sep=0.8pt,outer sep=0.8pt] at (3,0) (7) {$\bullet$};
\node [inner sep=0.8pt,outer sep=0.8pt] at (4,0) (8) {$\bullet$};
\node [inner sep=0.8pt,outer sep=0.8pt] at (0,-1) (2) {$\bullet$};
\draw (-2,0)--(4,0);
\draw (0,0)--(0,-1);
\draw [line width=0.5pt,line cap=round,rounded corners] (1.north west)  rectangle (1.south east);
\draw [line width=0.5pt,line cap=round,rounded corners] (6.north west)  rectangle (6.south east);
\draw [line width=0.5pt,line cap=round,rounded corners] (7.north west)  rectangle (7.south east);
\draw [line width=0.5pt,line cap=round,rounded corners] (8.north west)  rectangle (8.south east);
\end{tikzpicture}$\qquad 
$\sE_{8;8}=\begin{tikzpicture}[scale=0.5,baseline=-0.5ex]
\node at (0,0.8) {};
\node at (0,-0.8) {};
\node [inner sep=0.8pt,outer sep=0.8pt] at (-2,0) (1) {$\bullet$};
\node [inner sep=0.8pt,outer sep=0.8pt] at (-1,0) (3) {$\bullet$};
\node [inner sep=0.8pt,outer sep=0.8pt] at (0,0) (4) {$\bullet$};
\node [inner sep=0.8pt,outer sep=0.8pt] at (1,0) (5) {$\bullet$};
\node [inner sep=0.8pt,outer sep=0.8pt] at (2,0) (6) {$\bullet$};
\node [inner sep=0.8pt,outer sep=0.8pt] at (3,0) (7) {$\bullet$};
\node [inner sep=0.8pt,outer sep=0.8pt] at (4,0) (8) {$\bullet$};
\node [inner sep=0.8pt,outer sep=0.8pt] at (0,-1) (2) {$\bullet$};
\draw (-2,0)--(4,0);
\draw (0,0)--(0,-1);
\draw [line width=0.5pt,line cap=round,rounded corners] (8.north west)  rectangle (8.south east);
\draw [line width=0.5pt,line cap=round,rounded corners] (7.north west)  rectangle (7.south east);
\draw [line width=0.5pt,line cap=round,rounded corners] (6.north west)  rectangle (6.south east);
\draw [line width=0.5pt,line cap=round,rounded corners] (5.north west)  rectangle (5.south east);
\draw [line width=0.5pt,line cap=round,rounded corners] (4.north west)  rectangle (4.south east);
\draw [line width=0.5pt,line cap=round,rounded corners] (3.north west)  rectangle (3.south east);
\draw [line width=0.5pt,line cap=round,rounded corners] (2.north west)  rectangle (2.south east);
\draw [line width=0.5pt,line cap=round,rounded corners] (1.north west)  rectangle (1.south east);
\end{tikzpicture}$\\
$\sF_{4;0}=\begin{tikzpicture}[scale=0.5,baseline=-0.5ex]
\node at (0,0.8) {};
\node at (0,-0.8) {};
\node [inner sep=0.8pt,outer sep=0.8pt] at (-1.5,0) (1) {$\bullet$};
\node [inner sep=0.8pt,outer sep=0.8pt] at (-0.5,0) (2) {$\bullet$};
\node [inner sep=0.8pt,outer sep=0.8pt] at (0.5,0) (3) {$\bullet$};
\node [inner sep=0.8pt,outer sep=0.8pt] at (1.5,0) (4) {$\bullet$};
\phantom{\draw [line width=0.5pt,line cap=round,rounded corners] (1.north west)  rectangle (1.south east);}
\phantom{\draw [line width=0.5pt,line cap=round,rounded corners] (4.north west)  rectangle (4.south east);}
\draw (-1.5,0)--(-0.5,0);
\draw (0.5,0)--(1.5,0);
\draw (-0.5,0.1)--(0.5,0.1);
\draw (-0.5,-0.1)--(0.5,-0.1);
\draw (0-0.15,0.3) -- (0+0.08,0) -- (0-0.15,-0.3);
\end{tikzpicture}$\qquad
$\sF_{4;1}^{1}=\begin{tikzpicture}[scale=0.5,baseline=-0.5ex]
\node at (0,0.8) {};
\node at (0,-0.8) {};
\node [inner sep=0.8pt,outer sep=0.8pt] at (-1.5,0) (1) {$\bullet$};
\node [inner sep=0.8pt,outer sep=0.8pt] at (-0.5,0) (2) {$\bullet$};
\node [inner sep=0.8pt,outer sep=0.8pt] at (0.5,0) (3) {$\bullet$};
\node [inner sep=0.8pt,outer sep=0.8pt] at (1.5,0) (4) {$\bullet$};
\phantom{\draw [line width=0.5pt,line cap=round,rounded corners] (1.north west)  rectangle (1.south east);}
\phantom{\draw [line width=0.5pt,line cap=round,rounded corners] (4.north west)  rectangle (4.south east);}
\draw (-1.5,0)--(-0.5,0);
\draw (0.5,0)--(1.5,0);
\draw (-0.5,0.1)--(0.5,0.1);
\draw (-0.5,-0.1)--(0.5,-0.1);
\draw (0-0.15,0.3) -- (0+0.08,0) -- (0-0.15,-0.3);
\draw [line width=0.5pt,line cap=round,rounded corners] (1.north west)  rectangle (1.south east);
\end{tikzpicture}$\qquad
$\sF_{4;1}^{4}=\begin{tikzpicture}[scale=0.5,baseline=-0.5ex]
\node at (0,0.8) {};
\node at (0,-0.8) {};
\node [inner sep=0.8pt,outer sep=0.8pt] at (-1.5,0) (1) {$\bullet$};
\node [inner sep=0.8pt,outer sep=0.8pt] at (-0.5,0) (2) {$\bullet$};
\node [inner sep=0.8pt,outer sep=0.8pt] at (0.5,0) (3) {$\bullet$};
\node [inner sep=0.8pt,outer sep=0.8pt] at (1.5,0) (4) {$\bullet$};
\phantom{\draw [line width=0.5pt,line cap=round,rounded corners] (1.north west)  rectangle (1.south east);}
\phantom{\draw [line width=0.5pt,line cap=round,rounded corners] (4.north west)  rectangle (4.south east);}
\draw (-1.5,0)--(-0.5,0);
\draw (0.5,0)--(1.5,0);
\draw (-0.5,0.1)--(0.5,0.1);
\draw (-0.5,-0.1)--(0.5,-0.1);
\draw (0-0.15,0.3) -- (0+0.08,0) -- (0-0.15,-0.3);
\draw [line width=0.5pt,line cap=round,rounded corners] (4.north west)  rectangle (4.south east);
\end{tikzpicture}$\\
$\sF_{4;2}=\begin{tikzpicture}[scale=0.5,baseline=-0.5ex]
\node at (0,0.8) {};
\node at (0,-0.8) {};
\node [inner sep=0.8pt,outer sep=0.8pt] at (-1.5,0) (1) {$\bullet$};
\node [inner sep=0.8pt,outer sep=0.8pt] at (-0.5,0) (2) {$\bullet$};
\node [inner sep=0.8pt,outer sep=0.8pt] at (0.5,0) (3) {$\bullet$};
\node [inner sep=0.8pt,outer sep=0.8pt] at (1.5,0) (4) {$\bullet$};
\phantom{\draw [line width=0.5pt,line cap=round,rounded corners] (1.north west)  rectangle (1.south east);}
\phantom{\draw [line width=0.5pt,line cap=round,rounded corners] (4.north west)  rectangle (4.south east);}
\draw (-1.5,0)--(-0.5,0);
\draw (0.5,0)--(1.5,0);
\draw (-0.5,0.1)--(0.5,0.1);
\draw (-0.5,-0.1)--(0.5,-0.1);
\draw (0-0.15,0.3) -- (0+0.08,0) -- (0-0.15,-0.3);
\draw [line width=0.5pt,line cap=round,rounded corners] (1.north west)  rectangle (1.south east);
\draw [line width=0.5pt,line cap=round,rounded corners] (4.north west)  rectangle (4.south east);
\end{tikzpicture}$\qquad
$\sF_{4;4}=\begin{tikzpicture}[scale=0.5,baseline=-0.5ex]
\node at (0,0.8) {};
\node at (0,-0.8) {};
\node [inner sep=0.8pt,outer sep=0.8pt] at (-1.5,0) (1) {$\bullet$};
\node [inner sep=0.8pt,outer sep=0.8pt] at (-0.5,0) (2) {$\bullet$};
\node [inner sep=0.8pt,outer sep=0.8pt] at (0.5,0) (3) {$\bullet$};
\node [inner sep=0.8pt,outer sep=0.8pt] at (1.5,0) (4) {$\bullet$};
\phantom{\draw [line width=0.5pt,line cap=round,rounded corners] (1.north west)  rectangle (1.south east);}
\phantom{\draw [line width=0.5pt,line cap=round,rounded corners] (4.north west)  rectangle (4.south east);}
\draw (-1.5,0)--(-0.5,0);
\draw (0.5,0)--(1.5,0);
\draw (-0.5,0.1)--(0.5,0.1);
\draw (-0.5,-0.1)--(0.5,-0.1);
\draw (0-0.15,0.3) -- (0+0.08,0) -- (0-0.15,-0.3);
\draw [line width=0.5pt,line cap=round,rounded corners] (1.north west)  rectangle (1.south east);
\draw [line width=0.5pt,line cap=round,rounded corners] (2.north west)  rectangle (2.south east);
\draw [line width=0.5pt,line cap=round,rounded corners] (3.north west)  rectangle (3.south east);
\draw [line width=0.5pt,line cap=round,rounded corners] (4.north west)  rectangle (4.south east);
\end{tikzpicture}$\\
$\sG_{2;0}=\begin{tikzpicture}[scale=0.5,baseline=-0.5ex]
\node at (0,0.8) {};
\node at (0,-0.8) {};
\node [inner sep=0.8pt,outer sep=0.8pt] at (-0.5,0) (2) {$\bullet$};
\node [inner sep=0.8pt,outer sep=0.8pt] at (0.5,0) (3) {$\bullet$};
\phantom{\draw [line width=0.5pt,line cap=round,rounded corners] (2.north west)  rectangle (2.south east);}
\phantom{\draw [line width=0.5pt,line cap=round,rounded corners] (3.north west)  rectangle (3.south east);}
\draw (-0.5,0)--(0.5,0);
\draw (-0.5,0.11)--(0.5,0.11);
\draw (-0.5,-0.11)--(0.5,-0.11);
\draw (0+0.15,0.3) -- (0-0.08,0) -- (0+0.15,-0.3);
\end{tikzpicture}$\qquad
$\sG_{2;1}^{1}=\begin{tikzpicture}[scale=0.5,baseline=-0.5ex]
\node at (0,0.8) {};
\node at (0,-0.8) {};
\node [inner sep=0.8pt,outer sep=0.8pt] at (-0.5,0) (2) {$\bullet$};
\node [inner sep=0.8pt,outer sep=0.8pt] at (0.5,0) (3) {$\bullet$};
\phantom{\draw [line width=0.5pt,line cap=round,rounded corners] (2.north west)  rectangle (2.south east);}
\phantom{\draw [line width=0.5pt,line cap=round,rounded corners] (3.north west)  rectangle (3.south east);}
\draw (-0.5,0)--(0.5,0);
\draw (-0.5,0.11)--(0.5,0.11);
\draw (-0.5,-0.11)--(0.5,-0.11);
\draw (0+0.15,0.3) -- (0-0.08,0) -- (0+0.15,-0.3);
\draw [line width=0.5pt,line cap=round,rounded corners] (2.north west)  rectangle (2.south east);
\end{tikzpicture}$\qquad
$\sG_{2;1}^{2}=\begin{tikzpicture}[scale=0.5,baseline=-0.5ex]
\node at (0,0.8) {};
\node at (0,-0.8) {};
\node [inner sep=0.8pt,outer sep=0.8pt] at (-0.5,0) (2) {$\bullet$};
\node [inner sep=0.8pt,outer sep=0.8pt] at (0.5,0) (3) {$\bullet$};
\phantom{\draw [line width=0.5pt,line cap=round,rounded corners] (2.north west)  rectangle (2.south east);}
\phantom{\draw [line width=0.5pt,line cap=round,rounded corners] (3.north west)  rectangle (3.south east);}
\draw (-0.5,0)--(0.5,0);
\draw (-0.5,0.11)--(0.5,0.11);
\draw (-0.5,-0.11)--(0.5,-0.11);
\draw (0+0.15,0.3) -- (0-0.08,0) -- (0+0.15,-0.3);
\draw [line width=0.5pt,line cap=round,rounded corners] (3.north west)  rectangle (3.south east);
\end{tikzpicture}$\qquad
$\sG_{2;2}=\begin{tikzpicture}[scale=0.5,baseline=-0.5ex]
\node at (0,0.8) {};
\node at (0,-0.8) {};
\node [inner sep=0.8pt,outer sep=0.8pt] at (-0.5,0) (2) {$\bullet$};
\node [inner sep=0.8pt,outer sep=0.8pt] at (0.5,0) (3) {$\bullet$};
\phantom{\draw [line width=0.5pt,line cap=round,rounded corners] (2.north west)  rectangle (2.south east);}
\phantom{\draw [line width=0.5pt,line cap=round,rounded corners] (3.north west)  rectangle (3.south east);}
\draw (-0.5,0)--(0.5,0);
\draw (-0.5,0.11)--(0.5,0.11);
\draw (-0.5,-0.11)--(0.5,-0.11);
\draw (0+0.15,0.3) -- (0-0.08,0) -- (0+0.15,-0.3);
\draw [line width=0.5pt,line cap=round,rounded corners] (2.north west)  rectangle (2.south east);
\draw [line width=0.5pt,line cap=round,rounded corners] (3.north west)  rectangle (3.south east);
\end{tikzpicture}$
\end{center}
\caption{The admissible Dynkin diagrams of exceptional type}\label{fig:Dynkin}
\end{figure}
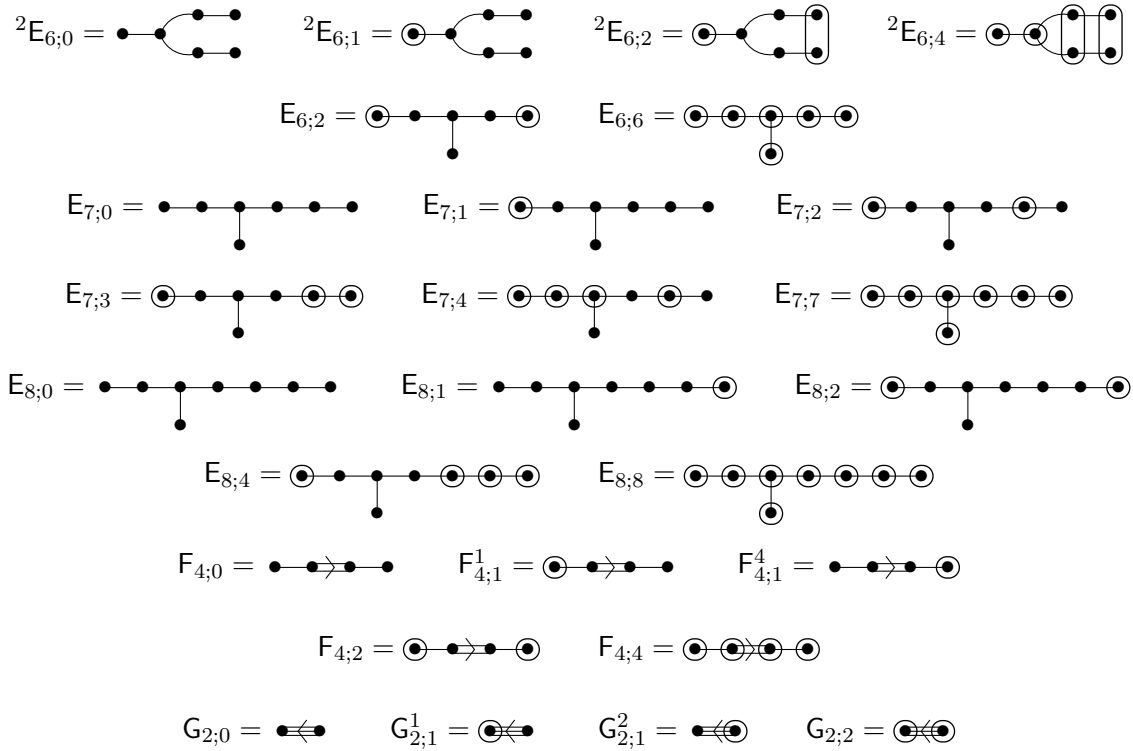

To summarise, if $\theta$ is an automorphism of a split spherical building of exceptional type, then the opposition diagram of $\theta$ is one of the diagrams listed in Figure~\ref{fig:Dynkin}. Uncapped automorphisms are studied in~\cite{PVM:19b}, and so for the remainder of this introduction we consider capped automorphisms (for example, if $\Delta$ is large then all automorphisms are capped). In this case, automorphisms with ``full'' opposition diagrams (in which all nodes are encircled) are necessarily not domestic, and hence are not discussed further here. Moreover, the ``empty'' diagrams (with no nodes encircled) correspond precisely to the trivial automorphisms, and hence are also of little interest here. Furthermore, automorphisms with diagram $\sE_{6;2}$ have been completely classified in \cite{HVM:13} (for large buildings) and \cite{PVM:19b} (for small buildings). This leaves us with $14$ remaining diagrams. 

The following terms will be defined more systematically in later sections, however for the purpose of this introduction, and in order to state our main results, we define:
\begin{compactenum}[$(1)$]
\item The \textit{polar diagrams} to be the diagrams ${^2}\sE_{6;1},\sE_{7;1},\sE_{8;1},\sF_{4;1}^1,\sG_{2;1}^2$;
\item The \textit{polar-copolar diagrams} to be the diagrams ${^2}\sE_{6;2}$, $\sE_{7;2}$, $\sE_{8;2}$, $\sF_{4;2}$;
\item The \textit{polar closed diagrams} to be all diagrams except $\sE_{6;2},\sE_{6;6},\sF_{4;1}^4,\sG_{2;1}^1$.
\end{compactenum}
We consider split spherical buildings arising from a Chevalley group $G=G_{\Phi}(\FF)$ associated to a crystallographic root system $\Phi$, with~$\FF$ a field. For the purpose of this introduction, unless stated explicitly otherwise we will assume that the characteristic of $\FF$ is not ``special'' (meaning $\mathrm{char}(\FF)\neq 2$ for $\sF_4$, and $\mathrm{char}(\FF)\neq 3$ for $\sG_2$). By a \textit{root elation} we mean an element, in Chevalley generators, conjugate to $x_{\alpha}(a)$ for some root $\alpha\in\Phi$ and some $a\neq 0$. In the non-simply laced case we talk of \textit{long} and \textit{short} root elations, and in the simply laced case all roots are considered long. Root elations $x_{\alpha}(a)$ and $x_{\beta}(b)$ are \textit{perpendicular} if $\alpha$ and $\beta$ are perpendicular roots. A \textit{positive root elation} is an element $x_{\alpha}(a)$ with $\alpha\in \Phi^+$, where $\Phi^+$ is a fixed choice of positive roots of~$\Phi$. By a \textit{homology} we mean an element conjugate to an element of the torus~$H$.

With these definitions and conventions, a summary of the main results of this paper is as follows. 
 We first give a complete classification of automorphisms with polar opposition diagram. 

\begin{thm1}\label{thm:polarclass}
An automorphism of a split spherical building of exceptional type has polar opposition diagram if and only if it is a long root elation.
\end{thm1}

In fact, some aspects of our analysis of long root elations applies to all Moufang spherical buildings, and leads to various corollaries, including the following.

\begin{cor1}\label{cor:existsdomestic}
Every irreducible Moufang spherical building, other than a projective plane, admits a nontrivial domestic collineation.
\end{cor1}

\begin{cor1}\label{cor:existsconjugacy}
Let $G$ be the collineation group of a Moufang spherical building $\Delta$ of type other than~$\sA_n$. There exists a nontrivial conjugacy class $\mathscr{C}$ in $G$ which is not transitive on the set of vertices of $\Delta$ of type~$s$, for any $s\in S$.
\end{cor1}

In the case of Ree-Tits octagons, Corollary~\ref{cor:existsdomestic} corrects a misunderstanding from \cite{ISX:18} (see Remark~\ref{rem:correction}), and Corollary~\ref{cor:existsconjugacy} answers a question asked to us by Barbara Baumeister. 

The classification of automorphisms with polar opposition diagram can be extended to the polar-copolar diagrams in certain cases. We prove:

\begin{thm1}\label{thm:pocopolarclass}
A collineation of a large split spherical building of type $\sE_7$, $\sE_8$, or $\sF_4$ has polar-copolar opposition diagram if and only if it is a product of two perpendicular long root elations. Moreover, for the $\sE_7$ and $\sE_8$ cases the collineations with polar-copolar diagram form a single conjugacy class. 
\end{thm1}

See Theorem~\ref{thm:F4class} below for more details on the conjugacy classes in type $\sF_4$. Also we note that the ``if'' part of Theorem~\ref{thm:pocopolarclass} holds for the $\sE_{6}$ case too, however the ``only if'' part fails, as there are also exist homologies with diagram~${^2}\sE_{6;2}$ (see Theorems~\ref{thm:homclass} and~\ref{thm:E6class} below). 

Our next main theorem classifies the opposition diagrams of elements of the unipotent subgroup~$U^+$ generated by the positive root elations. 

\begin{thm1}\label{thm:unipotent} Let $\Delta$ be a split spherical building with root system~$\Phi$ of exceptional type. An admissible Dynkin diagram~$\sX$ of type~$\Phi$ can be obtained as the opposition diagram of a product of positive root elations of $\Delta$ if and only if~$\sX$ is polar closed. 
\end{thm1}

Moreover we provide an algorithm to write down, for each polar closed diagram~$\sX$, an element  $\theta\in U^+$ with opposition diagram~$\sX$. In fact it turns out that every polar closed diagram can be obtained as the opposition diagram a product of mutually perpendicular positive root elations.

Next we give a complete classification of domestic homologies for split exceptional buildings. This classification is in terms of the type of the thick frame of the fixed subbuilding of the automorphism (c.f. \cite{Sch:87}). We summarise the statement below (see Section~\ref{sec:homologies} for explicit conjugacy class representatives for each case). 

\begin{thm1}\label{thm:homclass}
Let $\theta$ be a nontrivial homology of a split spherical building of exceptional type~$\Phi$, and let $\Phi'$ be the type of the thick frame of the subbuilding fixed by~$\theta$. Then $\theta$ is domestic if and only if
\begin{compactenum}[$(1)$]
\item $\Phi=\sE_6$ and $\Phi'=\sD_5$, in which case $\Diag(\theta)={^2}\sE_{6;2}$;
\item $\Phi=\sE_7$ and $\Phi'=\sE_6, \sD_6,\sD_6\times \sA_1$, in which case $\Diag(\theta)=\sE_{7;3},\sE_{7;4},\sE_{7;4}$ (respectively);
\item $\Phi=\sE_8$ and $\Phi'=\sE_7,\sE_7\times\sA_1$, in which case $\Diag(\theta)=\sE_{8;4}$;
\item $\Phi=\sF_4$ and $\Phi'=\sB_4$, in which case $\Diag(\theta)=\sF_{4;1}^4$;
\item $\Phi=\sG_2$ and $\Phi'=\sA_2$, in which case $\Diag(\theta)=\sG_{2;1}^1$.
\end{compactenum}
\end{thm1}

We also completely classify domestic automorphisms of split buildings of types~$\sE_6$, $\sF_4$, and $\sG_2$. For this introduction, let us state the result for $\sF_4$ over quadratically closed fields and finite fields, both of characteristic not~$2$ (see Subsection~\ref{sec:F4C} for statements applying to all fields). 

\begin{thm1}\label{thm:F4class}
Let $\Delta$ be the split spherical building of $G=\sF_4(\FF)$ with $\mathrm{char}(\FF)\neq 2$. If $\FF$ is quadratically closed (respectively finite) then there are precisely $3$ (respectively $4$) conjugacy classes of domestic collineations, consisting of
\begin{compactenum}[$(1)$]
\item the class of long root elations, with opposition diagram $\sF_{4;1}^1$;
\item the class of homologies fixing a subbuilding with thick frame of type~$\sB_4$, with opposition diagram~$\sF_{4;1}^4$;
\item one (respectively two) class(es) of products of two perpendicular long root elations, with opposition diagram $\sF_{4;2}$.
\end{compactenum}
\end{thm1}

For the $\sE_6$ case, domestic automorphisms of small buildings are already classified in \cite[Theorems~4.5 and~4.6]{PVM:19b}, and domestic dualities of large buildings are classified in~\cite{HVM:13}. Thus by Theorem~\ref{thm:polarclass} and the classification of admissible diagrams the remaining task is to classify the collineations of large $\sE_6$ buildings with diagram~${^2}\sE_{6;2}$. It turns out that the only examples are those described by Theorems~\ref{thm:unipotent} and~\ref{thm:homclass}, and thus for large buildings we have:

\begin{thm1}\label{thm:E6class}
Let $\Delta$ be a large building of type $\sE_6$. 
\begin{compactenum}[$(1)$]
\item A duality of $\Delta$ is domestic if and only if it is a symplectic polarity (that is, a duality fixing a split building of type $\sF_4$), in which case it has opposition diagram $\sE_{6;2}$. 
\item A collineation of $\Delta$ is domestic if and only if it is either
\begin{compactenum}[$(a)$]
\item a root elation, with opposition diagram ${^2}\sE_{6;1}$,
\item a product of two perpendicular root elations, with opposition diagram ${^2}\sE_{6;2}$, or
\item a homology fixing a subbuilding with thick frame of type~$\sD_5$, with opposition diagram~${^2}\sE_{6;2}$. 
\end{compactenum}
\end{compactenum}
\end{thm1}

We complete the analysis by classifying domestic automorphisms of split~$\sG_2$ buildings. Since no duality of a $\sG_2$ building is domestic \cite[Theorem~2.7]{PTM:15} it suffices to consider collineations. 

\begin{thm1}\label{thm:G2class}
Let $\Delta$ be the building of $\sG_2(\FF)$. There exists a unique conjugacy class $\mathscr{C}_1$ of collineations with opposition diagram $\sG_{2;1}^2$, and a unique conjugacy class $\mathscr{C}_2$ of collineations with opposition diagram $\sG_{2;1}^1$. The elements of $\mathscr{C}_1$ are long root elations, and the elements of $\mathscr{C}_2$
\begin{compactenum}[$(1)$]
\item are short root elations if $\mathrm{char}(\FF)=3$;
\item are homologies fixing a large full subhexagon if $\mathrm{char}(\FF)\neq 3$ and $z^2+z+1$ splits over~$\FF$;
\item fix a distance~$3$-ovoid if $z^2+z+1$ is irreducible over~$\FF$. 
\end{compactenum}
\end{thm1}

Consequently, the results of this paper (along with \cite{HVM:13} for the $\sE_{6;2}$ diagram) culminate in the classification of automorphisms of split spherical buildings of exceptional type having each non-full opposition diagram, with the exception of the $3$ diagrams $\sE_{7;3}$, $\sE_{7;4}$, and $\sE_{8;4}$. In these remaining cases we have provided examples of both unipotent elements and homologies with the given diagram (in Theorems~\ref{thm:unipotent} and~\ref{thm:homclass}). It turns out that for certain fields there also exist automorphisms with these opposition diagrams fixing no chamber of the building (hence these automorphisms are neither unipotent elements nor homologies). The description and classification of these automorphisms will be continued in future work~\cite{PVM:20c}.

We note that the results of this paper, combined with those of~\cite{PVM:20a} for the classical cases, show that every admissible Dynkin diagram can be obtained as the opposition diagram of an automorphism of a split spherical building. As discussed in~\cite{PVM:20a}, this statement is false for certain non-split buildings. More precisely, we have the following corollary.

\begin{cor1}\label{cor:existence}
Let $\Delta$ be a split spherical building of type $\Phi$. Every admissible Dynkin diagram of type $\Phi$ is the opposition diagram of some automorphism of $\Delta$.  Moreover, with only one exception, such an automorphism can be chosen such that it fixes a chamber of the building. This exception is the diagram $\sG_{2;1}^1$ in the case that the polynomial $z^2+z+1$ is irreducible over the underlying field~$\FF$. 
\end{cor1}

Finally, our results translate into group theoretic statements concerning conjugacy classes in Chevalley groups of exceptional type. To put these results into context, recall that by the \textit{Density Theorem} (see \cite[Section 22.2]{Hum:75}), if $G$ is a connected linear algebraic group over an algebraically closed field then the union of all conjugates of a Borel subgroup~$B$ is equal to~$G$. Equivalently, if $\mathscr{C}$ is a conjugacy class in $G$ then $\mathscr{C}\cap B\neq\emptyset$. This theorem is a cornerstone in the theory of algebraic groups, for example simple corollaries include the important facts that the centres of $G$ and $B$ coincide, and that the Cartan subgroups of $G$ are precisely the centralisers of maximal tori. 

The statement of the Density Theorem is clearly false in the general setting of a Chevalley group $G$ over an arbitrary field, as there typically exist elements $\theta\in G$ fixing no chamber of the building $\Delta=G/B$. However our classification theorems allow us to provide analogues in this setting, showing that every conjugacy class in $G$ intersects a union of a very small number of $B$-double cosets. For the purpose of this introduction we provide two examples; see Subsection~\ref{sec:conjugacyclasses} for further related statements. 

\begin{cor1}\label{cor:conj1}
Let $G$ be the Chevalley group of type $\sE_6$ or $\sF_4$ over a field $\FF$, and let $\mathscr{C}$ be a conjugacy class in $G$. Then $\mathscr{C}\cap (B\cup Bw_0B)\neq\emptyset$. 
\end{cor1}

The statement of Corollary~\ref{cor:conj1} fails for buildings of types $\sE_7$ and $\sE_8$ (see Remark~\ref{rem:nonchamberfix}).  Moreover, it is not true that $\mathscr{C}\cap Bw_0B\neq \emptyset$ for all nontrivial conjugacy classes. In fact, we have the following very general corollary of our results. 

\begin{cor1}\label{cor:conj2}
Let $G$ be the group of type preserving automorphisms of a Moufang spherical building not of type~$\sA_2$. There exists a nontrivial conjugacy class $\mathscr{C}$ with $\mathscr{C}\cap Bw_0B= \emptyset$.
\end{cor1}

Let us conclude this introduction with an outline of the structure of the paper. In Section~\ref{sec:background} we provide background on buildings, Chevalley groups, admissible diagrams, and prove some basic lemmas for later use. In Section~\ref{sec:polartype} we give the classification of automorphisms of split buildings with polar opposition diagram, proving Theorem~\ref{thm:polarclass} and Corollaries~\ref{cor:existsdomestic} and~\ref{cor:existsconjugacy}. Most of the arguments of this section are built around commutator relations in the Chevalley group, and we also discuss geometric characterisations of the polar diagram in the $\sE_{6,1}$ and $\sE_{7,7}$ Lie incidence geometries, and analyse short root elations in the non-simply laced case. 

In Section~\ref{sec:unipotent} we define polar closed diagrams, and present an algorithm for constructing unipotent elements with each polar closed diagram (proving Theorem~\ref{thm:unipotent}). Most of the arguments here are algebraic, however to complete the proof it is necessary to show that automorphisms with diagram $\sF_{4,1}^4$ are necessarily homologies (for $\mathrm{char}(\FF)\neq 2$), and we achieve this by arguing geometrically in the Lie incidence geometry $\sF_{4,4}(\FF)$. 

Section~\ref{sec:homologies} gives the complete classification of domestic homologies for split exceptional buildings (proving Theorem~\ref{thm:homclass}), making use of Scharlau's classification~\cite{Sch:87} of non-thick spherical buildings. In Section~\ref{sec:polarcopolar} we prove Theorem~\ref{thm:pocopolarclass} using geometric arguments  involving various Lie incidence geometries.

Finally, in Section~\ref{sec:classifications} we provide the complete classification of domestic collineations for split buildings of types $\sE_6$, $\sF_4$, and $\sG_2$, proving Theorems~\ref{thm:F4class}, \ref{thm:E6class}, and~\ref{thm:G2class}, and Corollaries~\ref{cor:existence}, \ref{cor:conj1}, and~\ref{cor:conj2}. We conclude with an appendix listing some relevant root system data for exceptional types. This data is useful at various stages of this paper, for example when performing commutator relations, or in Section~\ref{sec:homologies} when classifying domestic homologies.

\section{Background and definitions}\label{sec:background}

In this section we give a brief account of root systems, Chevalley groups and split spherical buildings, with our main references being~\cite{Bou:02,Car:89,St:16} (for root systems and Chevalley groups), and \cite{AB:08,Tit:74} (for buildings). We also recall the notions of admissible diagrams and opposition diagrams from \cite{PVM:19a,PVM:19b}, and record some basic lemmas for later use.

\subsection{Root systems and Chevalley groups}

Let $\Phi$ be a reduced irreducible crystallographic root system in an $n$-dimensional real vector space~$V$ with inner product $\langle\cdot,\cdot\rangle$, with $\alpha_1,\ldots,\alpha_n$ a choice of simple roots and $\Phi^+$ the associated positive roots. We will adopt the standard Bourbaki labelling~\cite{Bou:02} of the simple roots. Let $\alpha^{\vee}=2\alpha/\langle\alpha,\alpha\rangle$. Let $\omega_1,\ldots,\omega_n$ be the \textit{fundamental coweights}, defined by $\langle\omega_i,\alpha_j\rangle=\delta_{i,j}$. Let
$$
Q=\ZZ\alpha_1^{\vee}+\cdots+\ZZ\alpha_n^{\vee}\quad\text{and}\quad P=\ZZ\omega_1+\cdots+\ZZ\omega_n
$$
be the \textit{coroot lattice} and \textit{coweight lattice}, respectively, and note that $Q\subseteq P$. 

Let $\Gamma=\Gamma(\Phi)$ denote the Dynkin diagram of $\Phi$ (with the arrow pointing towards the short root in the case of double and triple bonds). The Coxeter diagram of $\Phi$ is obtained by removing all arrows from~$\Gamma$. The \textit{height} of a root $\alpha=k_1\alpha_1+\cdots+k_n\alpha_n$ is $\mathrm{ht}(\alpha)=k_1+\cdots+k_n$. There is a unique root $\varphi\in\Phi$ of maximal height (the \textit{highest root} of $\Phi$). The \textit{polar type} of $\Phi$ is the subset $\wp\subseteq\{1,2,\ldots,n\}$ given by
$$
\wp=\{1\leq i\leq n\mid \langle\alpha_i,\varphi\rangle\neq0\}.
$$
See Appendix~\ref{app:data} for the list of polar types. In particular, note that if $\Phi\neq \sA_n$ then $\wp=\{p\}$ is a singleton set, and in this case we often refer to the element $p$ as the \textit{polar node}.

Let $W=\langle s_{\alpha}\mid \alpha\in\Phi\rangle$ be the subgroup of $\mathsf{GL}(V)$ generated by the reflections $s_{\alpha}$, where
$$
s_{\alpha}(\lambda)=\lambda-\langle\lambda,\alpha\rangle\alpha^{\vee}\quad\text{for $\lambda\in V$}.
$$
Let $S=\{s_1,\ldots,s_n\}$, where $s_i=s_{\alpha_i}$. Then $(W,S)$ is a spherical Coxeter system. Writing $\ell: W\to\mathbb{Z}_{\geq 0}$ for the usual length function on~$W$, it is a well known fact that in the simply laced case, $\ell(s_{\alpha})=2\mathrm{ht}(\alpha)-1$.

Let $w_0$ denote the longest element of $(W,S)$, and let $\pi_0:\{1,\ldots,n\}\to\{1,\ldots,n\}$ be the \textit{opposition relation} given by $w_0\alpha_i=-\alpha_{\pi_0(i)}$ for $1\leq i\leq n$. We typically regard $\pi_0$ as an automorphism of the Dynkin diagram~$\Gamma$, and we say that ``opposition is type preserving'' if $\pi_0$ is the identity. If $J\subseteq S$ let $w_J$ be the longest element of the parabolic subgroup $W_J$ generated by $J$.

The \text{inversion set} of $w\in W$ is 
$
\Phi(w)=\{\alpha\in \Phi^+\mid w^{-1}\alpha\in -\Phi^+\}.
$
We note that the inversion set of the highest root $\varphi$ is
\begin{align}\label{eq:highestrootinversionset}
\Phi(s_{\varphi})=\{\alpha\in\Phi^+\mid \langle\alpha,\omega_i\rangle>0\text{ for some $i\in\wp$}\},
\end{align}
which follows directly from the equation $s_{\varphi}(\alpha)=\alpha-\langle\alpha,\varphi^{\vee}\rangle\varphi$.

Let $\mathbb{F}$ be a field, and let $G_0=G_0(\Phi,\mathbb{F})$ be the associated adjoint Chevalley group. Thus $G_0$ is generated by elements $x_{\alpha}(a)$ with $\alpha\in\Phi$ and $a\in\FF$, and writing (for $\alpha\in\Phi$ and $c\in\FF^{\times}$)
\begin{align*}
s_{\alpha}(c)=x_{\alpha}(c)x_{-\alpha}(-c^{-1})x_{\alpha}(c)\quad\text{and}\quad h_{\alpha^{\vee}}(c)=s_{\alpha}(c)s_{\alpha}(1)^{-1}
\end{align*}
the following relations hold (for $a,b\in\FF$, $\alpha,\beta\in \Phi$ with $\beta\neq \pm \alpha$, and $c,d\in\FF^{\times}$)
\begin{align*}
x_{\alpha}(a)x_{\alpha}(b)&=x_{\alpha}(a+b)\\
h_{\alpha^{\vee}}(c)h_{\alpha^{\vee}}(d)&=h_{\alpha^{\vee}}(cd)\\
x_{\alpha}(a)x_{\beta}(b)&=x_{\beta}(b)x_{\alpha}(a)\prod x_{i\alpha+j\beta}(C_{\alpha,\beta}^{i,j}a^ib^j),
\end{align*}
where the product is taken over $i,j\geq 1$ with $i\alpha+j\beta\in \Phi$ in any fixed order, and the elements $C_{\alpha,\beta}^{i,j}$ are integers (depending on the order chosen in the product). For example, in the simply laced case these commutator relations take the form $x_{\alpha}(a)x_{\beta}(b)=x_{\beta}(b)x_{\alpha}(a)$ (if $\alpha+\beta\notin\Phi$) or $x_{\alpha}(a)x_{\beta}(b)=x_{\beta}(b)x_{\alpha}(a)x_{\alpha+\beta}(C_{\alpha,\beta}ab)$ for some integer~$C_{\alpha,\beta}$ (if $\alpha+\beta\in\Phi$), and it turns out that in this case $C_{\alpha,\beta}=\pm 1$. 

The above relations imply the following useful formula (for $\alpha,\beta\in\Phi$ and $a\in\FF$)
\begin{align*}
s_{\alpha}(1)x_{\beta}(a)s_{\alpha}(1)^{-1}=x_{s_{\alpha}\beta}(\epsilon_{\alpha\beta}a)
\end{align*}
where $\epsilon_{\alpha\beta}=\pm 1$ are related to initial choices made in the Lie algebra (see \cite[Proposition~4.2.2]{Car:89}). For many calculations it is sufficient to simply know that $\epsilon_{\alpha\beta}\in\{-1,1\}$, however when more precise knowledge is required we will adopt the sign conventions from the Groups of Lie Type package in Magma~\cite{MAGMA,CMT:04}.

Let $G=G(\Phi,\mathbb{F})$ be the subgroup of $\mathrm{Aut}(G_0)$ generated by the inner automorphisms $G_0$ and the \textit{diagonal automorphisms}, as in \cite{St:60,Hum:69}. Thus $G$ is generated by $G_0$ and elements $h_{\lambda}(c)$ with $\lambda\in P$ and $c\in\FF^{\times}$, and the following relations hold (for $a\in \FF$, $c,d\in\FF^{\times}$, $\alpha,\beta\in\Phi$, and $\lambda,\mu\in P$)
\begin{align*}
h_{\lambda}(c)h_{\mu}(d)&=h_{\mu}(c)h_{\lambda}(d)& h_{\lambda}(c)h_{\lambda}(d)&=h_{\lambda}(cd)\\
h_{\lambda}(c)x_{\alpha}(a)h_{\lambda}(c)^{-1}&=x_{\alpha}(ac^{\langle\lambda,\alpha\rangle})&s_{\alpha}(1)h_{\lambda}(d)s_{\alpha}(1)^{-1}&=h_{s_{\alpha}\lambda}(d).
\end{align*}
%

For each $\alpha\in\Phi$ we write $U_{\alpha}=\langle x_{\alpha}(a)\mid a\in\mathbb{F}\rangle$ and $U^+=\langle U_{\alpha}\mid \alpha\in\Phi^+\rangle$. Let 
$$N=\langle s_{\alpha}(c)\mid \alpha\in\Phi, c\in\mathbb{F}^{\times}\rangle,\quad H=\langle h_{\lambda}(c)\mid \lambda\in P,c\in\mathbb{F}^{\times}\rangle,\quad\text{and}\quad B=\langle U^+,H\rangle=HU^+.$$ 
The subgroup $B$ is often called the (standard) \textit{Borel} subgroup. We have $H=B\cap N$, and $(B,N)$ is a $BN$-pair in $G$ with Weyl group $N/H\cong W$, where 
$$
s_{\alpha}(c)H\mapsto s_{\alpha}\quad\text{for all $c\in \FF^{\times}$}.
$$ 
We often write $wH$ (or $wB$) in place of $nH$ (or $nB$) whenever $n\in N$ with $nH\mapsto w$. In fact we will frequently write $s_{\alpha}$ in place of $s_{\alpha}(1)$ when there is no risk of confusion, however note that $s_{\alpha}\in G$ is typically not an involution.

The Bruhat decomposition gives
$$
G=\bigsqcup_{w\in W}BwB.
$$
For subsets $A\subseteq \Phi^+$ we write $U_A^+=\langle x_{\alpha}(a)\mid \alpha\in A,\,a\in\FF\rangle$. A subset $A\subseteq \Phi^+$ is \textit{closed} if $\alpha,\beta\in A$ and $\alpha+\beta\in \Phi$ implies that $\alpha+\beta\in A$. It is a fundamental fact that if $A\subseteq\Phi^+$ is closed, and if $(\beta_1,\ldots,\beta_k)$ is a fixed choice of ordering of the elements of~$A$, then each $u\in U_A^+$ has a unique expression as 
$
u=x_{\beta_1}(a_1)\cdots x_{\beta_k}(a_k)
$ for some $a_1,\ldots,a_k\in\FF$ (see~\cite[Lemma~17]{St:16}). In particular, since the set $A=\Phi(w)$, with $w\in W$, is closed, the $B$ cosets in $BwB$ are precisely
\begin{align}\label{eq:indexpoints}
x_{\beta_1}(a_1)\cdots x_{\beta_k}(a_k)wB,\quad\text{where $\Phi(w)=\{\beta_1,\ldots,\beta_k\}$ and $a_1,\ldots,a_k\in\FF$}.
\end{align}
We also note that
\begin{align}\label{eq:doublecoset}
BwB\cdot BsB=\begin{cases}
BwsB&\text{if $\ell(ws)=\ell(w)+1$}\\
BwB\cup BwsB&\text{if $\ell(ws)=\ell(w)-1$.}
\end{cases}
\end{align}

Let $U^-=\langle U_{\alpha}\mid \alpha\in-\Phi^+\rangle$. Throughout this paper we often need to convert an element in $U^-$ to an expression in $BwB$ for some $w$. To do so, we make frequent use of the relation 
\begin{align}\label{eq:folding}
x_{-\alpha}(a)=x_{\alpha}(a^{-1})s_{\alpha}(-a^{-1})x_{\alpha}(a^{-1})=x_{\alpha}(a^{-1})s_{\alpha}x_{\alpha}(a)h_{\alpha^{\vee}}(-a)\quad\text{for $a\neq 0$}
\end{align}
(which follows from the definition of $s_{\alpha}(a)$). We call this the \textit{folding relation}, due to connections with path models in algebraic combinatorics (see~\cite{PRS:09}).

We say that $\FF$ has ``special characteristic'' if $\mathrm{char}(\FF)=2$ for $\Phi=\sB_n,\sC_n,\sF_4$, or $\mathrm{char}(\FF)=3$ for $\Phi=\sG_2$. Often these cases behave differently due to additional symmetries being present.

We record some basic lemmas for later use.

\begin{lemma}\label{lem:perpendicularroots}
Let $\beta_1,\ldots,\beta_N\in\Phi^+$ be mutually perpendicular roots. Then 
\begin{align*}
x_{-\beta_1}(a_1)\cdots x_{-\beta_N}(a_N)\in Bs_{\beta_1}\cdots s_{\beta_N}B\quad\text{for all $a_1,\ldots,a_N\neq 0$.}
\end{align*}
\end{lemma}

\begin{proof}
Let $U_k^+=\langle U_{\beta_1},\ldots,U_{\beta_k}\rangle$ for $1\leq k\leq N$. We show, by induction, that
$$x_{-\beta_1}(a_1)\cdots x_{-\beta_N}(a_N)\in U_N^+s_{\beta_1}\cdots s_{\beta_N}U_N^+H.$$ 
The case $N=1$ is the folding relation~(\ref{eq:folding}). By the induction hypothesis, and the folding relation, for $k>1$ we have
\begin{align*}
x_{\beta_1}(a_1)\cdots x_{\beta_{k}}(a_{k})&=us_{\beta_1}\cdots s_{\beta_{k-1}}u'h\cdot x_{\beta_k}(a_k^{-1})s_{\beta_k}x_{\beta_k}(a_k)h_{\beta_k^{\vee}}(-a_k)
\end{align*}
for some $u,u'\in U_{k-1}^+$ and $h\in H$. Then since $h x_{\beta_k}(a_k^{-1})s_{\beta_k}x_{\beta_k}(a_k)=x_{\beta_k}(a)s_{\beta_k}x_{\beta_k}(b)h$ for some $a,b\in\FF$ and $s_{\beta_1}\cdots s_{\beta_{k-1}}u'x_{\beta_k}(a)=x_{\beta_k}(\pm a)s_{\beta_1}\cdots s_{\beta_{k-1}}u'$ (as $\beta_k$ is orthogonal to $\beta_1,\ldots,\beta_{k-1}$) we have
$
x_{\beta_1}(a_1)\cdots x_{\beta_{k}}(a_{k})=ux_{\beta_k}(\pm a)s_{\beta_1}\cdots s_{\beta_{k-1}}u'\cdot s_{\beta_k}x_{\beta_k}(b)hh_{\beta_k^{\vee}}(-a_k).$ 
Similarly, $u's_{\beta_k}=s_{\beta_k}u'$, and hence the result.
\end{proof}

 \begin{lemma}\label{lem:longelt}
Let $\Phi$ have rank~$n$, and suppose that the opposition relation is type preserving. If $\beta_1,\ldots,\beta_n\in\Phi^+$ are mutually perpendicular roots then $s_{\beta_1}\cdots s_{\beta_n}=w_0$. 
 \end{lemma}
 
 \begin{proof} 
 Since $\beta_1,\ldots,\beta_n$ are mutually perpendicular the product $s_{\beta_1}\cdots s_{\beta_n}$ acts by $-1$ on the vector space $V$. Since opposition is type preserving, the longest element $w_0$ also acts by $-1$ (mapping $\alpha_i$ to $-\alpha_i$ for all simple roots), hence the result.
 \end{proof}

 \begin{lemma}\label{lem:longeltE6}
Let $\Phi$ be a root system of type $\sE_6$ in a vector space~$V$, and let $\sigma:V\to V$ be the involution given by $\sigma(\alpha_i)=\alpha_{\pi_0(i)}$ for $1\leq i\leq 6$. Suppose that $\beta_1,\beta_2,\beta_3,\beta_4\in\Phi^+$ are mutually perpendicular roots with $\sigma(\beta_i)=\beta_i$ for $i=1,2,3,4$. Then $s_{\beta_1}s_{\beta_2}s_{\beta_3}s_{\beta_4}=w_0$.
 \end{lemma}

 \begin{proof}
 Let $w=s_{\beta_1}s_{\beta_2}s_{\beta_3}s_{\beta_4}$. Let $V'=\{v\in V\mid \sigma(v)=v\}$. Then $V'$ is $4$-dimensional, and since $\beta_1,\beta_2,\beta_3,\beta_4\in V'$ are mutually perpendicular we have $s_{\beta_1}\cdots s_{\beta_4}|_{V'}=-1$. In particular, $\alpha_2,\alpha_4\in\Phi(w^{-1})$. Moreover, since $w\alpha_6=w\sigma(\alpha_1)=\sigma(w\alpha_1)$ (because $\sigma$ commutes with each reflection $s_{\beta}$ with $\beta\in V'$) we have $w\alpha_1\in\Phi^+$ if and only if $w\alpha_6\in\Phi^+$. But $\alpha_1+\alpha_6\in V'$, and so $w(\alpha_1+\alpha_6)=-\alpha_1-\alpha_6$. It follows that $\alpha_1,\alpha_6\in\Phi(w^{-1})$, and similarly $\alpha_3,\alpha_5\in\Phi(w^{-1})$. Thus $\alpha_1,\ldots,\alpha_6\in\Phi(w^{-1})$, and so $w=w_0$. 
\end{proof}

\subsection{Split spherical buildings}

We assume that the reader is already with the basic theory of buildings, and our main reference for the general theory is~\cite{AB:08}. By a \textit{split} spherical building we shall mean a building associated to a Chevalley group via the standard $BN$-pair construction. It is easiest to define this building as a $W$-metric space (c.f. \cite[Chapter~5]{AB:08}), as follows.

\begin{defn}
The split spherical building $\Delta=\Delta_{\Phi}(\mathbb{F})$ associated to $G=G_{\Phi}(\FF)$ has chamber set $\Delta=G/B$ and Weyl distance function given by 
$$
\delta(gB,hB)=w\quad\text{if and only if}\quad g^{-1}h\in BwB.
$$ 
Chambers $c,d\in\Delta$ are $s$-\textit{adjacent} (with $s\in S$) if $\delta(c,d)=s$, and are \textit{adjacent} if they are $s$-adjacent for some $s\in S$.
\end{defn}

In particular, if $c=gB$ is a chamber of $\Delta$, then by~(\ref{eq:indexpoints}) the set of chambers $d\in\Delta$ with $\delta(c,d)=w$ is precisely
$$
g\cdot\{x_{\beta_1}(a_1)\cdots x_{\beta_k}(a_k)wB\mid a_1,\ldots,a_k\in\FF\}\quad\text{where $\Phi(w)=\{\beta_1,\ldots,\beta_k\}$}.
$$

We often regard $\Delta$ as a simplicial complex in the standard way (c.f. \cite[Chapter~4]{AB:08}). Let us briefly describe this conversion in a group theoretic way in the split context. For subsets $J\subseteq S$ let 
$$
P_{J}=\bigcup_{w\in W_{J}}BwB
$$
be the standard parabolic subgroup of $G$ of type $J$. For each nonempty $J\subseteq S$ the set of ``type $J$-simplices'' of the building is the set of cosets $G/P_{S\backslash J}$, and the simplicial complex structure is given by \textit{reverse} containment of cosets. For example, in the simplicial complex language the chamber $B$ ``contains'' the simplices $P_{S\backslash J}$ for all nonempty $J\subseteq S$, whereas on the level of cosets it is in fact the parabolic subgroups $P_{S\backslash J}$ that contain the Borel subgroup~$B$. 

If $J=\{s\}$ is a singleton we often write 
\begin{align}\label{eq:notationparabolic}
W_s=W_{S\backslash J}\quad\text{and}\quad P_s=W_{S\backslash J}
\end{align}
for the standard parabolic subgroups of $W$ and $G$ of type $S\backslash J$. Moreover, if $s=s_i$ we will often write $W_{s_i}=W_i$ and $P_{s_i}=P_i$.

Let $\tau(x)\subseteq S$ denote the type of the simplex~$x$ of $\Delta$. 
 Thus vertices are simplices $x$ with $\tau(x)=\{s\}$ for some $s\in S$, and chambers are the simplices $x$ with $\tau(x)=S$. A \textit{panel} is a codimension~$1$ simplex; that is, $\tau(x)=S\backslash \{s\}$ for some $s\in S$.

 An \textit{automorphism} of $\Delta$ is an adjacency preserving bijection $\theta:\Delta\to\Delta$. Each automorphism $\theta$ of $\Delta$ induces an automorphism $\pi_{\theta}$ of the Coxeter diagram by $\delta(c,d)=s$ if and only if $\delta(\theta(c),\theta(d))=\pi_{\theta}(s)$. We say that $\theta$ is a \textit{collineation} (or \textit{type preserving}) if $\pi_{\theta}=\mathrm{id}$, and a \textit{duality} if $\pi_{\theta}$ has order~$2$. 
 
By \cite[Corollaries~5.9 and~5.10]{Tit:74} (and using \cite{Hum:69,St:60}), every automorphism $\theta$ of $\Delta$ is of the form $\theta=g\circ \pi\circ\sigma$, where $g\in G$, $\pi=\pi_{\theta}$ is a Dynkin diagram automorphism, and $\sigma$ is a field automorphism (in the special characteristic case $\pi$ is a Coxeter diagram automorphism). Note that the ``diagonal automorphisms'' are already built into~$G$. By the Bruhat decomposition, each element $g\in G$ can be written as $g=uwb$ with $u\in U^+$, $w\in W$, and $b\in B$, and so each automorphism of $\Delta$ can be written as $\theta=uwb\circ\pi\circ\sigma$. If $\sigma$ is trivial, we say that $\theta$ is \emph{linear}.

By a \textit{root elation} we shall mean an automorphism $\theta$ conjugate to $x_{\alpha}(a)$ for some $\alpha\in\Phi$ and $a\in\FF^{\times}$, and we call $\theta$ a \textit{long} (respectively \textit{short}) root elation if $\alpha$ is a long (respectively short) root. By a \textit{homology} we shall mean an automorphism $\theta$ conjugate to a nontrivial element $h_{\lambda}(c)$ with $\lambda\in P$ and $c\in\FF^{\times}$.

\subsection{Opposition diagrams and admissible diagrams}\label{sec:autos}

Chambers $c,d\in\Delta$ are \textit{opposite} one another if and only if $\delta(c,d)=w_0$. If $x,y$ are simplices, with types $J,K$ respectively, then $x$ and $y$ are \textit{opposite} one another if and only if $K=\pi_0(J)$ and there exist opposite chambers $c,d$ with $x\subseteq c$ and $y\subseteq d$. That is, simplices are opposite one another if they have opposite types, and are contained in opposite chambers. In terms of double cosets, chambers $c=gB$ and $d=hB$ are opposite if and only if $g^{-1}h\in Bw_0B$, and simplices $x=gP_{S\backslash J}$ and $y=hP_{S\backslash K}$ are opposite if and only if $K=\pi_0(J)$ and $g^{-1}h\in P_{S\backslash J}w_0P_{S\backslash K}$.

Let $\theta$ be an automorphism of $\Delta$. Recall, from the introduction, that $\Opp(\theta)$ denotes the set of all simplices $x$ such that $x^{\theta}$ is opposite $x$. The \textit{type} $\Type(\theta)$ of $\theta$ is the union of all subsets $J\subseteq S$ such that there exists a type~$J$ simplex mapped to an opposite simplex by~$\theta$. The \textit{opposition diagram} of $\theta$ is the triple $(\Gamma,\Type(\theta),\pi_{\theta})$. 

Less formally, the opposition diagram of $\theta$ is depicted by drawing $\Gamma$ and encircling the nodes of $\Type(\theta)$, where we encircle nodes in minimal subsets invariant under $\pi_0\circ\pi_{\theta}$. We draw the diagram ``bent'' (in the standard way) if $\pi_0\circ\pi_{\theta}\neq\mathrm{id}$. For example, consider the diagrams 
$$\sX=
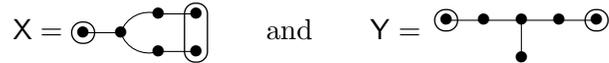
\begin{tikzpicture}[scale=0.5,baseline=-0.5ex]
\node [inner sep=0.8pt,outer sep=0.8pt] at (-2,0) (2) {$\bullet$};
\node [inner sep=0.8pt,outer sep=0.8pt] at (-1,0) (4) {$\bullet$};
\node [inner sep=0.8pt,outer sep=0.8pt] at (0,-0.5) (5) {$\bullet$};
\node [inner sep=0.8pt,outer sep=0.8pt] at (0,0.5) (3) {$\bullet$};
\node [inner sep=0.8pt,outer sep=0.8pt] at (1,-0.5) (6) {$\bullet$};
\node [inner sep=0.8pt,outer sep=0.8pt] at (1,0.5) (1) {$\bullet$};
\draw (-2,0)--(-1,0);
\draw (-1,0) to [bend left=45] (0,0.5);
\draw (-1,0) to [bend right=45] (0,-0.5);
\draw (0,0.5)--(1,0.5);
\draw (0,-0.5)--(1,-0.5);
\draw [line width=0.5pt,line cap=round,rounded corners] (2.north west)  rectangle (2.south east);
\draw [line width=0.5pt,line cap=round,rounded corners] (1.north west)  rectangle (6.south east);
\end{tikzpicture}\qquad\text{and}\qquad\mathsf{Y}=\,\begin{tikzpicture}[scale=0.5,baseline=-1.5ex]
\node [inner sep=0.8pt,outer sep=0.8pt] at (-2,0) (1) {$\bullet$};
\node [inner sep=0.8pt,outer sep=0.8pt] at (-1,0) (3) {$\bullet$};
\node [inner sep=0.8pt,outer sep=0.8pt] at (0,0) (4) {$\bullet$};
\node [inner sep=0.8pt,outer sep=0.8pt] at (1,0) (5) {$\bullet$};
\node [inner sep=0.8pt,outer sep=0.8pt] at (2,0) (6) {$\bullet$};
\node [inner sep=0.8pt,outer sep=0.8pt] at (0,-1) (2) {$\bullet$};
\draw (-2,0)--(2,0);
\draw (0,0)--(0,-1);
\draw [line width=0.5pt,line cap=round,rounded corners] (1.north west)  rectangle (1.south east);
\draw [line width=0.5pt,line cap=round,rounded corners] (6.north west)  rectangle (6.south east);
\end{tikzpicture}$$
Diagram $\sX$ represents a collineation $\theta$ of an $\sE_6$ building with $\Type(\theta)=\{1,2,6\}$, and diagram $\mathsf{Y}$ represents a duality of an $\sE_6$ building with $\Type(\theta)=\{1,6\}$.

An automorphism $\theta$ is \textit{domestic} if $\Opp(\theta)$ contains no chamber (that is, $\theta$ maps no chamber to an opposite chamber). More generally, if $J\subseteq S$ then $\theta$ is $J$-domestic if no type $J$ simplex is mapped onto an opposite simplex by~$\theta$. To avoid trivialities, the definition of $J$-domesticity is restricted to subsets $J$ with $\pi_0(J)=\pi_{\theta}(J)$ (for if $J$ does not satisfy this, then $\theta$ is $J$-domestic for trivial reasons).

An automorphism $\theta$ is called \textit{capped} if the following closure property holds: If there exist type $J_1$ and $J_2$ simplices in $\Opp(\theta)$, then there exists a type $J_1\cup J_2$ simplex in $\Opp(\theta)$. Equivalently, $\theta$ is capped if and only if there exists a type $\Type(\theta)$ simplex in $\Opp(\theta)$. By \cite[Theorem~1]{PVM:19a} every automorphism of a ``large'' spherical building of rank at least~$3$ is capped, where a building is called \textit{large} if it contains no Fano plane residues (for split buildings this simply means $|\FF|>2$).

In \cite{PVM:19a,PVM:19b} we showed that the opposition diagrams of automorphisms of spherical buildings satisfy various restrictive properties, and we used these properties to determine a list of all possible opposition diagrams. We call the diagrams $(\Gamma,J,\pi)$ in this list \textit{admissible Dynkin diagrams} (more precisely, in \cite{PVM:19a,PVM:19b} we considered Coxeter diagrams rather than Dynkin diagrams, however the arguments are nearly identical). The complete list of admissible Dynkin diagrams of exceptional type is given in Figure~\ref{fig:Dynkin} (taken from~\cite{PVM:19a}). Each admissible diagram $(\Gamma,J,\pi)$ of exceptional type is denoted by a symbol
$$
{^t}\sX_{n;i}\quad\text{or}\quad {^t}\sX_{n;i}^{k},
$$
where 
\begin{compactenum}[$(1)$]
\item $\sX\in\{\sE,\sF,\sG\}$ is the type of $\Gamma$, and $n$ is the rank;
\item $t\in\{1,2\}$ is the order of the graph automorphism $\pi_0\circ\pi$ (the ``twisting'');
\item $i$ is the number of distinguished orbits contained in $J$; and
\item $k$ is an additional index occurring only for $\sF_4$ and $\sG_2$ in the case that a single node is encircled, in which case $k$ is the type of this node.
\end{compactenum}
In the case $t=1$ (that is, when $\pi_0\circ\pi=\mathrm{id}$) we usually omit the $t$ from the notation, writing simply ${^t}\sX_{n;i}=\sX_{n;i}$. For example, the diagrams $\sX$ and $\mathsf{Y}$ given above are $\sX={^2}\sE_{6;2}$ and $\mathsf{Y}=\sE_{6;2}$. Similar notation is introduced in \cite{PVM:20a} for the classical types (except with a different meaning for the index~$k$).

In special characteristic one often ignores the arrows on Dynkin diagrams, thus giving the additional admissible (Coxeter) diagrams in Figure~\ref{fig:Coxeter}.
\begin{figure}[h!]
\begin{center}
$
{^2}\sG_{2;1}=\begin{tikzpicture}[scale=0.5,baseline=-0.5ex]
\node at (7,0.8) {};
\node at (7,-0.8) {};
\node at (7,0) (0) {};
\node [inner sep=0.8pt,outer sep=0.8pt] at (7,0.5) (6a) {$\bullet$};
\node [inner sep=0.8pt,outer sep=0.8pt] at (7,-0.5) (6b) {$\bullet$};
\draw [line width=0.5pt,line cap=round,rounded corners] (6a.north west)  rectangle (6b.south east);
\draw [domain=-90:90] plot ({7+(0.6)*cos(\x)}, {(0.4)*sin(\x)});
\draw [domain=-90:90] plot ({7+(0.7)*cos(\x)}, {(0.5)*sin(\x)});
\draw [domain=-90:90] plot ({7+(0.8)*cos(\x)}, {(0.6)*sin(\x)});
\end{tikzpicture}\qquad\text{and}\qquad
{^2}\sF_{4;2}=\begin{tikzpicture}[scale=0.5,baseline=-0.5ex]
\node at (6,0.8) {};
\node at (6,-0.8) {};
\node at (6,0) (0) {};
\node [inner sep=0.8pt,outer sep=0.8pt] at (6,0.5) (5a) {$\bullet$};
\node [inner sep=0.8pt,outer sep=0.8pt] at (6,-0.5) (5b) {$\bullet$};
\node [inner sep=0.8pt,outer sep=0.8pt] at (7,0.5) (6a) {$\bullet$};
\node [inner sep=0.8pt,outer sep=0.8pt] at (7,-0.5) (6b) {$\bullet$};
\draw (6,0.5)--(7,0.5);
\draw (6,-0.5)--(7,-0.5);
\draw [line width=0.5pt,line cap=round,rounded corners] (6a.north west)  rectangle (6b.south east);
\draw [line width=0.5pt,line cap=round,rounded corners] (5a.north west)  rectangle (5b.south east);
\draw [domain=-90:90] plot ({7+(0.6)*cos(\x)}, {(0.4)*sin(\x)});
\draw [domain=-90:90] plot ({7+(0.8)*cos(\x)}, {(0.6)*sin(\x)});
\end{tikzpicture}
$
\end{center}
\caption{Additional admissible diagrams for special characteristic}\label{fig:Coxeter}
\end{figure}

Summarising the above discussion, we have the following result from \cite{PVM:19a,PVM:19b}.

\begin{thm}\label{thm:admissible}
If $\theta$ is an automorphism of a spherical building of exceptional type, then the opposition diagram of $\theta$ is listed in Figures~\ref{fig:Dynkin} or~\ref{fig:Coxeter}.
\end{thm}

A striking feature of Theorem~\ref{thm:admissible} is that there are very few possible opposition diagrams. We note that the analysis in~\cite{PVM:19a,PVM:19b} does not prove the converse to Theorem~\ref{thm:admissible}. That is, a priori there may be redundancies in the list of admissible diagrams, in the sense that some admissible diagrams may not actually be the opposition diagram of any automorphism of any spherical building. It is a consequence of the work of this paper, combined with \cite{PVM:20a}, that no such redundancies exist -- more precisely, Corollary~\ref{cor:existence} holds.

We call an admissible diagram \textit{empty} if no nodes are encircled, and \textit{full} if all nodes are encircled. We call an admissible diagram $(\Gamma,J,\pi)$ \textit{type preserving} if $\pi=\mathrm{id}$. By inspection of the list in Figure~\ref{fig:Dynkin} we note that for each Dynkin diagram~$\Gamma$ there exists a type preserving admissible diagram $(\Gamma,J,\mathrm{id})$ with $J=\wp$. This diagram is called the \textit{polar diagram}. 

The \textit{polar-copolar diagram} is the type preserving admissible diagram with $J=\wp\cup \wp^*$, where $\wp^*$ is the polar type of the type $S\backslash \wp$ residue (we call $\wp^*$ the \textit{copolar type}; again, by inspection of Figure~\ref{fig:Dynkin} the triple $(\Gamma,\wp\cup\wp^*,\mathrm{id})$ is always an admissible diagram). See Appendix~\ref{app:data} for the list of copolar types. Alternatively, the polar-copolar diagrams of exceptional type are characterised as the type preserving diagrams in which exactly two orbits of nodes are encircled. Specifically, the polar diagrams of exceptional type are ${^2}\sE_{6;1}$, $\sE_{7;1}$, $\sE_{8;1}$, $\sF_{4;1}^1$, and $\sG_{2;1}^2$, and the polar-copolar diagrams are ${^2}\sE_{6;2}$, $\sE_{7;2}$, $\sE_{8;2}$, $\sF_{4;2}$, and $\sG_{2;2}$.

Every duality of a thick $\sG_2$ (respectively $\sF_4$) building has opposition diagram ${^2}\sG_{2;1}$ (respectively ${^2}\sF_{4;2}$), and no such dualities are domestic (see  \cite[Theorem~2.7]{PTM:15} and \cite[Lemma~4.1]{PVM:19a}). We note that dualities of split $\sF_4$ and $\sG_2$ buildings only exist for perfect fields of special characteristic.

\subsection{Basic techniques}

It is generally rather difficult to prove that a given automorphism is domestic, and moreover to compute its opposition diagram. In this section we describe some of the techniques that we will use in this paper. 

The \textit{displacement} of an automorphism~$\theta$ is
$$
\disp(\theta)=\max\{\ell(\delta(c,c^{\theta}))\mid c\in\Delta\}.
$$
Thus $\theta$ is domestic if and only if $\disp(\theta)<\ell(w_0)$. Moreover, by \cite[Corollary~2.29]{PVM:19b} we have
$$
\disp(\theta)=\begin{cases}
\ell(w_{S\backslash J}w_0)&\text{if $\theta$ is capped}\\
\ell(w_{S\backslash J}w_0)-1&\text{if $\theta$ is uncapped}
\end{cases}
$$
where $J=\Type(\theta)$. Uncapped automorphisms will not play a significant role here, so assume that $\theta$ is capped. By the above comments, and the classification of admissible diagrams, the list of possible displacements of $\theta$ is very restricted. For example, for a capped automorphism of an $\sE_7$ building, the displacements for the non-full non-empty opposition diagrams are:
\begin{align*}
33\text{ for $\sE_{7;1}$},\quad 50\text{ for $\sE_{7;2}$},\quad 51\text{ for $\sE_{7;3}$},\quad 60\text{ for $\sE_{7;4}$}.
\end{align*}
Thus, for example, to show that a capped automorphism $\theta$ of an $\sE_7$ building is:
\begin{compactenum}[$(1)$]
\item not domestic, it is sufficient to show that $\disp(\theta)>60$;
\item has opposition diagram $\sE_{7;3}$ it is sufficient to show that $\disp(\theta)<60$ and that there is a type $\{1,6,7\}$ simplex mapped to an opposite (in fact, from the classification of diagrams, it would be sufficient to show that there is a type $7$ vertex mapped to an opposite). 
\end{compactenum}
Such arguments will be used on multiple occasions throughout the paper.

The following lemma is useful to compute displacement. 

\begin{lemma}\label{lem:red2}
Let $\theta$ be an automorphism of a thick spherical building~$\Delta$, and let $N=\disp(\theta)$. Let $c$ be any chamber. Suppose that either
\begin{compactenum}[$(1)$]
\item each panel of $\Delta$ has at least $4$ chambers, or
\item $\theta$ is an involution, or
\item $\theta$ induces opposition and $N=\ell(w_0)$.
\end{compactenum}
Then $\theta$ is necessarily capped, and there exists a chamber $d$ with $\delta(c,d)=w_0$ and $\ell(\delta(d,d^{\theta}))=N$. In particular,
$$
\disp(\theta)=\max\{\ell(\delta(d,d^{\theta}))\mid d\in \Delta\text{ with $\delta(c,d)=w_0$}\}.
$$
\end{lemma}

\begin{proof}
To see that $\theta$ is capped: If each panel has at least $4$ chambers then $\theta$ is capped by \cite[Theorem~1]{PVM:19a}, if $\theta$ is an involution then $\theta$ is capped by~\cite[Corollary~2.22]{PVM:19b}, and if $\disp(\theta)=\ell(w_0)$ then $\theta$ maps a chamber to an opposite, and hence is capped. The rest of the lemma is contained in \cite[Lemma~4.1]{PVM:19b}.
\end{proof}

The following Corollary is useful to prove $J$-domesticity. 

\begin{cor}\label{cor:oppsphere}
Let $\theta$ be an automorphism of a thick spherical building~$\Delta$, and suppose that the hypothesis of Lemma~\ref{lem:red2} is satisfied. Let $c$ be any chamber. For each subset $J\subseteq S$, the automorphism $\theta$ is $J$-domestic if and only if it is $J$-domestic when restricted to the sphere of chambers opposite~$c$. That is, for every chamber $d$ with $\delta(c,d)=w_0$, the type $J$-simplex of~$d$ is not mapped onto an opposite simplex. 
\end{cor}

\begin{proof}
By Lemma~\ref{lem:red2} $\theta$ is capped and there is a chamber $d$ with $\delta(c,d)=w_0$ such that $\ell(\delta(d,d^{\theta}))=\disp(\theta)$. Since $\theta$ is capped, by \cite[Theorem~2.6]{PVM:19a} we have $\delta(d,d^{\theta})=w_{S\backslash J}w_0$, where $J=\Type(\theta)$. In particular, the type $J$-simplex of $d$ is mapped to an opposite simplex, hence the result.
\end{proof}

The following proposition gives a useful technique for proving domesticity. We refer to this technique as the ``standard technique''.

\begin{prop}\label{prop:standardtechnique}
Let $\Delta=\Delta_{\Phi}(\mathbb{F})$ be split, and let $\theta\in G=G_{\Phi}(\mathbb{F})$. Suppose that the hypothesis of Lemma~\ref{lem:red2} is satisfied for the automorphism $\theta$. If there is $w_1\in W$ with 
$$w_1^{-1}w_0^{-1}u^{-1}\theta uw_0w_1\in B\quad\text{for all $u\in U^+$}
$$ 
then $
\disp(\theta)\leq 2\ell(w_1)-1$. Thus, in particular, if $\ell(w_1)\leq \ell(w_0)/2$ then $\theta$ is domestic. 
\end{prop}

\begin{proof} Each chamber $gB$ of the building $G/B$ can be written uniquely as $uwB$ for some $w\in W$ and some $u\in \langle U_{\alpha}\mid \alpha\in \Phi(w)\rangle$. Then $\delta(gB,\theta gB)$ is the unique element $v\in W$ such that $w^{-1}u^{-1}\theta uw\in BvB$. If the hypothesis of Lemma~\ref{lem:red2} is satisfied, then the displacement of $\theta$ is achieved for some chamber $gB$ opposite the base chamber $B$. These chambers are of the form $uw_0B$ with $u\in U^+$. By the hypothesis we have
$$
w_0^{-1}u^{-1}\theta uw_0\in w_1Bw_1^{-1}\subseteq Bw_1B\cdot Bw_1^{-1}B.
$$
In particular, if $w=\delta(uw_0B,\theta uw_0B)$ then $BwB\subseteq Bw_1B\cdot Bw_1^{-1}B$. Thus
\begin{align}\label{eq:further}
\disp(\theta)\leq\max\{\ell(w)\mid BwB\subseteq Bw_1B\cdot Bw_1^{-1}B\}.
\end{align}
Writing $w_1=w_2s$ with $\ell(w_2s)=\ell(w_2)+1$, by~(\ref{eq:doublecoset}) we have
 $$Bw_1B\cdot Bw_1^{-1}B=(Bw_2sB\cdot Bw_2^{-1}B)\cup (Bw_2B\cdot Bw_2^{-1}B),$$
  and it follows from double coset combinatorics (\ref{eq:doublecoset}) that $
\disp(\theta)\leq 2\ell(w_1)-1$.
\end{proof}

In the case that $\theta\in U^+$ the following lemma is helpful in finding an element $w_1\in W$ as in Proposition~\ref{prop:standardtechnique}. If $A\subseteq \Phi$ let 
$
A_{\geq}=\{\beta\in \Phi\mid \beta\geq \alpha\text{ for some $\alpha\in A$}\},
$
where $\alpha\leq \beta$ if and only if $\beta-\alpha$ is a nonnegative linear combination of positive roots. Let $\pi_0$ be the automorphism of $\Phi$ given by $\pi_0(\alpha)=-w_0\alpha$ (thus $\pi_0$ is the automorphism induced by the opposition diagram automorphism). 

\begin{lemma}\label{lem:technique}
If $\theta\in \langle U_{\alpha}\mid \alpha\in A\rangle$ for some $A\subseteq \Phi^+$ then  
$
u^{-1}\theta u\in \langle U_{\alpha}\mid \alpha\in A_{\geq}\rangle
$ for all $u\in U^+$. Moreover, if $w_1\in W$ is such that $\pi_0(A_{\geq})\subseteq \Phi(w_1)$ then $w_1^{-1}w_0^{-1}u^{-1}\theta u w_0w_1\in B$ for all $u\in U^+$. 
\end{lemma}

\begin{proof}
Let $F=\langle U_{\alpha}\mid \alpha\in A_{\geq}\rangle$. If $\beta\in \Phi^+$ and $\alpha\in A_{\geq}$ then by the commutator relations we have, for all $a\in\FF$,
$
x_{\beta}(a)^{-1}U_{\alpha}x_{\beta}(a)\subseteq U_{\alpha}U_{\alpha+\beta}U_{2\alpha+\beta}\cdots\subseteq F$ (with the convention that $U_{\gamma}=\{1\}$ if $\gamma\notin \Phi$) and it follows that $x_{\beta}(a)^{-1}Fx_{\beta}(a)\subseteq F$. Thus, by induction, $u^{-1}Fu\subseteq F$ for all $u\in U$, and since $\theta\in F$ we have $u^{-1}\theta u\in F$ as required. Thus if $w_1\in W$ is such that $\pi_0(A_{\geq})\subseteq \Phi(w_1)$ then
$$
w_0^{-1}u^{-1}\theta u w_0\in \langle U_{-\alpha}\mid \alpha\in \pi_0(A_{\geq})\rangle =w_1\langle U_{-w_1^{-1}\alpha}\mid \alpha\in \pi_0(A_{\geq})\rangle w_1^{-1}\subseteq w_1Bw_1^{-1},
$$
hence the result. 
\end{proof}

\subsection{Parapolar spaces and Lie incidence geometries}\label{sec:parapolar}

At certain points of this paper (in particular in Section~\ref{sec:polarcopolar}) we will work with Lie incidence geometries~$\sX_{n,J}(\FF)$. For example, the long root geometries $\sE_{6,2}(\FF)$, $\sE_{7,1}(\FF)$, $\sE_{8,8}(\FF)$, $\sF_{4,1}(\FF)$, and the geometries $\sE_{6,1}(\FF)$, $\sE_{7,7}(\FF)$ and $\sF_{4,4}(\FF)$. We note the similarity of notation with that used for opposition diagrams, however no confusion should arise. 

In general, the Lie incidence geometry $\sX_{n,J}(\FF)$, with $J\subseteq S$, is defined from the building $\Delta=\Delta_\Phi(\FF)$, with $\Phi$ the root system of type $\sX_n$, as the point-line geometry with point set the set of simplices (or flags) of type~$J$ of~$\Delta$, and a typical line is the set of flags of type $J$ incident with a flag of type $S\backslash\{j\}$ for some $j\in J$. 

If $J=\wp$, then we call the geometry $\sX_{n,J}(\FF)$ a \emph{long root} geometry. 

The geometries listed in the first paragraph of this subsection are all examples of \textit{parapolar spaces} (see \cite[Chapter~13]{Sch:10} for the basic terminology and definition). In particular, these point-line geometries $\mathscr{G}=(\mathcal{P},\mathcal{L})$ contain \textit{symplecta} (or \textit{symps} for short), being convex subsets that are non-degenerate polar spaces of rank at least~$2$. If all symplecta have the same rank~$r\geq 2$ then $\mathscr{G}$ is said to have \textit{symplectic rank}~$r$. Recall that in any incidence geometry, $x^{\perp}$ denotes the set of all points collinear to the point~$x$. Also, each symp $\xi$ is, by convexity, determined by any pair $\{x,y\}$ of non-collinear points and we denote $\xi=\xi(x,y)$. If $\sX_n=\sE_6,\sE_7,\sE_8,\sF_4$, then $\sX_{n,\wp}(\FF)$ has symplectic rank $4,5,7,3$, respectively. Moreover, the symps precisely correspond to the residues of vertices of type $\wp^*$.   

In the parapolar spaces $\sE_{6,2}(\FF)$, $\sE_{7,1}(\FF)$, $\sE_{8,8}(\FF)$, $\sF_{4,1}(\FF)$ and $\sF_{4,4}(\FF)$ there are precisely $5$ possible ``distances'' between two points $x,y$. Either (1) $x=y$,  (2) $x$ and $y$ are collinear, (3) $\{x,y\}$ lies in a symplecton, (4) $|x^{\perp}\cap y^{\perp}|=1$, or
(5) $|x^{\perp}\cap y^{\perp}|=0$. In case (3) we say that $x,y$ are \textit{symplectic} (or \textit{at symplectic distance}) and we denote by $x^{\perp\!\!\!\perp}$ the set of points at symplectic distance from $x$. in case (4) we say that $x,y$ are \textit{special} (or a \textit{special pair}, or \textit{at special distance}), and we denote by $x^{\Join}$ the set of points at special distance from $x$. Finally, in case (5) the points $x,y$ are opposite each other (in the building theoretic sense). 

The parapolar spaces $\sF_{4,1}(\FF)$ and $\sF_{4,4}(\FF)$ are also called \textit{metasymplectic spaces}. In the parapolar spaces $\sE_{6,1}(\FF)$ and $\sE_{7,7}(\FF)$ there are no special pairs, and so in these spaces there are only $4$ possible distances between two points (these spaces are called \textit{strong parapolar spaces}).

Let us briefly describe the long root geometry parapolar spaces in more algebraic terms. Let $\sX_n=\sE_n$ ($n=6,7,8$) or $\sF_4$, and let $p\in S$ be the polar node. Recall the notation~(\ref{eq:notationparabolic}). Let $M_p$ denote the set of minimal length coset representatives of cosets in $W/W_p$, and let $R_p$ denote the set of minimal length representatives for the double cosets in $W_p\backslash W/W_p$. The points of the long root geometry $\sX_{n,p}(\FF)$ are the cosets in $G/P_p$. We have $
G=\bigsqcup_{w\in R_p}P_pwP_p,
$
and the \textit{Weyl-distance} $\delta(g_1P_p,g_2P_p)$ between points $g_1P_p$ and $g_2P_p$ is defined to be the unique element $\delta(g_1P_p,g_2P_p)=w\in R_p$ with $g_1^{-1}g_2\in P_pwP_p$. 

In each case $R_p$ contains precisely~$5$ elements $e,s_p,w_1,w_2,s_{\varphi}$, arranged in increasing length, corresponding to the $5$ possible distances between points. Explicitly, the points $x=g_1P_p$ and $y=g_2P_p$ are collinear (respectively symplectic, at special distance, opposite) if $\delta(x,y)=s_p$ (respectively $w_1$, $w_2$, $s_{\varphi}$). Similar remarks hold for the metasymplectic space $\sF_{4,4}(\FF)$, and also the strong parapolar spaces $\sE_{6,1}(\FF)$ and $\sE_{7,7}(\FF)$ (where only $4$ distances are possible, and hence $w_2$ is omitted).

\section{Root elations and the polar type}\label{sec:polartype}

In this section we prove Theorem~\ref{thm:polarclass} for exceptional types (see Theorems~\ref{thm:RGD} and~\ref{thm:converse}). In fact we focus on types $\sE_6,\sE_7,\sE_8$ and $\sF_4$, with the case of $\sG_2$ following from the classification in Theorem~\ref{thm:G2class} (see Subsection~\ref{sec:G2}). We also provide a geometric characterisation of root elations for buildings of type $\sE_6$ and $\sE_7$ in Theorem~\ref{thm:geomchar}, and we discuss short root elations in the non-simply laced case in Theorem~\ref{thm:short}. The proofs of Corollaries~\ref{cor:existsdomestic} and~\ref{cor:existsconjugacy} are given in Subsection~\ref{sec:exists} (see also Corollary~\ref{cor:barbara2}). 

\subsection{Long root elations}\label{sec:lre}

To prove Theorem~\ref{thm:polarclass} we must first show that long root elations have polar opposition diagram, and conversely that every automorphism with polar opposition diagram is necessarily a long root elation. We prove the first statement in Theorem~\ref{thm:RGD} in a more general context of Moufang spherical buildings, and the second statement is proved in Theorem~\ref{thm:converse}.

%
%
%
%
%

Let $\Delta$ be an irreducible Moufang spherical building. Recall (for example, from~\cite{Tim:00}) that if $\Delta$ is not a generalised octagon, then one can associate a crystallographic (not necessarily reduced) root system $\Phi$ to $\Delta$ in such a way that the root subgroup $U_{\varphi}$, with $\varphi$ the highest root of $\Phi$, is contained in the centre of the positive root subgroup $U^+=\langle U_{\alpha}\mid \alpha\in \Phi^+\rangle$. In the case that $\Phi$ has only one root length we call all roots long. Let $\wp\subseteq S$ denote the polar type of the Dynkin diagram of $\Phi$. If $\Delta$ is a generalised octagon then by \cite{Tit:83} one may associate a non-crystallographic (and non-reduced) root system~$\Phi$, and again there is a ``highest root'' $\varphi$ such that $U_{\varphi}$ is contained in the centre of $U^+$ (the root $\varphi$ is $\alpha_4'$ in \cite{Tit:83}, and the polar type corresponds to ``points'' of the octagon). In all cases, let $W$ be the Weyl group of $\Phi$ (so $W$ is a dihedral group of order~$16$ in the octagon case).

\begin{thm}\label{thm:RGD} Let $\Delta$ be an irreducible Moufang spherical building with associated root system $\Phi$ and Weyl group~$W$ as above. Let $\alpha\in\Phi$ be a long root, and let $\theta\in U_{\alpha}\backslash\{1\}$. Then the collineation $\theta$ of $\Delta=G/B$ has polar opposition diagram. Moreover,
\begin{align}
\label{eq:attained}
\{\delta(c,c^{\theta})\mid c\in \Delta\}=\{1\}\cup\{s_{\alpha}\mid \alpha\in W\varphi\}.
\end{align}
\end{thm}

\begin{proof} Since $\alpha$ is a long root it is in the $W$-orbit of the highest root, and it follows from standard RGD properties (see \cite[Section~7.8]{AB:08}) that $\theta$ is conjugate to an element of $U_{\varphi}\backslash\{1\}$. Thus, after conjugation, we may assume that $\theta$ is central in $U^+$. A chamber $gB$ is opposite its image $\theta gB$ if and only if $\delta(gB,\theta gB)=w_0$, if and only if $g^{-1}\theta g\in Bw_0B$. By the Bruhat decomposition each chamber $gB$ can be written as
$
gB=uwB$ for some $u\in U^+$ and $w\in W$. Since $\theta$ is central in $U^+$ we have 
$
w^{-1}u^{-1}\theta uw=w^{-1}\theta w\in U_{w^{-1}\varphi}
$. Thus if $w^{-1}\varphi\in \Phi^+$ we have $\delta(gB,\theta gB)=1$, and if $w^{-1}\varphi\in -\Phi^+$ then $w^{-1}\theta w\in U_{w^{-1}\varphi}\subseteq Bs_{w^{-1}\varphi}B$, and so $\delta(gB,\theta gB)=s_{w^{-1}\varphi}$. Equation~(\ref{eq:attained}) follows.

Since $\theta$ is a nontrivial collineation there is some simplex mapped onto an opposite simplex. Let $J\subseteq S$ be the type of such a simplex $x$, and write $J'=S\backslash J$. Thus for each chamber $c\in\Delta$ containing $x$ we have $\delta(c,c^{\theta})\in W_{J'}w_0W_{J'}$. Since $J$, and hence also $J'$, are stable under opposition we have $W_{J'}w_0W_{J'}=w_0W_{J'}$. It follows from (\ref{eq:attained}) that there is a root $\alpha\in W\varphi$ with $s_{\alpha}=w_0w$ for some $w\in W_{J'}$. Since $W_{J'}$ is a proper parabolic subgroup of $W$, each $w\in W_{J'}$ maps the highest root $\varphi$ to a positive root, and since $w_0$ maps all positive roots to negative roots we have $s_{\alpha}\varphi\in -\Phi^+$. Since $\varphi$ is the highest root of $\Phi$ this forces $\alpha=\pm\varphi$, and so $w_0^{-1}s_{\varphi}\in W_{J'}$. For all $t\in S$ we have
$$
w_0^{-1}s_{\varphi}(\alpha_t)=-\alpha_{\pi_0(t)}+\langle\alpha_t,\varphi^{\vee}\rangle\varphi
$$
(where $\pi_0(t)=w_0tw_0^{-1}$), and since $w(\Phi_{J'})\subseteq \Phi_{J'}$ for all $w\in W_{J'}$ we deduce that $\langle\alpha_t,\varphi^{\vee}\rangle=0$ for all $t\in {J'}=S\backslash J$. It follows that $J=\wp$. 
\end{proof}

We note the following corollary.

\begin{cor}\label{cor:existsdomestic}
Every irreducible Moufang spherical building distinct from a projective plane admits a nontrivial domestic collineation.
\end{cor}

\begin{proof}
This follows immediately from Theorem~\ref{thm:RGD} and that fact that $\wp$ is a strict subset of $S$ in all cases except for $\Phi=\sA_2$ (which is the case of projective planes).
\end{proof}

\begin{remark}\label{rem:correction} Theorem~\ref{thm:RGD} implies that every Ree-Tits octagon admits nontrivial line-domestic collineations. We note that it is erroneously stated in \cite{ISX:18} that these octagons do not admit line-domestic collineations. The error appears to be as follows. If a thick generalised octagon admits a line-domestic colllineation $\theta$ then by \cite[Theorem~2.8 and Proposition~4.1]{PTM:15} the fixed element structure of~$\theta$ is either a large full suboctagon, a distance $4$-ovoid, or a ball of radius $4$ in the incidence graph centred at a point. For finite Ree-Tits octagons we proved in \cite[Proposition~4.4]{PTM:15} that large full suboctagons do not exist, and in \cite{ISX:18} it is shown that distance $4$-ovoids do not exist. Thus any line-domestic collineation of a finite Ree-Tits octagon necessarily fixes a ball of radius $4$ centred at a point (and is thus a central collineation, the example given by Theorem~\ref{thm:RGD}). It follows from \cite[Proposition~4.5]{PTM:15} that no collineation of a finite Ree-Tits octagon fixes a ball of radius $4$ centred at a \textit{line}, and we believe that this may be the source of the misunderstanding in~\cite{ISX:18} (with $s=t^2$ misread as $s^2=t$ in \cite[Proposition~4.5]{PTM:15}). 
\end{remark}

We now prove the converse to Theorem~\ref{thm:RGD} for split buildings of types $\sE_n$ and $\sF_4$. Let $\Phi$ be a root system of type $\sE_n$ for $n=6,7,8$, or of type $\sF_4$. Let $\varphi$ be the highest root. Let $i_0$ be the polar node. Let $\Phi_1$ be the polar subsystem, generated by the simple roots $\{\alpha_i\mid i\neq i_0\}$, and let $W_1$ be the parabolic subgroup generated by $\{s_i\mid i\neq i_0\}$. Let $w_0$ be the longest element, and let $w_1$ be the longest element of~$W_1$. Let $j_0$ be the unique node joined to the polar node in the Dynkin diagram. Write $\pi=\alpha_{i_0}$ and $\pi'=\alpha_{j_0}$. Explicitly, $(\pi,\pi')=(\alpha_2,\alpha_4)$, $(\alpha_1,\alpha_3)$, $(\alpha_8,\alpha_7)$, and $(\alpha_1,\alpha_2)$ for types $\sE_6$, $\sE_7$, $\sE_8$, and $\sF_4$, respectively. Let $\omega=\omega_{i_0}$ be the fundamental coweight corresponding to the polar node (thus $\langle\omega,\alpha_{i}\rangle=
\delta_{i,i_0}$). We note the following facts:
\begin{compactenum}[$(1)$]
\item $\omega=\varphi^{\vee}$ (because $\langle\varphi^{\vee},\alpha_{i}\rangle=\delta_{i,i_0}$).
\item $\varphi$ is the unique root whose coefficient of $\pi$ is $2$ (because $2=\langle\varphi,\varphi^{\vee}\rangle=\langle\varphi,\omega\rangle$). 
\item $\varphi-\alpha_i\in \Phi$ if and only if $i=i_0$ (by (2)).
\item The elements $\varphi-\pi$ and $\varphi-\pi-\pi'$ are roots, but $\varphi-\pi'$ is not a root (by (2)). 
\item $s_{\varphi}=w_1w_0$ (since both elements have inversion set $\Phi^+\backslash\Phi_1=\{\alpha\in\Phi^+\mid \langle\alpha,\omega\rangle\in\{1,2\}\}$).
\end{compactenum}

\begin{thm}\label{thm:converse}
Let $\Delta$ be a split building of type $\sE_n$ or $\sF_4$, and let $\theta$ be an automorphism of $\Delta$. If $\Type(\theta)=\wp$ then $\theta$ is a long root elation.
\end{thm}

\begin{proof}
By the classification of admissible diagrams, if $\Type(\theta)=\wp$ then $\theta$ is necessarily type preserving, and capped. Thus $\disp(\theta)=\ell(s_{\varphi})$, and since $\theta$ is capped $\ell(\delta(gB,\theta gB))=\ell(s_{\varphi})$ if and only if $\delta(gB,\theta gB)=s_{\varphi}$. Moreover, after replacing $\theta$ by a conjugate, we may assume that the base chamber $B$ is mapped to Weyl distance $s_{\varphi}=w_1w_0$. Since the stabiliser of $B$ is transitive on each $w$-sphere centred at~$B$ we may assume that $B$ is mapped to the chamber $x_{\varphi}(1)s_{\varphi}B$. By the folding relation we have $x_{\varphi}(1)s_{\varphi}B=x_{-\varphi}(1)B$. The condition $\theta(B)=x_{-\varphi}(1)B$ gives
$$
\theta=x_{-\varphi}(1)uh\sigma\quad\text{for some $u\in U^+$, $h\in H$, and $\sigma\in\mathrm{Aut}(\FF)$}. 
$$
We will now determine $u,h$ and $\sigma$. The primary strategy is to show that if these elements do not take certain particular forms, then one can find elements $g\in G$ such that $g^{-1}\theta g\in BwB\sigma$ with $\ell(w)>\ell(w_1w_0)$, which contradicts the fact that $\Type(\theta)=\wp$. A useful observation is that if $w=s_{\varphi}v$ with $v\in W_1$ then $\ell(w)=\ell(s_{\varphi})+\ell(v)$ (because $s_{\varphi}=w_1w_0=w_0w_1$). We now proceed with the analysis.
\medskip

\noindent\textit{Claim 1: We have $u\in U_{\Phi\backslash\Phi_1}^+$}. Write $u=u_1u_2$ with $u_1\in U_{\Phi_1}^+$ and $u_2\in U_{\Phi\backslash\Phi_1}^+$. Then 
$$
w_{1}^{-1}\theta w_{1}= x_{-\varphi}(\pm 1)u_1^-u_2'h'\sigma\quad\text{with $u_1^-=w_{1}^{-1}u_1w_{1}\in U_{\Phi_1}^-$, $u_2'\in U_{\Phi\backslash\Phi_1}^+$, $h'\in H$}.
$$
Since $u_1^-\in BW_{1}B$ we have $u_1^-\in BvB$ for some $v\in W_{1}$. But since $\ell(s_{\varphi}v)=\ell(s_{\varphi})+\ell(v)$ we have 
$$
w_{1}^{-1}\theta w_{1}\in Bs_{\varphi}B\cdot BvB\sigma=Bs_{\varphi}vB\sigma,
$$ 
and so $v=1$ (as $\disp(\theta)=\ell(s_{\varphi})$). So $u_1^-\in B\cap U_{\Phi_1}^-$, and so $u_1^-=1$, hence $u_1=1$. 
\medskip

\noindent\textit{Claim 2: We have $h=h_{\omega}(c)$ for some $c\in\mathbb{F}^{\times}$}. Write $h=h_{\omega_1}(c_1)\cdots h_{\omega_n}(c_n)$. Let $i\neq i_0$. Then
$$
x_{-\alpha_i}(-1)\theta x_{-\alpha_i}(1)=x_{-\varphi}(1)x_{-\alpha_i}(-1)ux_{-\alpha_i}(c_i^{-1})h\sigma=x_{-\varphi}(1)x_{-\alpha_i}(c_i^{-1}-1)u'h\sigma,
$$
with $u'\in U^+$ (here we have used the fact, from Claim 1, that $x_{\alpha_i}(a)$ does not appear as a factor in $u$). Thus, if $c_i\neq 1$ the folding relation gives
$$
x_{-\alpha_i}(-1)\theta x_{-\alpha_i}(1)\in Bs_{\varphi}B\cdot Bs_iB\sigma=Bs_{\varphi}s_iB\sigma,
$$
a contradiction as before. Hence $c_i=1$ for all $i\neq i_0$, hence the claim.
\medskip

\noindent\textit{Claim 3: We have $\sigma=\mathrm{id}$}. Suppose not. Let $i\neq i_0$ and let $a\in \mathbb{F}$ with $a^{\sigma}\neq a$. Then 
$$
h_{\alpha_i^{\vee}}(a)^{-1}\theta h_{\alpha_i^{\vee}}(a)=x_{-\varphi}(1)uh_{\omega_i}(c)h_{\alpha_i^{\vee}}(a^{\sigma}a^{-1})\sigma.
$$
Claim 2 now gives a contradiction. 
\medskip

\noindent\textit{Claim 4: We have $u\in U_{\varphi}$}. Suppose not, and write $\theta=x_{-\varphi}(1)uh_{\omega}(c)$ with
$
u=x_{\beta_k}(a_k)\cdots x_{\beta_1}(a_1)
$
in decreasing root height. By assumption, $\beta_1\neq\varphi$. If $\beta_1\neq \pi$ (the polar simple root) then there exists $\alpha\in \Phi_1^+$ with $\beta_1-\alpha\in\Phi^+$. Then
$$
x_{-\alpha}(-b)ux_{-\alpha}(b)=u'x_{\beta_1-\alpha}(\pm a_1b),
$$
with $u'$ a product of roots in $\Phi^+\backslash\Phi_1$ of height at least $\mathrm{ht}(\beta_1-\alpha)$. Continuing  in this way, there exists an element $g\in U_{\Phi_1}^-$ with 
$$
g^{-1}ug=u'x_{\pi+\pi'}(b)x_{\pi}(a)\quad\text{with $a\neq 0$ and $b\in \FF$},
$$
where $u'$ is a product of elements $x_{\beta}(\cdot)$ with $\beta\in\Phi^+\backslash(\Phi_1\cup\{\pi,\pi+\pi'\})$. We have $g^{-1}x_{-\varphi}(1)g=x_{-\varphi}(1)$ (as $-\varphi+\alpha\notin\Phi$ for all $\alpha\in \Phi_1$), and $h_{\omega}(c)gh_{\omega}(c)^{-1}=g$ (as $\langle\omega,\alpha\rangle=0$ for all $\alpha\in\Phi_1$), and hence
\begin{align*}
g^{-1}\theta g&=g^{-1}x_{-\varphi}(1)uh_{\omega}(c)g=x_{-\varphi}(1)g^{-1}ugh_{\omega}(c)=x_{-\varphi}(1)u'x_{\pi+\pi'}(b)x_{\pi}(a)h_{\omega}(c).
\end{align*}

Let $d\in \FF$ with $d\neq 0$, and write $g_1=x_{-\pi-\pi'}(d)$. Then
\begin{align*}
Bg_1^{-1}g^{-1}\theta gg_1B&=Bx_{-\pi-\pi'}(-d)x_{-\varphi}(1)u'x_{\pi+\pi'}(b)x_{\pi}(a)h_{\omega}(c)x_{-\pi-\pi'}(d)B\\
&= Bx_{-\varphi}(1)x_{-\pi-\pi'}(-d)u'x_{\pi+\pi'}(b)x_{\pi}(a)x_{-\pi-\pi'}(dc^{-1})B\\
&= Bx_{-\varphi}(1)x_{-\pi-\pi'}(-d)u'x_{\pi+\pi'}(b)x_{-\pi-\pi'}(dc^{-1})x_{-\pi'}(\pm adc^{-1})B,
\end{align*}
where we have used the commutator relation $x_{\pi}(a)x_{-\pi-\pi'}(dc^{-1})=x_{-\pi-\pi'}(dc^{-1})x_{-\pi'}(\pm adc^{-1})x_{\pi}(a)$. Note that
\begin{align*}
u'x_{\pi+\pi'}(b)x_{-\pi-\pi'}(dc^{-1})=x_{\pi+\pi'}(b)x_{-\pi-\pi'}(dc^{-1})u''\quad\text{where $u''\in U^+$}
\end{align*}
(this follows from the fact that $u'$ is a product of elements $x_{\beta}(\cdot)$ with $\beta\in\Phi^+\backslash(\Phi_1\cup\{\pi,\pi+\pi'\})$, and for such $\beta$, if $\beta-\pi-\pi'\in\Phi$ then $\beta-\pi-\pi'\in \Phi^+$). Therefore
\begin{align*}
Bg_1^{-1}g^{-1}\theta gg_1B&=Bx_{-\varphi}(1)x_{-\pi-\pi'}(-d)x_{\pi+\pi'}(b)x_{-\pi-\pi'}(dc^{-1})u''x_{-\pi'}(\pm adc^{-1})B\\
&=Bs_{\varphi}x_{\varphi}(1)x_{-\pi-\pi'}(-d)x_{\pi+\pi'}(b)x_{-\pi-\pi'}(dc^{-1})u''x_{-\pi'}(\pm adc^{-1})B.
\end{align*}
From the commutator relations we have
$$
x_{\varphi}(1)x_{-\pi-\pi'}(-d)x_{\pi+\pi'}(b)x_{-\pi-\pi'}(dc^{-1})=x_{-\pi-\pi'}(-d)x_{\pi+\pi'}(b)x_{-\pi-\pi'}(dc^{-1})x_{\varphi-\pi-\pi'}(a')x_{\varphi}(b')
$$
for some $a',b'\in\FF$, and hence
\begin{align*}
Bg_1^{-1}g^{-1}\theta gg_1B&= Bs_{\varphi}x_{-\pi-\pi'}(-d)x_{\pi+\pi'}(b)x_{-\pi-\pi'}(dc^{-1})u'''x_{-\pi'}(\pm adc^{-1})B
\end{align*}
for some $u'''\in U^+$. 

There are now two cases to consider. If $b=0$ then, since $s_{\varphi}(-\pi-\pi')\in\Phi^+$, we have
\begin{align*}
Bg_1^{-1}g^{-1}\theta gg_1B&= Bs_{\varphi}x_{-\pi-\pi'}(-d+dc^{-1})u'''x_{-\pi'}(\pm adc^{-1})B\\
&=Bs_{\varphi}u'''x_{-\pi'}(\pm adc^{-1})B\\
&\subseteq Bs_{\varphi}B\cdot Bs_{\pi'}B\\
&=Bs_{\varphi}s_{\pi'}B,
\end{align*}
a contradiction. If $b\neq 0$ then, again using $s_{\varphi}(-\pi-\pi')\in\Phi^+$, we have $Bs_{\varphi}x_{-\pi-\pi'}(-d)=Bs_{\varphi}x_{-\pi-\pi'}(dc^{-1})$, and so
\begin{align*}
Bg_1^{-1}g^{-1}\theta gg_1B&= Bs_{\varphi}x_{-\pi-\pi'}(dc^{-1})x_{\pi+\pi'}(b)x_{-\pi-\pi'}(dc^{-1})u'''x_{-\pi'}(\pm adc^{-1})B.
\end{align*}
Choosing $d=-b^{-1}c$ we have $Bs_{\varphi}x_{-\pi-\pi'}(dc^{-1})x_{\pi+\pi'}(b)x_{-\pi-\pi'}(dc^{-1})=Bs_{\varphi}s_{\pi+\pi'}$, and hence
\begin{align*}
Bg_1^{-1}g^{-1}\theta gg_1B&= Bs_{\varphi}s_{\pi+\pi'}u'''x_{-\pi'}(\pm adc^{-1})B\subseteq Bs_{\varphi}s_{\pi+\pi'}B\cdot Bs_{\pi'}B.
\end{align*}
We have $\Phi(s_{\varphi}s_{\pi+\pi'})=(\Phi(s_{\varphi})\backslash\{\varphi,\varphi-\pi\})\cup \{\pi'\}$, and it follows that $\ell(s_{\varphi}s_{\pi+\pi'})=\ell(s_{\varphi})-1$ and $\ell(s_{\varphi}s_{\pi+\pi'}s_{\pi'})=\ell(s_{\varphi})$. Thus
\begin{align*}
Bg_1^{-1}g^{-1}\theta gg_1B&=Bs_{\varphi}s_{\pi+\pi'}s_{\pi'}B.
\end{align*}
Thus the chamber $gg_1B$ is mapped to distance~$\ell(s_{\varphi})$ by $\theta$, contradicting the second sentence of the proof (as $s_{\varphi}s_{\pi+\pi'}s_{\pi'}\neq s_{\varphi}$). This completes the proof of the claim.
\medskip

\noindent\textit{Claim 5: We have $\theta=x_{-\varphi}(1)x_{\varphi}(-(c-1)^2)h_{\omega}(c)$ for some $c\in\mathbb{F}^{\times}$}. From Claims 1--4 we have 
$$
\theta=x_{-\varphi}(1)x_{\varphi}(a)h_{\omega}(c)\quad\text{for some $a\in \FF$ and $c\in\FF^{\times}$}.
$$
We will now make a careful commutator relation calculation, using the consistent sign conventions:
\begin{align*}
x_{\varphi}(a)x_{-\varphi+\pi}(b)&=x_{-\varphi+\pi}(b)x_{\varphi}(a)x_{\pi}(ab)\\
x_{\pi}(a)x_{-\pi-\pi'}(b)&=x_{-\pi-\pi'}(b)x_{\pi}(a)x_{-\pi'}(-ab)\\
x_{\varphi}(a)x_{-\pi-\pi'}(b)&=x_{-\pi-\pi'}(b)x_{\varphi}(a)x_{\varphi-\pi-\pi'}(-ab)\\
x_{\varphi-\pi-\pi'}(a)x_{-\varphi+\pi}(b)&=x_{-\varphi+\pi}(b)x_{\varphi-\pi-\pi'}(a)x_{-\pi'}(ab).
\end{align*}
Let $g=x_{-\varphi+\pi}(1)x_{-\pi-\pi'}(1)$. Then
\begin{align*}
Bg^{-1}\theta gB&=Bx_{-\pi-\pi'}(-1)x_{-\varphi+\pi}(-1)x_{-\varphi}(1)x_{\varphi}(a)h_{\omega}(c)x_{-\varphi+\pi}(1)x_{-\pi-\pi'}(1)B\\
&=Bx_{-\varphi}(1)x_{-\pi-\pi'}(-1)x_{-\varphi+\pi}(-1)x_{\varphi}(a)x_{-\varphi+\pi}(c^{-1})x_{-\pi-\pi'}(c^{-1})B\\
&= Bs_{\varphi}x_{\varphi}(1)x_{-\pi-\pi'}(-1)x_{-\varphi+\pi}(c^{-1}-1)x_{\varphi}(a)x_{\pi}(ac^{-1})x_{-\pi-\pi'}(c^{-1})B.
\end{align*}
We have
\begin{align*}
x_{\varphi}(a)x_{\pi}(ac^{-1})x_{-\pi-\pi'}(c^{-1})B&=x_{\varphi}(a)x_{-\pi-\pi'}(c^{-1})x_{-\pi'}(-ac^{-2})B\\
&=x_{-\pi-\pi'}(c^{-1})x_{\varphi}(a)x_{\varphi-\pi-\pi'}(-ac^{-1})x_{-\pi'}(-ac^{-2})B\\
&=x_{-\pi-\pi'}(c^{-1})x_{\varphi}(a)x_{-\pi'}(-ac^{-2})B\\
&=x_{-\pi-\pi'}(c^{-1})x_{-\pi'}(-ac^{-2})B,
\end{align*}
aan hence
\begin{align*}
Bg^{-1}\theta gB&= Bs_{\varphi}x_{\varphi}(1)x_{-\pi-\pi'}(-1)x_{-\varphi+\pi}(c^{-1}-1)x_{-\pi-\pi'}(c^{-1})x_{-\pi'}(-ac^{-2})B\\
&= Bs_{\varphi}x_{\varphi}(1)x_{-\pi-\pi'}(c^{-1}-1)x_{-\varphi+\pi}(c^{-1}-1)x_{-\pi'}(-ac^{-2})B\\
&= Bs_{\varphi}x_{-\pi-\pi'}(c^{-1}-1)x_{\varphi}(1)x_{\varphi-\pi-\pi'}(1-c^{-1})x_{-\varphi+\pi}(c^{-1}-1)x_{-\pi'}(-ac^{-2})B\\
&= Bs_{\varphi}x_{\varphi}(1)x_{-\varphi+\pi}(c^{-1}-1)x_{\varphi-\pi-\pi'}(1-c^{-1})x_{-\pi'}(-(c^{-1}-1)^2)x_{-\pi'}(-ac^{-2})B\\
&= Bs_{\varphi}x_{\varphi}(1)x_{-\varphi+\pi}(c^{-1}-1)x_{-\pi'}(-ac^{-2}-(c^{-1}-1)^2)B\\
&= Bs_{\varphi}x_{-\varphi+\pi}(c^{-1}-1)x_{\varphi}(1)x_{\pi}(c^{-1}-1)x_{-\pi'}(-ac^{-2}-(c^{-1}-1)^2)B\\
&= Bs_{\varphi}x_{\varphi}(1)x_{\pi}(c^{-1}-1)x_{-\pi'}(-ac^{-2}-(c^{-1}-1)^2)B.
\end{align*}
Thus, if $ac^{-2}+(c^{-1}-1)^2\neq 0$ we have
$$
Bg^{-1}\theta gB= Bs_{\varphi}B\cdot Bs_{\pi'}B=Bs_{\varphi}s_{\pi'}B,
$$
a contradiction. Thus $a=-(c-1)^2$, completing the proof of the claim.
\medskip

\noindent\textit{Claim 6: $\theta$ is a long root elation.} Let $g=x_{\varphi}(-c)x_{-\varphi}(1)s_{\varphi}$. Then, from Claim 5, we have
\begin{align*}
g\theta g^{-1}&=x_{\varphi}(-c)x_{-\varphi}(1)s_{\varphi}x_{-\varphi}(1)x_{\varphi}(-(c-1)^2)h_{\omega}(c)s_{\varphi}^{-1}x_{-\varphi}(-1)x_{\varphi}(c).
\end{align*}
Using the relations $s_{\varphi}x_{\pm\varphi}(a)s_{\varphi}^{-1}=x_{\mp\varphi}(-a)$ and $s_{\varphi}h_{\omega}(c)s_{\varphi}^{-1}=h_{\omega}(c^{-1})$ we have
\begin{align*}
g\theta g^{-1}&=x_{\varphi}(-c)x_{-\varphi}(1)x_{\varphi}(-1)x_{-\varphi}((c-1)^2)h_{\omega}(c^{-1})x_{-\varphi}(-1)x_{\varphi}(c)\\
&=x_{\varphi}(-c)x_{-\varphi}(1)x_{\varphi}(-1)x_{-\varphi}((c-1)^2)x_{-\varphi}(-c^2)x_{\varphi}(c^{-1})h_{\omega}(c^{-1})\\
&=x_{\varphi}(-c)[x_{-\varphi}(1)x_{\varphi}(-1)x_{-\varphi}(1)]x_{-\varphi}(-2c)x_{\varphi}(c^{-1})h_{\omega}(c^{-1}).
\end{align*}
Now, $x_{-\varphi}(1)x_{\varphi}(-1)x_{-\varphi}(1)=s_{\varphi}(-1)=h_{\omega}(-1)s_{\varphi}$, and so
\begin{align*}
g\theta g^{-1}&=x_{\varphi}(-c)h_{\omega}(-1)s_{\varphi}x_{-\varphi}(-2c)x_{\varphi}(c^{-1})h_{\omega}(c^{-1})\\
&=x_{\varphi}(-c)h_{\omega}(-1)x_{\varphi}(2c)x_{-\varphi}(-c^{-1})s_{\varphi}h_{\omega}(c^{-1})\\
&=x_{\varphi}(c)x_{-\varphi}(-c^{-1})s_{\varphi}h_{\omega}(-c^{-1})
\end{align*}
We have $x_{\varphi}(c)x_{-\varphi}(-c^{-1})=s_{\varphi}(c)x_{\varphi}(-c)=h_{\omega}(c)s_{\varphi}x_{\varphi}(-c)=x_{-\varphi}(c^{-1})h_{\omega}(c)s_{\varphi}$, and since $s_{\varphi}^2=h_{\omega}(-1)$ we have
\begin{align*}
g\theta g^{-1}&=x_{-\varphi}(c^{-1})h_{\omega}(c)h_{\omega}(-1)h_{\omega}(-c^{-1})=x_{-\varphi}(c^{-1}).
\end{align*}
Thus $\theta$ is conjugate to the long root elation $x_{-\varphi}(c^{-1})$, completing the proof.
\end{proof}

\subsection{A geometric characterisation of root elations for $\sE_6$ and $\sE_7$ buildings}

Theorems~\ref{thm:RGD} and~\ref{thm:converse} imply an interesting geometric characterisation of root elations for types $\sE_6$ and $\sE_7$ (see Theorem~\ref{thm:geomchar} below). In the following lemma, and again in the following subsection, we make use of \cite[\S10.3 Lemma~B]{Hum:72}, which says that if $\la\in P$ is dominant (that is, $\lambda\in \ZZ_{\geq 0}\omega_1+\cdots+\ZZ_{\geq 0}\omega_n$), and if $w\in W$ with $w\lambda=\lambda$, then $w\in W_J$ where $J=\{s\in S\mid s\lambda=\lambda\}$.

\begin{lemma}\label{lem:skipping}
Let $i=1$ if $\Phi=\sE_6$ and $i=7$ if $\Phi=\sE_7$. Then $s_{\alpha}\in W_i\cup W_is_iW_i$ for all $\alpha\in \Phi$, where $W_i$ denotes the parabolic subgroup of the Weyl group generated by $S\backslash\{s_i\}$. 
\end{lemma}

\begin{proof}
Consider the $\Phi=\sE_7$ case. Since $s_{-\alpha}=s_{\alpha}$ we may assume that $\alpha\in\Phi^+$. Let $\Phi_7=\{\alpha\in\Phi\mid \langle\alpha,\omega_7\rangle=0\}$ be the $\sE_6$ subsystem. If $\alpha\in \Phi_7^+$ then, since $W_7$ is transitive on $\Phi_7$, we have $\alpha=w\alpha_1$ for some $w\in W_7$, and hence $s_{\alpha}=ws_1w^{-1}\in W_7$. If $\alpha\in \Phi^+\backslash\Phi_7$ then we claim that $\alpha\in W_7\cdot \alpha_7$, from which it follows that $s_{\alpha}\in W_7s_7W_7$. To see this, note that $\alpha_7=-\omega_6+2\omega_7$, and hence for $w\in W_7$ we have $w\alpha_7=\alpha_7$ if and only if $w\omega_6=\omega_6$ (as $w\omega_7=\omega_7$). Since $\omega_6$ is dominant (in the space of $\sE_6$ coweights) it follows from \cite[\S10.3 Lemma~B]{Hum:72} that $w\in W_{\sD_5}$ (the subgroup of $W_7$ generated by $s_1,\ldots,s_5$). Thus the stabiliser of $\alpha_7$ in $W_7$ is $W_{\sD_5}$, and so by counting $|W_7\cdot\alpha_7|=|W_7|/|W_{\sD_5}|=27$. Clearly each root $w\alpha_7$ with $w\in W_7$ is in $\Phi^+\backslash\Phi_7$ (as the coefficient of $\alpha_7$ is~$1$), and since $|\Phi^+\backslash \Phi_7|=63-36=27$ we conclude that $W_7$ is transitive on $\Phi^+\backslash\Phi_7$, and hence the result. 

The argument for the $\sE_6$ case is similar. 
\end{proof}

\begin{thm}\label{thm:geomchar}
Let $\theta$ be a type preserving automorphism of a thick building $\Delta$. If $\Delta$ has type $\sE_6$ (respectively $\sE_7$) then $\theta$ is a root elation if and only if each point of the Lie incidence geometry $\sE_{6,1}(\FF)$ (respectively $\sE_{7,7}(\FF)$) is either fixed or mapped to a collinear point by~$\theta$. 
\end{thm}

\begin{proof}
If $\theta$ is a root elation, then by Theorem~\ref{thm:RGD} we have that $\delta(c,c^{\theta})$ is a reflection (or the identity) for all chambers~$c\in\Delta$. It follows from Lemma~\ref{lem:skipping} that $\delta(c,c^{\theta})\in W_i\cup W_is_iW_i$ (with $i=1$ in the $\sE_6$ case and $i=7$ in the $\sE_7$ case). In geometric terms, this says that points of the geometries $\sE_{6,1}(\FF)$ and $\sE_{7,7}(\FF)$ are either fixed, or are mapped to collinear points (see Lemma~\ref{paraargument} for another proof, applying to geometries including the $\sE_{7,7}(\FF)$ case). 

To prove the converse for $\sE_7$, note that if each point of $\sE_{7,7}(\FF)$ is either fixed or mapped to a collinear point, then no line of the $\sE_{7,7}(\FF)$ geometry is mapped to an opposite line. Thus $\theta$ is $\{6\}$-domestic, and from the classification of admissible diagrams this forces $\theta$ to have the polar diagram. Thus $\theta$ is a root elation by Theorem~\ref{thm:converse}.

We now prove the converse for $\sE_6$. If nontrivial $\theta$ is not a root elation, then $\theta$ does not have polar diagram (by Theorem~\ref{thm:converse}), and hence by the classification of admissible diagrams $\theta$ maps a (point,symp)-pair  $(p,\xi)$ of $\sE_{6,1}(\FF)$ to an opposite (here points are type $1$ vertices, and symps are type~$6$ vertices). Then, since no point of $\xi^\theta$ is collinear to $p$, the point $p^\theta$ is at distance $2$ from~$p$, completing the proof.
\end{proof}

\subsection{Distances attained by long root elations}\label{sec:exists}

Here we prove Corollaries~\ref{cor:existsdomestic} and~\ref{cor:existsconjugacy}. 
Let $\Delta$ be an irreducible Moufang spherical building other than a generalised octagon, and recall (as in Subsection~\ref{sec:lre}) that one may associate a crystallographic root system $\Phi$ to $\Delta$. Let $\Phi_r$ denote the associated reduced root system (thus $\Phi_r=\Phi$ if $\Phi$ is reduced, and $\Phi_r$ is the $\sC_n$ subsystem consisting of the middle and long length roots in the non-reduced $\mathsf{BC}_n$ case). Consider the long root geometry~$\mathscr{G}$. Let $P=P_{S\backslash \wp}$ be the standard parabolic subgroup of $G$ of type $S\backslash\wp$, and let $W'=W_{S\backslash\wp}$. The points of $\mathscr{G}$ are the cosets in $G/P$, and we have
$
G=\bigsqcup_{w\in R(\wp)}PwP,
$
where $R(\wp)$ is the set of minimal length double coset representatives for $W'\backslash W/W'$. The \textit{Weyl-distance} $\delta(g_1P,g_2P)$ between points $g_1P$ and $g_2P$ is defined to be the unique element $\delta(g_1P,g_2P)=w\in R(\wp)$ with $g_1^{-1}g_2\in PwP$. Points $g_1P$ and $g_2P$ are (i) \textit{collinear} if $\delta(g_1P,g_2P)=s$ for some $s\in \wp$ (thus in type $\sA$ there are two ``flavours'' of collinearity), and (ii) \textit{opposite} if $\delta(g_1P,g_2P)=w_{S\backslash\wp}w_0$. Note that $w_{S\backslash \wp}w_0$ is the minimal length representative of $W'w_0W'$, and that $w_{S\backslash \wp}w_0=s_{\varphi}$ (by comparing inversion sets).

\begin{thm}\label{thm:skippeddistances} Let $\theta$ be a long root elation of a Moufang spherical building~$\Delta$.
\begin{compactenum}[$(1)$]
\item Suppose that $\theta$ is not a generalised octagon. Let~$\Phi_r$ be the reduced root system of~$\Delta$, and let $\mathscr{G}$ be the long root geometry. 
\begin{compactenum}[$(a)$]
\item If $\Phi_r=\sC_n$ with $n\geq 2$, or if $\Phi_r=\sB_2$, then every point of $\mathscr{G}$ is either fixed, or is mapped onto an opposite point by~$\theta$. 
\item In all other cases, every point of $\mathscr{G}$ is either fixed, mapped onto a collinear point, or mapped onto an opposite point by~$\theta$. 
\end{compactenum}
\item Suppose that $\Delta$ is a Ree-Tits octagon. Then every point of $\Delta$ is mapped by $\theta$ onto a point at distance $0$, $4$, or $8$ in the incidence graph. 
\end{compactenum}
In particular, for each type there exists at least one element $w\in R(\wp)$ such that no point is mapped onto a point at distance $w$ by $\theta$.
\end{thm}

\begin{proof}
Let $W'=W_{S\backslash\wp}$, and let $D(\theta)=\{\delta(gP,\theta gP) \mid gP\in G/P\}\subseteq R(\wp)$ be the set of distances realised by $\theta$. From Theorem~\ref{thm:RGD} we see that $D(\theta)$ consists precisely of the identity, along with the minimal length representatives of the double cosets $W's_{\alpha}W'$ with $\alpha$ a long root.

Consider the $\sA_n$ case. We claim that
\begin{align}\label{eq:decomposition}
\Phi^+=\Phi_{S\backslash\wp}^+\sqcup(W'\cdot\alpha_1)\sqcup (W'\cdot\alpha_n)\sqcup\{\varphi\}.
\end{align}
The result follows from this claim, because if $\alpha\in \Phi_{S\backslash\wp}^+$ then $s_{\alpha}\in W'$, if $\alpha\in W'\cdot\alpha_1$ then $s_{\alpha}=ws_1w^{-1}$ for some $w\in W'$ and so $W's_{\alpha}W'=W's_1W'$, if $\alpha\in W'\cdot \alpha_n$ then $W's_{\alpha}W'=W's_nW'$, and if $\alpha=\varphi$ then $W's_{\alpha}W'=W'w_{S\backslash\wp}w_0W'$. 

To prove~(\ref{eq:decomposition}), we first claim that
\begin{align*}
W'\cdot\alpha_1=\{\alpha\in\Phi^+\mid \langle\alpha,\omega_1\rangle=1\text{ and }\langle\alpha,\omega_n\rangle=0\}.
\end{align*}
Denote the right hand side by~$X$. If $\alpha=w\alpha_1$ with $w\in W'$ then $\langle\alpha,\omega_1\rangle=\langle\alpha_1,w^{-1}\omega_1\rangle=\langle\alpha_1,\omega_1\rangle=1$, because $w^{-1}\omega_1=\omega_1$ for all $w\in W'$. It follows that $\alpha\in \Phi^+$, and similarly we have $\langle\alpha,\omega_n\rangle=0$, and so $W'\cdot\alpha_1\subseteq X$. By the orbit-stabiliser theorem we have
$$
|W'\cdot\alpha_1|=|W'|/|\mathrm{stab}_{W'}(\alpha_1)|=|W_{\sA_{n-2}}|/|W_{\sA_{n-3}}|=n-1
$$
where the stabiliser computation follows from the fact that if $w\in W'$ then $w\alpha_1=\alpha_1$ if and only if $w\omega_2=\omega_2$ (as $\alpha_1=2\omega_1-\omega_2$ and $w\omega_1=\omega_1$), if and only if $w\in W_{S\backslash \{1,2,n\}}$ (by \cite[\S10.3 Lemma~B]{Hum:72}). Since $|X|=n-1$ we have $W'\cdot\alpha_1=X$, and hence the claim. 

Dually we have $W'\cdot\alpha_n=\{\alpha\in\Phi^+\mid \langle\alpha,\omega_1\rangle=0\text{ and }\langle\alpha,\omega_n\rangle=1\}$. Since every positive root $\alpha$ either has $\langle\alpha,\omega_1\rangle=\langle\alpha,\omega_n\rangle=0$ (in which case $\alpha\in \Phi_{S\backslash\wp}^+$), or $\langle\alpha,\omega_1\rangle=1$ and $\langle \alpha,\omega_n\rangle=0$ (in which case $\alpha\in W'\cdot\alpha_1$), or the dual situation (with $\alpha\in W'\cdot\alpha_n$), or $\langle\alpha,\omega_1\rangle=\langle\alpha,\omega_n\rangle=1$ (in which case $\alpha=\varphi$) the claim~(\ref{eq:decomposition}) follows.

Consider the $\sB_n$ case with $n\geq 3$, and let $\Phi_L$ be the set of long roots. Let $Y=\{\alpha\in\Phi_L^+\mid \langle\alpha,\omega_2\rangle=0\}$ and $X=\Phi_L^+\backslash(Y\cup\{\varphi\})$. Thus $\Phi_L^+=X\sqcup Y\sqcup \{\varphi\}$. If $\alpha\in Y$ then $W's_{\alpha}W'=W'$. By inspection of the root system we have
$$
X=\{\alpha\in\Phi_L^+\mid \langle\alpha,\omega_2\rangle=1\text{ and }\langle\alpha,\omega_n\rangle\in\{0,2\}\}.
$$
From this description it is clear that $W'\cdot\alpha_2\subseteq X$, and a similar orbit-stabiliser calculation as in the $\sA_n$ case gives $X=W'\cdot\alpha_2$. Hence the result in this case.

The $\sB_2$ and $\sC_n$ cases are immediate, as the polar node corresponds to a short root in these cases, and the $\sD_n$ case is very similar to the $\sB_n$ case. 

Consider the cases $\sE_n$ ($n=6,7,8$) and $\sF_4$. Let $\wp=\{p\}$, $Y=\{\alpha\in \Phi_L^+\mid \langle\alpha,\omega_p\rangle=0\}$, and $X=\Phi_L^+\backslash (Y\cup\{\varphi\})$. Then $s_{\alpha}\in W'$ for all $\alpha\in Y$, and by inspection of the root systems we have
$
X=\{\alpha\in\Phi_L^+\mid \langle\alpha,\omega_p\rangle=1\},
$
from which it follows that $W'\cdot\alpha_p\subseteq X$. Then
$
|W'\cdot \alpha_p|=|W'|/|\mathrm{stab}_{W'}(\alpha_p)|,
$
and we compute $|\mathrm{stab}_{W'}(\alpha_p)|=|W_{\sA_2\times\sA_2}|$, $|W_{\sA_5}|$, $|W_{\sE_6}|$, $|W_{\sA_2}|$ in the cases $\sE_6$, $\sE_7$, $\sE_8$, $\sF_4$ (respectively). For example, in the $\sE_6$ case if $w\in W'$ then one has $w\alpha_2=\alpha_2$ if and only if $w\omega_4=\omega_4$ (as $\alpha_2=2\omega_2-\omega_4$ and $w\omega_2=\omega_2$ for all $w\in W'$), if and only if $w\in W_{S\backslash\{2,4\}}$. Thus $|W'\cdot\alpha_p|=20,32,56,8$ in the cases $\sE_6$, $\sE_7$, $\sE_8$, $\sF_4$, and thus $W'\cdot\alpha_p=X$ in all cases. 

The cases $\mathsf{G}_2$ and Ree-Tits octagons are elementary from the geometry of these generalised polygons, because the fixed elements of $\theta$ form a ball centred at a point with radius $3$ (for $\mathsf{G}_2$) or $4$ (for octagons) in the incidence graph.

Finally, we note that in all cases $|D(\theta)|<|R(\wp)|$, and so there exists at least one $w\in R(\wp)$ such that no point is mapped onto distance $w$ by~$\theta$. For example, in type $\sA_n$ we have $|R(\wp)|=7$ and $|D(\theta)|=4$, and in the cases $\sE_n$ and $\sF_4$ we have $|R(\wp)|=5$ and $|D(\theta)|=3$. 
\end{proof}

The following corollary stems from a question asked to us by Barbara Baumeister.

\begin{cor}\label{cor:barbara} Let $G$ be the group of type preserving automorphisms of a Moufang spherical building $\Delta$ of type other than~$\sA_n$. There exists a nontrivial conjugacy class $\mathscr{C}$ in $G$ which is not transitive on any vertex type.
\end{cor}

\begin{proof}
Let $\mathscr{C}$ be the conjugacy class of long root elations. Consider first the case that opposition is type preserving, and let $i$ be a vertex type. Let $x$ be a type $i$ vertex. If $i$ is not the polar node then from Theorem~\ref{thm:RGD} no element of $\mathscr{C}$ maps $x$ to an opposite vertex, and hence $\mathscr{C}$ is not transitive on type $i$ vertices. If $i$ is the polar node then by Theorem~\ref{thm:skippeddistances} there is a distance in the long root geometry such that no element of $\mathscr{C}$ maps a point of the long root geometry to this distance, and hence the result in this case. 

Now suppose that opposition is not type preserving. Thus $\Delta$ is of type $\sD_{2n+1}$ or $\sE_6$. Consider the $\sD_{2n+1}$ case. By Theorem~\ref{thm:RGD} no vertex of type $1,3,4,\ldots,2n-1$ is mapped to an opposite vertex, and so $\mathscr{C}$ is not transitive on these vertex types, and by Theorem~\ref{thm:skippeddistances} $\mathscr{C}$ is not transitive on the vertices of polar type~$2$. It is easy to see, as in Theorem~\ref{thm:skippeddistances}, that in the $\sD_{2n+1,2n+1}(\FF)$ geometry points are either fixed or mapped to collinear points by long root elations, and similarly in the $\sD_{2n+1,2n}(\FF)$ geometry. Thus $\mathscr{C}$ is not transitive on the vertices of types $2n$ or $2n+1$ either. 

Consider the $\sE_6$ case. As above, $\mathscr{C}$ is not transitive on the vertices of types $2$ or $4$. Moreover, since points of the $\sE_{6,1}(\FF)$ geometry are either fixed or mapped to collinear points by long root elations (see Theorem~\ref{thm:geomchar}) we see that $\mathscr{C}$ is not transitive on vertices of type $1$, or dually type $6$. Similar calculations show that in the $\sE_{6,3}(\FF)$ geometry, a long root elation either fixes points, maps them to collinear points, or maps them to distance $s_{\varphi_{\sD_5}}$ (with the $\sD_5$ system generated by $\alpha_1,\ldots,\alpha_5$). Thus $\mathscr{C}$ is not transitive on any vertex type.
\end{proof}

\begin{remark}
In the $\sA_n$ case the class $\mathscr{C}$ of long root elations is transitive on vertices of types $1$ and~$n$, and is not transitive on any other vertex types.
\end{remark}

\subsection{Short root elations}

We now record the situation for short root elations of split buildings. In this case there is some dependence on the characteristic of the underlying field. The proof for the $\sF_4$ case is postponed to Section~\ref{sec:polarcopolar}. For $i\leq n$ let $\sB_{n;i}^1$ (respectively $\sC_{n;i}^1$) denote the admissible $\sB_n$ (respectively $\sC_n$) diagram $(\Gamma,\{1,\ldots,i\},\mathrm{id})$.

\begin{thm}\label{thm:short}
Let $\theta\in U_{\alpha}\backslash\{1\}$ for some short root $\alpha$. 
\begin{compactenum}[$(1)$]
\item If $\Phi=\sB_n$ then $\theta$ has opposition diagram $\sB_{n;2}^1$ if $\mathrm{char}(\FF)\neq 2$, and $\sB_{n;1}^1$ if $\mathrm{char}(\FF)=2$. 
\item If $\Phi=\sC_n$ then $\theta$ has opposition diagram $\sC_{n;2}^1$ if $\mathrm{char}(\FF)\neq 2$, and $\sC_{n;1}^2$ if $\mathrm{char}(\FF)=2$.
\item If $\Phi=\sF_4$ then $\theta$ has opposition diagram $\sF_{4;2}$ if $\mathrm{char}(\FF)\neq 2$, and $\sF_{4;1}^4$ if $\mathrm{char}(\FF)=2$.
\item If $\Phi=\sG_2$ then $\theta$ has opposition diagram $\sG_{2;2}$ if $\mathrm{char}(\FF)\neq 3$, and $\sG_{2;1}^1$ if $\mathrm{char}(\FF)=3$. 
\end{compactenum}
In particular, with the exception of the cases $\Phi=\sB_2$ and $\Phi=\sC_2$ with $\mathrm{char}(\FF)\neq 2$, and $\Phi=\sG_2$ with $\mathrm{char}(\FF)\neq 3$, the collineation $\theta$ is domestic. 
\end{thm}

\begin{proof}
The statements for the polar spaces $\sB_n$ and $\sC_n$ are easily proved using the matrix descriptions of these groups, and we omit the details.

Consider the case $\Phi=\sF_4$. If $\mathrm{char}(\FF)=2$ then the $\sF_{4,4}(\FF)$ geometry isometrically embeds into the $\sF_{4,1}(\FF)$ geometry (surjectively if $\FF$ is perfect), with short root elations becoming long root elations, and so Theorem~\ref{thm:RGD} implies that the opposition diagram of $\theta$ is $\sF_{4;1}^4$. The proof for the case $\mathrm{char}(\FF)\neq 2$ is postponed until Corollary~\ref{cor:F4longandshort}.

Consider the case $\Phi=\sG_2$. If $\mathrm{char}(\FF)=3$ then, as in the $\sF_4$ case, we have opposition diagram~$\sG_{2,1}^1$. Thus suppose that $\mathrm{char}(\FF)\neq 3$. A direct calculation shows that
$$
Bw_0^{-1}x_{\alpha_1+\alpha_2}(1)^{-1}x_{\varphi'}(a)x_{\alpha_1+\alpha_2}(1)w_0B=Bw_0B,
$$
and so $\theta$ is not domestic, completing the proof. 
\end{proof}

One can also show that the class $\mathscr{C}$ of short root elations in split type $\sF_4$ gives another class not transitive on any vertex type (cf. Corollary~\ref{cor:barbara}). 

\begin{cor}\label{cor:barbara2}
The class $\mathscr{C}$ of short root elations of $\sF_4(\FF)$ does not act transitively on the set of type~$i$ vertices, for each $i=1,2,3,4$. 
\end{cor}

\begin{proof} If $\mathrm{char}(\FF)=2$ then the $\sF_{4,4}(\FF)$ geometry isomoetrically embeds into the $\sF_{4,1}(\FF)$ geometry, with short root elations becoming long root elations, and hence the result.  Suppose that $\mathrm{char}(\FF)\neq 2$. As in the proof of Corollary~\ref{cor:barbara}, it is sufficient to show that for each $i=1,2,3,4$ there exists at least one distance such that no point of the geometry $\sF_{4,i}(\FF)$ is mapped to this distance by a short root elation~$\theta$. By Theorem~\ref{thm:short} no vertices of types $2$ or $3$ are mapped onto opposite vertices, and so it remains to consider vertices of types $1$ and $4$. 

After conjugating, we may assume that $\theta=x_{\varphi'}(1)$, where $\varphi'$ is the highest short root. As in Subsection~\ref{sec:parapolar}, let $M_1$ denote the set of minimal length coset representatives of cosets in $W/W_1$ (recall the notation~(\ref{eq:notationparabolic})), and let $R_1$ denote the set of minimal length representatives for the double cosets in $W_1\backslash W/W_1$. Each vertex of type~$1$ is of the form $x=uvP_1$, $v\in M_1$, and $u\in U_{\Phi(v)}^+$, and the distance between $x$ and $x^{\theta}$ is the unique element $w\in R_1$ such that $v^{-1}u^{-1}\theta uv\in P_1wP_1$.

For any $u\in U^+$ we have, by commutator relations,
$$
u^{-1}\theta u=u^{-1}x_{\varphi'}(1)u=x_{\alpha}(a)x_{\beta}(b)x_{\gamma}(c)x_{\delta}(d)
$$
for some $a,b,c,d\in\FF$, where $\alpha,\beta,\gamma,\delta$ are the unique roots of heights $8,9,10,11$. Explicitly these roots are $\varphi'$, $\varphi'+\alpha_3$, $\varphi'+\alpha_2+\alpha_3$, and $\varphi$, and since the corresponding root subgroups commute with each other the order in the above product is irrelevant. 

It follows that 
$$
v^{-1}u^{-1}\theta uv=x_{v^{-1}\alpha}(a)x_{v^{-1}\beta}(b)x_{v^{-1}\gamma}(c)x_{v^{-1}\delta}(d)
$$
for some $a,b,c,d\in\FF$. If either $v^{-1}\alpha\in \Phi^+$ or $v^{-1}\alpha\in \Phi_1$ (where $\Phi_1$ is generated by $\alpha_2,\alpha_3,\alpha_4$) then $x_{v^{-1}\alpha}(a)\in P_1$, and hence can be ignored as we are interested in $P_1$-double cosets. Similarly for the other terms (as they pairwise commute). By a direct calculation (using $\mathsf{MAGMA}$), for all $v\in M_1$ it turns out that
$$
\{v^{-1}\alpha,v^{-1}\beta,v^{-1}\gamma,v^{-1}\delta\}\cap (\Phi^-\backslash\Phi_1)\subseteq \{-(1000),-(1100),-(1110),-(2342)\}.
$$
It therefore suffices to consider $P_1gP_1$, where
$$
g=x_{-(2342)}(a)x_{-(1110)}(b)x_{-(1100)}(c)x_{-(1000)}(d)\quad\text{with $a,b,c,d\in\FF$}.
$$
A straightforward calculation, using the folding relation, gives
\begin{align*}
BgB=\begin{cases}
Bs_{(2342)}B&\text{if $a\neq 0$}\\
Bs_{(1110)}B&\text{if $a= 0$ and $b\neq 0$}\\
Bs_{(1100)}B&\text{if $a=b=0$ and $c\neq 0$}\\
Bs_{(1000)}B&\text{if $a=b=c=0$ and $d\neq 0$}\\
B&\text{if $a=b=c=d=0$}.
\end{cases}
\end{align*}
Since $P_1s_{(1100)}P_1=P_1s_{(1000)}P_1=P_1s_1P_1$, and $s_{(1110)}=s_1s_2s_3s_2s_1$, it follows that
$$
P_1gP_1\in \{P_1,P_1s_1P_1,P_1s_1s_2s_3s_2s_1P_1,P_1s_{\varphi}P_1\}
$$
In the language of parapolar spaces, this means that every point of $\sF_{4,1}(\FF)$ is either fixed, mapped to a collinear point, mapped to a symplectic point, or mapped to an opposite point by~$\theta$. In particular, no point is mapped to a point at special distance $s_1s_2s_3s_2s_4s_3s_2s_1$ (see Subsection~\ref{sec:parapolar}), and hence $\mathscr{C}$ is not transitive on type $1$ vertices. The arguments for type $4$ vertices are entirely analogous.
\end{proof}

\section{Unipotent elements}\label{sec:unipotent}

Let $\Delta$ be a split irreducible spherical building of exceptional type with root system~$\Phi$. In this section we give an extension of Theorem~\ref{thm:RGD}, showing that every ``polar closed'' (see below) type preserving admissible Dynkin diagram of type $\Phi$ can be realised as the opposition diagram of a unipotent element $u\in U^+$. In fact, we show that these are precisely the diagrams that arise as opposition diagrams of elements in~$U^+$ (for non-special characteristic). 

Let $\sX$ be a type preserving  admissible Dynkin diagram, and let $\sX=\sX_0,\sX_1,\ldots$ be sub-diagrams such that, for $j\geq 1$, the diagram $\sX_{j}$ is obtained from $\sX_{j-1}$ by removing the polar type from one of the connected components of $\sX_{j-1}$. Suppose that this process terminates at step $j=k$ (that is, $\sX_k$ has no polar nodes encircled). We say that a type preserving diagram $\sX$ is \textit{polar closed} if $\sX_k$ is an empty diagram (that is, has no nodes encircled). 

For example, the following diagrams are polar closed
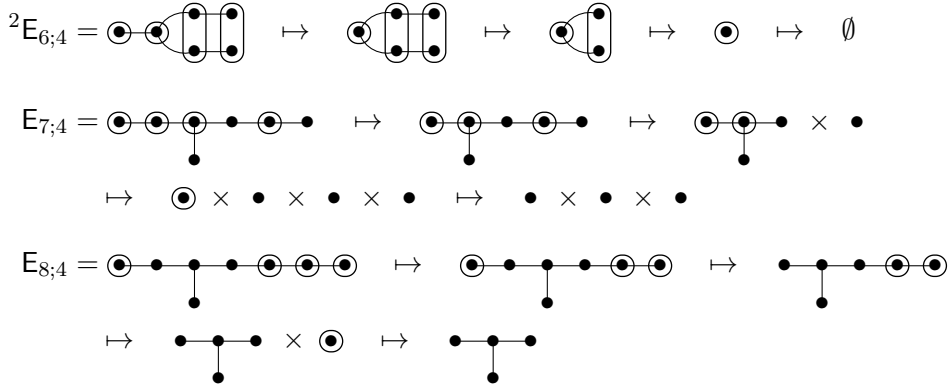
\begin{align*}
{^2}\sE_{6;4}&=\begin{tikzpicture}[scale=0.5,baseline=-0.5ex]
\node at (0,0.8) {};
\node at (0,-0.8) {};
\node [inner sep=0.8pt,outer sep=0.8pt] at (-2,0) (2) {$\bullet$};
\node [inner sep=0.8pt,outer sep=0.8pt] at (-1,0) (4) {$\bullet$};
\node [inner sep=0.8pt,outer sep=0.8pt] at (0,-0.5) (5) {$\bullet$};
\node [inner sep=0.8pt,outer sep=0.8pt] at (0,0.5) (3) {$\bullet$};
\node [inner sep=0.8pt,outer sep=0.8pt] at (1,-0.5) (6) {$\bullet$};
\node [inner sep=0.8pt,outer sep=0.8pt] at (1,0.5) (1) {$\bullet$};
\draw (-2,0)--(-1,0);
\draw (-1,0) to [bend left=45] (0,0.5);
\draw (-1,0) to [bend right=45] (0,-0.5);
\draw (0,0.5)--(1,0.5);
\draw (0,-0.5)--(1,-0.5);
\draw [line width=0.5pt,line cap=round,rounded corners] (2.north west)  rectangle (2.south east);
\draw [line width=0.5pt,line cap=round,rounded corners] (4.north west)  rectangle (4.south east);
\draw [line width=0.5pt,line cap=round,rounded corners] (3.north west)  rectangle (5.south east);
\draw [line width=0.5pt,line cap=round,rounded corners] (1.north west)  rectangle (6.south east);
\end{tikzpicture}\quad\mapsto\quad
 \begin{tikzpicture}[scale=0.5,baseline=-0.5ex]
\node [inner sep=0.8pt,outer sep=0.8pt] at (-1,0) (4) {$\bullet$};
\node [inner sep=0.8pt,outer sep=0.8pt] at (0,-0.5) (5) {$\bullet$};
\node [inner sep=0.8pt,outer sep=0.8pt] at (0,0.5) (3) {$\bullet$};
\node [inner sep=0.8pt,outer sep=0.8pt] at (1,-0.5) (6) {$\bullet$};
\node [inner sep=0.8pt,outer sep=0.8pt] at (1,0.5) (1) {$\bullet$};
\draw (-1,0) to [bend left=45] (0,0.5);
\draw (-1,0) to [bend right=45] (0,-0.5);
\draw (0,0.5)--(1,0.5);
\draw (0,-0.5)--(1,-0.5);
\draw [line width=0.5pt,line cap=round,rounded corners] (4.north west)  rectangle (4.south east);
\draw [line width=0.5pt,line cap=round,rounded corners] (3.north west)  rectangle (5.south east);
\draw [line width=0.5pt,line cap=round,rounded corners] (1.north west)  rectangle (6.south east);
\end{tikzpicture}\quad\mapsto\quad
 \begin{tikzpicture}[scale=0.5,baseline=-0.5ex]
\node [inner sep=0.8pt,outer sep=0.8pt] at (-1,0) (4) {$\bullet$};
\node [inner sep=0.8pt,outer sep=0.8pt] at (0,-0.5) (5) {$\bullet$};
\node [inner sep=0.8pt,outer sep=0.8pt] at (0,0.5) (3) {$\bullet$};
\draw (-1,0) to [bend left=45] (0,0.5);
\draw (-1,0) to [bend right=45] (0,-0.5);
\draw [line width=0.5pt,line cap=round,rounded corners] (4.north west)  rectangle (4.south east);
\draw [line width=0.5pt,line cap=round,rounded corners] (3.north west)  rectangle (5.south east);
\end{tikzpicture}\quad\mapsto\quad
 \begin{tikzpicture}[scale=0.5,baseline=-0.5ex]
\node [inner sep=0.8pt,outer sep=0.8pt] at (-1,0) (4) {$\bullet$};
\draw [line width=0.5pt,line cap=round,rounded corners] (4.north west)  rectangle (4.south east);
\end{tikzpicture}
\quad\mapsto\quad
 \emptyset\\
 \sE_{7;4}&=\begin{tikzpicture}[scale=0.5,baseline=-0.5ex]
\node at (0,0.8) {};
\node at (0,-0.8) {};
\node [inner sep=0.8pt,outer sep=0.8pt] at (-2,0) (1) {$\bullet$};
\node [inner sep=0.8pt,outer sep=0.8pt] at (-1,0) (3) {$\bullet$};
\node [inner sep=0.8pt,outer sep=0.8pt] at (0,0) (4) {$\bullet$};
\node [inner sep=0.8pt,outer sep=0.8pt] at (1,0) (5) {$\bullet$};
\node [inner sep=0.8pt,outer sep=0.8pt] at (2,0) (6) {$\bullet$};
\node [inner sep=0.8pt,outer sep=0.8pt] at (3,0) (7) {$\bullet$};
\node [inner sep=0.8pt,outer sep=0.8pt] at (0,-1) (2) {$\bullet$};
\draw (-2,0)--(3,0);
\draw (0,0)--(0,-1);
\draw [line width=0.5pt,line cap=round,rounded corners] (1.north west)  rectangle (1.south east);
\draw [line width=0.5pt,line cap=round,rounded corners] (3.north west)  rectangle (3.south east);
\draw [line width=0.5pt,line cap=round,rounded corners] (4.north west)  rectangle (4.south east);
\draw [line width=0.5pt,line cap=round,rounded corners] (6.north west)  rectangle (6.south east);
\end{tikzpicture}\quad\mapsto\quad
\begin{tikzpicture}[scale=0.5,baseline=-0.5ex]
\node [inner sep=0.8pt,outer sep=0.8pt] at (-1,0) (3) {$\bullet$};
\node [inner sep=0.8pt,outer sep=0.8pt] at (0,0) (4) {$\bullet$};
\node [inner sep=0.8pt,outer sep=0.8pt] at (1,0) (5) {$\bullet$};
\node [inner sep=0.8pt,outer sep=0.8pt] at (2,0) (6) {$\bullet$};
\node [inner sep=0.8pt,outer sep=0.8pt] at (3,0) (7) {$\bullet$};
\node [inner sep=0.8pt,outer sep=0.8pt] at (0,-1) (2) {$\bullet$};
\draw (-1,0)--(3,0);
\draw (0,0)--(0,-1);
\draw [line width=0.5pt,line cap=round,rounded corners] (3.north west)  rectangle (3.south east);
\draw [line width=0.5pt,line cap=round,rounded corners] (4.north west)  rectangle (4.south east);
\draw [line width=0.5pt,line cap=round,rounded corners] (6.north west)  rectangle (6.south east);
\end{tikzpicture}\quad\mapsto\quad
\begin{tikzpicture}[scale=0.5,baseline=-0.5ex]
\node [inner sep=0.8pt,outer sep=0.8pt] at (-1,0) (3) {$\bullet$};
\node [inner sep=0.8pt,outer sep=0.8pt] at (0,0) (4) {$\bullet$};
\node [inner sep=0.8pt,outer sep=0.8pt] at (1,0) (5) {$\bullet$};
\node [inner sep=0.8pt,outer sep=0.8pt] at (2,0) (6) {$\times$};
\node [inner sep=0.8pt,outer sep=0.8pt] at (3,0) (7) {$\bullet$};
\node [inner sep=0.8pt,outer sep=0.8pt] at (0,-1) (2) {$\bullet$};
\draw (-1,0)--(1,0);
\draw (0,0)--(0,-1);
\draw [line width=0.5pt,line cap=round,rounded corners] (3.north west)  rectangle (3.south east);
\draw [line width=0.5pt,line cap=round,rounded corners] (4.north west)  rectangle (4.south east);
\end{tikzpicture}\\
&\quad\mapsto\quad\begin{tikzpicture}[scale=0.5,baseline=-0.5ex]
\node [inner sep=0.8pt,outer sep=0.8pt] at (-2,0) (1) {$\bullet$};
\node [inner sep=0.8pt,outer sep=0.8pt] at (-1,0) (3) {$\times$};
\node [inner sep=0.8pt,outer sep=0.8pt] at (0,0) (4) {$\bullet$};
\node [inner sep=0.8pt,outer sep=0.8pt] at (1,0) (5) {$\times$};
\node [inner sep=0.8pt,outer sep=0.8pt] at (2,0) (6) {$\bullet$};
\node [inner sep=0.8pt,outer sep=0.8pt] at (3,0) (7) {$\times$};
\node [inner sep=0.8pt,outer sep=0.8pt] at (4,0) (8) {$\bullet$};
\draw [line width=0.5pt,line cap=round,rounded corners] (1.north west)  rectangle (1.south east);
\end{tikzpicture}
\quad\mapsto\quad
\begin{tikzpicture}[scale=0.5,baseline=-0.5ex]
\node [inner sep=0.8pt,outer sep=0.8pt] at (0,0) (4) {$\bullet$};
\node [inner sep=0.8pt,outer sep=0.8pt] at (1,0) (5) {$\times$};
\node [inner sep=0.8pt,outer sep=0.8pt] at (2,0) (6) {$\bullet$};
\node [inner sep=0.8pt,outer sep=0.8pt] at (3,0) (7) {$\times$};
\node [inner sep=0.8pt,outer sep=0.8pt] at (4,0) (8) {$\bullet$};
\end{tikzpicture}\\
\sE_{8;4}&=\begin{tikzpicture}[scale=0.5,baseline=-0.5ex]
\node at (0,0.8) {};
\node at (0,-0.8) {};
\node [inner sep=0.8pt,outer sep=0.8pt] at (-2,0) (1) {$\bullet$};
\node [inner sep=0.8pt,outer sep=0.8pt] at (-1,0) (3) {$\bullet$};
\node [inner sep=0.8pt,outer sep=0.8pt] at (0,0) (4) {$\bullet$};
\node [inner sep=0.8pt,outer sep=0.8pt] at (1,0) (5) {$\bullet$};
\node [inner sep=0.8pt,outer sep=0.8pt] at (2,0) (6) {$\bullet$};
\node [inner sep=0.8pt,outer sep=0.8pt] at (3,0) (7) {$\bullet$};
\node [inner sep=0.8pt,outer sep=0.8pt] at (4,0) (8) {$\bullet$};
\node [inner sep=0.8pt,outer sep=0.8pt] at (0,-1) (2) {$\bullet$};
\draw (-2,0)--(4,0);
\draw (0,0)--(0,-1);
\draw [line width=0.5pt,line cap=round,rounded corners] (1.north west)  rectangle (1.south east);
\draw [line width=0.5pt,line cap=round,rounded corners] (6.north west)  rectangle (6.south east);
\draw [line width=0.5pt,line cap=round,rounded corners] (7.north west)  rectangle (7.south east);
\draw [line width=0.5pt,line cap=round,rounded corners] (8.north west)  rectangle (8.south east);
\end{tikzpicture}
\quad\mapsto\quad
\begin{tikzpicture}[scale=0.5,baseline=-0.5ex]
\node [inner sep=0.8pt,outer sep=0.8pt] at (-2,0) (1) {$\bullet$};
\node [inner sep=0.8pt,outer sep=0.8pt] at (-1,0) (3) {$\bullet$};
\node [inner sep=0.8pt,outer sep=0.8pt] at (0,0) (4) {$\bullet$};
\node [inner sep=0.8pt,outer sep=0.8pt] at (1,0) (5) {$\bullet$};
\node [inner sep=0.8pt,outer sep=0.8pt] at (2,0) (6) {$\bullet$};
\node [inner sep=0.8pt,outer sep=0.8pt] at (3,0) (7) {$\bullet$};
\node [inner sep=0.8pt,outer sep=0.8pt] at (0,-1) (2) {$\bullet$};
\draw (-2,0)--(3,0);
\draw (0,0)--(0,-1);
\draw [line width=0.5pt,line cap=round,rounded corners] (1.north west)  rectangle (1.south east);
\draw [line width=0.5pt,line cap=round,rounded corners] (6.north west)  rectangle (6.south east);
\draw [line width=0.5pt,line cap=round,rounded corners] (7.north west)  rectangle (7.south east);
\end{tikzpicture}
\quad\mapsto\quad
\begin{tikzpicture}[scale=0.5,baseline=-0.5ex]
\node [inner sep=0.8pt,outer sep=0.8pt] at (-1,0) (3) {$\bullet$};
\node [inner sep=0.8pt,outer sep=0.8pt] at (0,0) (4) {$\bullet$};
\node [inner sep=0.8pt,outer sep=0.8pt] at (1,0) (5) {$\bullet$};
\node [inner sep=0.8pt,outer sep=0.8pt] at (2,0) (6) {$\bullet$};
\node [inner sep=0.8pt,outer sep=0.8pt] at (3,0) (7) {$\bullet$};
\node [inner sep=0.8pt,outer sep=0.8pt] at (0,-1) (2) {$\bullet$};
\draw (-1,0)--(3,0);
\draw (0,0)--(0,-1);
\draw [line width=0.5pt,line cap=round,rounded corners] (6.north west)  rectangle (6.south east);
\draw [line width=0.5pt,line cap=round,rounded corners] (7.north west)  rectangle (7.south east);
\end{tikzpicture}\\
&\quad\mapsto\quad\begin{tikzpicture}[scale=0.5,baseline=-0.5ex]
\node [inner sep=0.8pt,outer sep=0.8pt] at (-1,0) (3) {$\bullet$};
\node [inner sep=0.8pt,outer sep=0.8pt] at (0,0) (4) {$\bullet$};
\node [inner sep=0.8pt,outer sep=0.8pt] at (1,0) (5) {$\bullet$};
\node [inner sep=0.8pt,outer sep=0.8pt] at (2,0) (6) {$\times$};
\node [inner sep=0.8pt,outer sep=0.8pt] at (3,0) (7) {$\bullet$};
\node [inner sep=0.8pt,outer sep=0.8pt] at (0,-1) (2) {$\bullet$};
\draw (-1,0)--(1,0);
\draw (0,0)--(0,-1);
\draw [line width=0.5pt,line cap=round,rounded corners] (7.north west)  rectangle (7.south east);
\end{tikzpicture}
\quad\mapsto\quad
\begin{tikzpicture}[scale=0.5,baseline=-0.5ex]
\node [inner sep=0.8pt,outer sep=0.8pt] at (-1,0) (3) {$\bullet$};
\node [inner sep=0.8pt,outer sep=0.8pt] at (0,0) (4) {$\bullet$};
\node [inner sep=0.8pt,outer sep=0.8pt] at (1,0) (5) {$\bullet$};
\node [inner sep=0.8pt,outer sep=0.8pt] at (0,-1) (2) {$\bullet$};
\draw (-1,0)--(1,0);
\draw (0,0)--(0,-1);
\end{tikzpicture}
\end{align*}
whereas the diagrams $\sF_{4;1}^4$ and $\sG_{2;1}^1$ are not polar closed (as the polar node is not encircled). Indeed, by direct inspection of the list of admissible diagrams these two diagrams are the only non-polar closed diagrams of exceptional type.

Suppose that $\sX$ is polar closed. Thus one may define diagrams $\sX_0,\ldots,\sX_k$ by successively removing polar types, until no polar nodes are encircled, and $\sX_k$ is an empty diagram. Let $\varphi_1,\ldots,\varphi_k\in\Phi^+$ be the highest roots removed at each stage. For example, the highest roots corresponding to the three polar closed diagrams above are:
\begin{align*}
\varphi_{\sE_6}&=(122321)\mapsto\varphi_{\sA_5}=(101111)\mapsto\varphi_{\sA_3}=(001110)\mapsto\varphi_{\sA_1}=(000100)\\
\varphi_{\sE_7}&=(2234321)\mapsto\varphi_{\sD_6}=(0112221)\mapsto\varphi_{\sD_4}=(0112100)\mapsto\varphi_{\sA_1}=(0010000)\\
\varphi_{\sE_8}&=(23465432)\mapsto \varphi_{\sE_7}=(22343210)\mapsto\varphi_{\sD_6}=(01122210)\mapsto\varphi_{\sA_1}=(00000010).
\end{align*}
In general the sequence of diagrams $\sX_0,\ldots,\sX_k$ is not unique (for example, starting with the diagram $\sE_{7;7}$ we have $\sE_{7;7}\mapsto\sD_{6;6}\mapsto\sD_{5;5}\times \sA_{1;1}$, from which point one may choose to either remove the polar node of the $\sD_5$ component, or the $\sA_1$ component). However it is clear that the set $\{\varphi_1,\ldots,\varphi_k\}$ of highest roots obtained is independent of the choices made. Moreover, note that these roots are mutually perpendicular (by definition of the polar type), and hence the subgroup $U(\sX)$ of $G$ generated by the root subgroups $U_{\varphi_1}^+,\ldots,U_{\varphi_k}^+$ is abelian. We call an element $g\in U(\sX)$ \textit{generic} if 
$$
g=x_{\varphi_1}(a_1)\cdots x_{\varphi_k}(a_k)\quad\text{with $a_1,\ldots,a_k\neq 0$}.
$$

Let the \textit{dual polar node} of a Dynkin diagram be the subset $\wp'\subseteq S$ corresponding to the polar node of the dual diagram. Thus $\wp'=\wp$ in the simply laced case, and $\wp'=\{1\},\{2\},\{4\},\{1\}$ in the cases $\Phi=\sB_n,\sC_n,\sF_4,\sG_2$, respectively. We call a type preserving admissible diagram $\sX$ \textit{dual polar closed} if the above algorithm, with each occurrence of ``polar node'' replaced by ``dual polar node'' terminates in an empty diagram. Let $\varphi_1',\ldots,\varphi_{\ell}'\in\Phi^+$ denote the sequence of highest short roots obtained in an analogous way. In the case of special characteristic the subgroup $U(\sX)'$ of $G_{\Phi}(\FF)$ generated by $U_{\varphi_1'}^+,\ldots, U_{\varphi_{\ell}'}^+$ is commutative. We define \textit{generic} element of $U(\sX)'$ in an analogous way. By inspection, note that if $\sX$ is type preserving and is not polar closed, then $\sX$ is necessarily dual polar closed. 

In this section we prove the following theorem.

\begin{thm}\label{thm:unipotentdiags}
Let $\sX=(\Gamma,J,\pi)$ be a type preserving admissible Dynkin diagram of exceptional type~$\Phi$, and let $G=G_{\Phi}(\FF)$. 
\begin{compactenum}[$(1)$]
\item Suppose that $\mathrm{char}(\FF)$ is not special. Then $\sX$ is the opposition diagram of an element of $U^+$ if and only if $\sX$ is polar closed. Moreover, if $\sX$ is polar closed then each generic element $\theta\in U(\sX)$ has opposition diagram $\sX$.
\item Suppose that $\mathrm{char}(\FF)$ is special. Then $\sX$ is the opposition diagram of an element of $U^+$. Moreover, if $\sX$ is polar closed then each generic element $\theta\in U(\sX)$ has opposition diagram~$\sX$, and if $\sX$ is dual polar closed, then each generic element $\theta\in U(\sX)'$ has opposition diagram~$\sX$.
\end{compactenum}
\end{thm}

The proof of Theorem~\ref{thm:unipotentdiags} is given in this section, however one ingredient -- showing that in non-special characteristic the diagrams $\sF_{4;1}^4$ and $\sG_{2,1}^1$ are not the opposition diagrams of any element $\theta\in U^+$ -- will be postponed until later in the paper (see Theorems~\ref{thm:F41hom} and~\ref{thm:mainG2}). 


\begin{lemma}\label{lem:E7symplectic}
Let $\Phi$ be of type $\sE_7$ with highest root $\varphi$. Then $(Bs_{\varphi}B\cdot Bs_{\varphi}B)\cap Bw_0w_{\sE_6}B=\emptyset$. 
\end{lemma}

\begin{proof}
Write $v=w_{\sD_5}w_{\sE_6}$ (where $\sD_5$ is generated by $\{s_2,s_3,s_4,s_5,s_6\}$). Since $v^{-1}\varphi=\alpha_7$ we have $
v^{-1}s_{\varphi}v=s_{v^{-1}\varphi}=s_{7},
$
and thus $s_{\varphi}=vs_7v^{-1}$. Note that $v\in W_7$ (the parabolic subgroup generated by $S\backslash\{s_7\}$). Thus if $w\in W$ is such that 
$$BwB\subseteq Bs_{\varphi}B\cdot Bs_{\varphi}B=Bvs_7v^{-1}B\cdot Bvs_7v^{-1}B$$ then by~(\ref{eq:doublecoset}) and the deletion condition there exists a reduced expression for $w$ containing at most $2$ occurrences of the generator $s_7$. However every reduced expression for $y=w_0w_{\sE_6}$ contains at least $3$ occurrences of the generator $s_7$. To see this, note that every reduced expression for $y$ must start and end with $s_7$ (as it is minimal length in its $W_7$-double coset). It is thus sufficient to show that $s_7ys_7$ is not in $W_7$. But we have
$$
(s_7ys_7)^{-1}(\varphi_{\sD_6})=s_7w_{\sE_6}w_{\sE_7}\varphi_{\sD_6}=-s_7w_{\sE_6}\varphi_{\sD_6}=-s_7\varphi_{\sD_6}=-\varphi_{\sD_6},
$$
and so $\varphi_{\sD_6}\in \Phi(s_7ys_7)$, and so indeed $s_7ys_7\notin W_{7}$. 
\end{proof}

\begin{remark}\label{rem:E77translation}
Geometrically Lemma~\ref{lem:E7symplectic} boils down to the following statement in the $\sE_{7,7}(\FF)$ geometry (with point set $G/P_7$; recall the notation introduced in Subsection~\ref{eq:notationparabolic}): If $x,y,z$ are points of $\sE_{7,7}(\FF)$ with $x$ and $y$ collinear, and $y$ and $z$ collinear, then $x$ and $z$ are not opposite in $\sE_{7,7}(\FF)$. To make this translation, note that if $(Bs_{\varphi}B\cdot Bs_{\varphi}B)\cap Bw_0w_{\sE_6}B= \emptyset$ then, following the above proof, we have $(P_7s_7P_7\cdot P_7s_7P_7)\cap P_7w_0P_7= \emptyset$. Then note that $(P_7s_7P_7\cdot P_7s_7P_7)/P_7$ can be interpreted as the set of points collinear to some point collinear to the base point $P_7$, and $(P_7w_0P_7)/P_7$ is the set of points opposite the base point $P_7$. 
\end{remark}

Recall that for roots $\alpha,\beta\in\Phi$ we write $\alpha\leq \beta$ if and only if $\beta-\alpha$ is a nonnegative linear combination of simple roots. 

 \begin{lemma}\label{lem:restrictions}
 Let $\Delta$ be an irreducible split spherical building with root system~$\Phi$, and suppose that either $|\FF|>2$, or that $\theta$ is an involution. 
 \begin{compactenum}[$(1)$]
 \item Suppose that $\Phi=\sE_6$. If $\theta\in \langle U_{\alpha}\mid \alpha\geq \alpha_1\rangle$ then $\disp(\theta)\leq 30$. 
 \item Suppose that $\Phi=\sE_7$, and let $\varphi_2=\varphi_{\sD_6}=(0112221)$. 
 \begin{compactenum}[$(a)$]
\item If $\theta\in\langle U_{\alpha}\mid \alpha\geq \varphi_2\rangle$ then $\disp(\theta)\leq 50$.
\item If $\theta\in \langle U_{\alpha}\mid \alpha\geq \alpha_7\rangle$ then $\disp(\theta)\leq 51$. 
\item If $\theta\in\langle U_{\alpha}\mid \alpha\geq \alpha_1\rangle$ then $\disp(\theta)\leq 60$. 
\end{compactenum}
\item Suppose that $\Phi=\sE_8$.
\begin{compactenum}[$(a)$]
\item If $\theta\in \langle U_{\alpha}\mid \height(\alpha)\geq 23\rangle$ then $\disp(\theta)\leq 90$. 
\item If $\theta\in \langle U_{\alpha}\mid \alpha\geq \alpha_8\rangle$ then $\disp(\theta)\leq 108$. 
\end{compactenum}
 \end{compactenum}
 \end{lemma}
 
 \begin{proof}
(1) Consider the $\sD_5$ subsystem generated by $\{\alpha_j\mid j\neq 1\}$. If $\alpha\geq \alpha_1$ then $\langle\alpha,\omega_1\rangle>0$ and so $\alpha\in \Phi^+\backslash\sD_5=\Phi(w_0w_{\sD_5})$. Thus, with $w_1=w_0w_{\sD_5}$, Proposition~\ref{prop:standardtechnique} gives $\disp(\theta)\leq 2\ell(w_1)-1=31$. By the classification of admissible diagrams this implies that $\disp(\theta)\leq 30$. 

(2)(a) We claim that $\{\alpha\mid \alpha\geq \varphi_2\}\subseteq \Phi(s_7w_0w_{\sE_6})$. Let $\alpha\geq \varphi_2$. Then $\langle\alpha,\omega_6\rangle=2$ and $\langle\alpha,\omega_7\rangle=1$ (by inspecting $\varphi$ and $\varphi_2$), and since $\langle\alpha_i,\alpha_7\rangle=0$ for $i\neq 6,7$ and $\langle\alpha_6,\alpha_7\rangle=-1$ and $\langle\alpha_7,\alpha_7\rangle=2$ it follows that $\langle \alpha,\alpha_7\rangle=0$. Thus $s_7\alpha=\alpha$. Writing $\alpha=\varphi_2+\beta$ we have $\beta\in \sE_6$ (because $\langle\varphi_2,\omega_7\rangle=1=\langle\varphi,\omega_7\rangle$), and thus, since $w_{\sE_6}\varphi_2=\varphi_2$, we have
$$
(s_7w_0w_{\sE_6})^{-1}\alpha=w_{\sE_6}w_0\alpha=-w_{\sE_6}\alpha=-w_{\sE_6}\varphi_2-w_{\sE_6}\beta=-\varphi_2-w_{\sE_6}\beta.
$$
Since $-w_{\sE_6}\beta\in \sE_6$ we have $\langle-\varphi_2-w_{\sE_6}\beta,\omega_7\rangle=-\langle\varphi_2,\omega_7\rangle=-1$, and thus $(s_7w_0w_{\sE_6})^{-1}\alpha\in -\Phi^+$. 

Thus, by (\ref{eq:further}) we have $\disp(\theta)\leq \max\{\ell(w)\mid BwB\subseteq Bw_1B\cdot Bw_1^{-1}B\}$, where $w_1=s_7w_0w_{\sE_6}$. Note that $\ell(w_1)=63-37=26$, and so Proposition~\ref{prop:standardtechnique} gives $\disp(\theta)\leq 51$. It remains to eliminate the possibility of $\disp(\theta)=51$. By the last sentence in the proof of Proposition~\ref{prop:standardtechnique}, if $\disp(\theta)=51$ then $\ell(w_2sw_2^{-1})=51$, where $w_1=w_2s$ with $\ell(w_2s)=\ell(w_2)+1$. Since $\ell(w_1s_7)=\ell(w_1)-1$, (as $w_1\alpha_7\in-\Phi^+$) we can take $w_2=w_1s_7$ and $s=s_7$. But then 
$$
w_2s_7w_2^{-1}=w_1s_7w_1^{-1}=s_7w_0w_{\sE_6}s_7w_{\sE_6}w_0s_7=s_{\varphi},
$$
which only has length $2\mathrm{ht}(\varphi)-1=33$, a contradiction.

(2)(b) If $\alpha\geq \alpha_7$ then $\alpha\in \Phi^+\backslash\sE_6=\Phi(w_0w_{\sE_6})$. Thus by Proposition~\ref{prop:standardtechnique} we have $\disp(\theta)\leq 2\ell(w_0w_{\sE_6})-1=2(63-36)-1=53$. This in turn implies, from the classification of admissible diagrams, that $\disp(\theta)\leq 51$.

(2)(c) If $\alpha\geq \alpha_1$ then $\alpha\in \Phi^+\backslash\sD_6=\Phi(w_0w_{\sD_6})$. Now note that $w_0w_{\sD_6}=s_{\varphi}$. By Lemma~\ref{lem:E7symplectic} $Bs_{\varphi}B\cdot Bs_{\varphi}B$ does not intersect $Bw_0B$, and so from the proof of Proposition~\ref{prop:standardtechnique} we see that $\theta$ is domestic. Thus $\disp(\theta)\leq 60$.

(3)(a) We claim that $\{\alpha\mid \height(\alpha)\geq 23\}\subseteq \Phi(w_1)$, where $w_1 =s_4s_5s_6s_7s_8s_{\varphi}$ (this element has length $\ell(s_{\varphi})-5=52$). Direct calculation shows that each of the elements $s_4,s_5,s_6,s_7,s_8$ preserve the set of $6$ roots $\{\alpha\in\Phi^+\mid\mathrm{ht}(\alpha)\geq 23\}$. Thus if $\height(\alpha)\geq 23$ then $w_1^{-1}\alpha=s_{\varphi}s_8s_7s_6s_5s_4\alpha=s_{\varphi}\beta$ for some $\beta$ with $\height(\beta)\geq 23$, and thus $\langle w_1^{-1}\alpha,\omega_8\rangle=\langle\beta,\omega_8\rangle-2\langle\beta,\varphi\rangle =-\langle\beta,\omega_8\rangle<0$ (using $\langle\alpha_i,\varphi\rangle=\delta_{i,8}$) and so $w_1^{-1}\alpha\in -\Phi^+$, hence the claim. It follows from Proposition~\ref{prop:standardtechnique} that $\disp(\theta)\leq 103$, and hence by the classification of admissible diagrams $\disp(\theta)\leq 90$.

(3)(b) Note that if $\alpha\geq \alpha_8$ then $\langle s_{\varphi}\alpha,\omega_8\rangle=-\langle\alpha,\omega_8\rangle$ and thus $\alpha\in \Phi(s_{\varphi})$. Since $\ell(s_{\varphi})=2\height(\varphi)-1=57$ it follows from Proposition~\ref{prop:standardtechnique} that $\disp(\theta)\leq 113$, and thus by the classification of admissible diagrams $\disp(\theta)\leq 108$.  
\end{proof}

\begin{lemma}\label{lem:long2}
Let $\sX=(\Gamma,J,\pi)$ be polar closed, and let $\varphi_1,\ldots,\varphi_k$ be the highest roots obtained by the above algorithm. Then $\ell(s_{\varphi_1}\cdots s_{\varphi_k})=\sum_{j=1}^k\ell(s_{\varphi_j})$ and $s_{\varphi_1}\cdots s_{\varphi_k}=w_{S\backslash J}w_0$.
\end{lemma}

\begin{proof}
Let $\wp_1,\ldots,\wp_k$ be the polar nodes. Let $S_0=S$, and define $S_j=S_{j-1}\backslash\wp_j$ and $\Phi_j=\Phi_{S_j}$ for $j=1,\ldots,k$. By~(\ref{eq:highestrootinversionset}) we have $\Phi(s_{\varphi_1})=\Phi_0^+\backslash\Phi_1$, and it follows by induction that $\Phi(s_{\varphi_j})=\Phi_{j-1}^+\backslash\Phi_j$ for $1\leq j\leq k$. In particular, the inversion sets $\Phi(s_{\varphi_1}),\ldots,\Phi(s_{\varphi_k})$ are disjoint, and thus 
$$
\Phi(s_{\varphi_1}\cdots s_{\varphi_k})=\bigcup_{j=1}^k\Phi(s_{\varphi_j})=\Phi^+\backslash \Phi_k=\Phi^+\backslash\Phi(w_{S\backslash J}).
$$
But also clearly $\Phi(w_{S\backslash J}w_0)=\Phi^+\backslash \Phi(w_{S\backslash J})$, and hence the result.
\end{proof}

We are now ready to prove Theorem~\ref{thm:unipotentdiags}. 

\begin{proof}[Proof of Theorem~\ref{thm:unipotentdiags}]
 As noted above, we postpone the proof of the fact that if characteristic is not special, and $\sX$ is not polar closed, then $\sX$ is not the opposition diagram of any element of~$U^+$ until Theorem~\ref{thm:F41hom} and Theorem~\ref{thm:mainG2}. 
 
 Thus suppose that $\sX$ is polar closed, and let $\varphi_1,\ldots,\varphi_k$ be the highest roots obtained from the above algorithm. If $k=1$ then the result follows from Theorem~\ref{thm:RGD} (as $\varphi_1$ is a long root). So suppose that $k>1$. Write $\theta=x_{\varphi_1}(a_1)\cdots x_{\varphi_k}(a_k)$ with $a_1,\ldots,a_k\neq 0$. By Lemma~\ref{lem:long2} and Lemma~\ref{lem:perpendicularroots} we have 
$$
w_0^{-1}\theta w_0=x_{-\varphi_1}(\pm a_1)\cdots x_{-\varphi_k}(\pm a_k)\in Bs_{\varphi_1}\cdots s_{\varphi_k}B=Bw_{S\backslash J}w_0B
$$
(we have used the fact that $w_0\varphi_j=-\varphi_j$ for all $j$, which follows from the defining property of the highest root). Thus the chamber $w_0B$ is mapped to Weyl distance~$w_{S\backslash J}w_0$. Thus the type $J$-simplex of the chamber $w_0B$ is mapped onto an opposite simplex, and so $J\subseteq\Type(\theta)$. Hence it remains to show that $\Type(\theta)\subseteq J$. We achieve this by bounding the displacement by an appropriate bound, and appealing to the classification of admissible diagrams. Note that if $\FF=\FF_2$ then $\theta$ is an involution, and so Lemma~\ref{lem:restrictions} holds in all cases. Also, if $J=S$ (the full opposition diagram) then there is nothing remaining to prove (as the above shows that $\theta$ is not domestic in this case).

 We consider each diagram.
\smallskip

\noindent\textit{The case $\Phi=\sE_6$:} Consider the diagram ${^2}\sE_{6;2}$. Then $\varphi=\varphi_1=\varphi_{\sE_6}$ and $\varphi_2=\varphi_{\sA_5}$. Since $\theta\in \langle U_{\alpha}\mid \alpha\geq \varphi_2\rangle$ and since $\varphi_2\geq \alpha_1$, we have $\disp(\theta)\leq 30$ by Lemma~\ref{lem:restrictions}, and the result follows. 
\smallskip

\noindent\textit{The case $\Phi=\sE_7$:} Consider the diagram $\sE_{7;2}$. Then $\varphi_1=\varphi_{\sE_7}$ and $\varphi_2=\varphi_{\sD_6}$. By Lemma~\ref{lem:restrictions} we have $\disp(\theta)\leq 50$, hence the result. Consider the diagram $\sE_{7;3}$. Then $\varphi_1=\varphi_{\sE_7}$, $\varphi_2=\varphi_{\sD_6}$, and $\varphi_3=\alpha_7$. Lemma~\ref{lem:restrictions} gives $\disp(\theta)\leq 51$, and hence the result. Consider the diagram $\sE_{7;4}$. Then $\varphi_1=\varphi_{\sE_7}$, $\varphi_2=\varphi_{\sD_6}$, $\varphi_3=\varphi_{\sD_4}$, and $\varphi_4=\alpha_3$. Replace $\theta$ by the conjugate $\theta'=s_1^{-1}\theta s_1$. Since $s_1\varphi_1, s_1\varphi_2, s_1\varphi_3,s_1\varphi_4\geq \alpha_1$ Lemma~\ref{lem:restrictions} gives $\disp(\theta)=\disp(\theta')\leq 60$. 
\smallskip

\noindent\textit{The case $\Phi=\sE_8$:} Consider the diagram $\sE_{8;2}$. Then $\varphi_1=\varphi_{\sE_8}$ and $\varphi_2=\varphi_{\sD_6}$. We claim that $\disp(\theta)\leq 90$. To see this, let $\beta=(00111111)$ be the highest root of an $\sA_6$ subsystem. Then $s_{\beta}\varphi_1=(23354321)$ and $s_{\beta}\varphi_2=(22454321)$ are the two roots of $\sE_8$ with height $23$. Thus the conjugate $\theta'=s_{\beta}^{-1}\theta s_{\beta}$ satisfies $\theta'\in \langle U_{\alpha}\mid \height(\alpha)\geq 23\rangle$, and so by Lemma~\ref{lem:restrictions} we have $\disp(\theta)=\disp(\theta')\leq 90$.

Consider the diagram $\sE_{8;4}$. Then $\varphi_1=\varphi_{\sE_8}$, $\varphi_2=\varphi_{\sE_7}$, $\varphi_3=\varphi_{\sD_6}$, and $\varphi_4=\alpha_7$. By direct calculation the roots $s_8\varphi_1$, $s_8\varphi_2$, $s_8\varphi_3$, and $s_8\alpha_7$ are all elements of $\Phi^+\backslash\Phi_{\sE_7}$. Therefore the conjugate $\theta'=s_8^{-1}\theta s_8$ satisfies $\theta'\in\langle U_{\alpha}\mid \alpha\geq \alpha_8\rangle$, and so by Lemma~\ref{lem:restrictions} we have $\disp(\theta)=\disp(\theta')\leq 108$, completing the proof for $\sE_8$.
\smallskip

\noindent\textit{The case $\Phi=\sF_4$:} Consider the diagram $\sF_{4;2}$. Then $\varphi_1=\varphi_{\sF_4}$ and $\varphi_2=\varphi_{\sC_3}$. Let $v=s_1s_2s_1$. Then $v\varphi_1=\varphi'-\alpha_3$ and $v\varphi_2=\varphi'+\alpha_3$, where $\varphi'=(1232)$ is the highest short root. Replace $\theta$ by the conjugate $\theta'=v\theta v^{-1}=x_{\varphi'-\alpha_3}(a)x_{\varphi'+\alpha_3}(b)$. Let $X=\{(0100),(0010),(1100),(0120),(1120)$. Then $x_{\gamma}(c)\theta'=\theta' x_{\gamma}(c)$ for all $\gamma\in\Phi^+\backslash X$. Each $u\in U^+$ can be written as $u=u_1u_2$ with $u_1\in U_{\Phi^+\backslash X}^+$ and $u_2\in U_X^+$. Write $u_2=x_{(0010)}(z_1)x_{(0120)}(z_2)x_{(1120)}(z_3)x_{(0100)}(z_4)x_{(1100)}(z_5)$. Since $u_1^{-1}\theta' u_1=\theta'$, a calculation using commutator relations gives
\begin{align*}
u^{-1}\theta' u&=x_{(1222)}(a)x_{(1232)}(-z_1a)x_{(1242)}(c)x_{(1342)}(-z_4c+z_2a)x_{(2342)}(-z_5c+z_3a),
\end{align*}
where $c=z_1^2a+b$, and hence
\begin{align*}
w_0^{-1}u^{-1}\theta' uw_0&=x_{-(1222)}(-a)x_{-(1232)}(z_1a)x_{-(1242)}(-c)x_{-(1342)}(z_4c-z_2a)x_{-(2342)}(z_5c-z_3a).
\end{align*}
Using the folding relation and commutator relations we obtain
\begin{align*}
Bw_0^{-1}u^{-1}\theta uw_0B&=\begin{cases}
Bs_{\varphi_{\sC_3}}s_{\varphi}B=Bw_{\sB_2}w_0B&\text{if $z_5c-z_3a\neq 0$}\\
Bs_{(0110)}s_{\varphi'}B&\text{if $z_5c-z_3a=0$ and $z_4c-z_2a\neq 0$}\\
Bs_{(0010)}s_{\varphi'}B&\text{if $z_5c-z_3a=0$ and $z_4c-z_2a= 0$ and $c\neq 0$}\\
Bs_{\varphi'}B&\text{if $z_5c-z_3a=0$ and $z_4c-z_2a= 0$ and $c=0$},
\end{cases}
\end{align*}
and so by the standard technique $\theta$ is domestic with opposition diagram $\sF_{4;2}$. 
\smallskip

Finally, in special characteristic, in type $\sF_4$ the element $x_{\varphi'}(a)$ ($a\neq 0$) has opposition diagram $\sF_{4;4}^1$ and in type $\sG_2$ the element $x_{\varphi'}(a)$ ($a\neq 0$) has opposition diagram $\sG_{2;1}^1$ (by Theorem~\ref{thm:short}). Moreover, very similar calculations to those above shows that in type $\sF_4$, for all fields, $x_{\varphi'}(a)x_{\varphi_{\sB_3}'}(b)$ with $a,b\neq 0$ has opposition diagram $\sF_{4;2}$, completing the proof.
\end{proof}

\section{Classification of domestic homologies}\label{sec:homologies}

In this section we classify the domestic homologies of split buildings of exceptional types. Throughout this section we may assume that $\Delta$ is a large building, for over the field $\FF_2$ there are no nontrivial homologies. In Lemma~\ref{lem:thickframe} we recall the basic fact that the fixed element structure of a homology $\theta$ is a (typically non-thick) building~$\Delta_{\theta}$ of the same type as~$\Delta$. Following Scharlau~\cite{Sch:87}, the \textit{thick frame} $\Delta_{\theta}'$ of $\Delta_{\theta}$ is a thick building naturally associated to~$\Delta_{\theta}$, and the type $W_{\theta}$ of this building is a reflection subgroup of~$W$. Our classification of domestic homologies is in terms of these reflection subgroups. The data provided in Appendix~\ref{app:data} is useful for this section.

\begin{lemma}\label{lem:thickframe}
Let $\Delta$ be a split spherical building of type $(W,S)$ and let $\theta$ be a homology of $\Delta$. Let $\Delta_{\theta}$ be the set of fixed chambers of $\theta$. Then $\Delta_{\theta}$ is a (typically non-thick) building of type $(W,S)$. Moreover, if $P$ is an $s$-panel of $\Delta$ with $\cC(P)\cap\Delta_{\theta}\neq\emptyset$ then either $\cC(P)\subseteq\Delta_{\theta}$ or $|\cC(P)\cap\Delta_{\theta}|=2$.
\end{lemma}

\begin{proof} The proof is straightforward. 
\end{proof}

If $\Delta_{\theta}$ is the fixed subbuilding of a homology~$\theta$, we refer to the \textit{thick} (if $\cC(P)\subseteq\Delta_{\theta}$) and \textit{thin} (if $|\cC(P)\cap\Delta_{\theta}|=2$) panels of $\Delta_{\theta}$. If $\cA$ is an apartment of $\Delta_{\theta}$, then we refer to the thin and thick walls of $\cA$.

The \textit{thick frame} $\Delta_{\theta}'$ of $\Delta_{\theta}$ is a building whose chambers are the thin-classes of chambers of $\Delta_{\theta}$, with adjacency given by adjacency of representatives of these classes in $\Delta_{\theta}$ (see \cite{Sch:87}; here \textit{thin-classes} are the classes of the finest equivalence relation containing the thin panels). The building $\Delta_{\theta}'$ has type $W_{\theta}$, where $W_{\theta}$ is the reflection subgroup of $W$ generated by the reflections about the thick walls of any apartment of $\Delta_{\theta}$. 

The embedding $W_{\theta}\hookrightarrow W$ is only defined up to conjugacy, however in practice if $\theta\in H$ we fix this embedding by taking reflections in the base apartment. If $\theta\in H$ then the thick walls of the base apartment $\cA_0$ are easily computed as follows. Writing $\theta=h_{\omega_1}(c_1)\cdots h_{\omega_N}(c_N)$ we have
$$
\theta x_{\alpha}(a)\theta^{-1}=x_{\alpha}(ac_1^{\langle\alpha,\omega_1\rangle}\cdots c_N^{\langle\alpha,\omega_N\rangle}),
$$
and so the $\alpha$-wall of the base apartment is thick if and only if $c_1^{\langle\alpha,\omega_1\rangle}\cdots c_N^{\langle\alpha,\omega_N\rangle}=1$. For homologies $\theta\in H$, we define the \textit{root system of $\theta$} by
$$
\Phi_{\theta}=\{\alpha\in\Phi\mid \theta x_{\alpha}(1)\theta^{-1}=x_{\alpha}(1)\}=\{\alpha\in \Phi\mid \text{ the $\alpha$-wall of $\cA_0$ is thick}\}.
$$
It is easy to check that $\Phi_{\theta}$ is indeed a crystallographic root system, and note that $W_{\theta}$ is generated by the reflections in the hyperplanes perpendicular to the roots in~$\Phi_{\theta}$.

\begin{example} Consider the homology $\theta=h_{\omega_5}(c)h_{\omega_6}(c^{-2})$ with $c^2\neq 1$ of an $\sE_6$ building. Then $\alpha\in\Phi_{\theta}$ if and only if $c^{\langle\alpha,\omega_5\rangle-2\langle\alpha,\omega_6\rangle}=1$, and since $c^2\neq 1$, inspection of the root system shows that $\Phi_{\theta}^+$ consists  precisely the roots of the form $(****00)$ (there are $10$ such roots) or $(****21)$ (there are $5$ such roots). Hence $|\Phi_{\theta}^+|=15$. Scharlau's classification~\cite[Proposition~2]{Sch:87} (see below) forces $\Delta_{\theta}'$ to have type $\sA_5$, lying inside of a maximal reflection subgroup of type $\sA_5\times\sA_1$. To see this, note that $\sA_2\times\sA_2\times\sA_2$ has only $9$ reflections, eliminating this case. Then $\sD_5$ has $20$ reflections, however the maximal reflection subgroups of $\sD_5$ have types $\sD_3\times\sD_2$ ($6+2=8$ reflections), $\sD_4$ ($12$ reflections), or $\sA_4$ ($10$ reflections), eliminating these possibilities.
\end{example}

Let us now outline our strategy for classifying domestic homologies. Let $\theta$ be a homology, and after conjugating we assume that $\theta\in H$, and so the base apartment~$\cA_0$ of $\Delta$ is fixed by~$\theta$. Let $W_{\theta}$ be the reflection subgroup of $W$ generated by the thick walls of~$\cA_0$. This subgroup in turn lies in a maximal reflection subgroup $W_{\theta}'$ of $W$. By \cite[Proposition~2]{Sch:87} each maximal reflection subgroup is determined up to conjugation by its type, and hence we may assume the root system of $W_{\theta}'$ has simple roots as listed below for the cases that we will require (where $\sD_2=\sA_1\times\sA_1$, and $\sD_3=\sA_3$ in the natural way): 
\begin{compactenum}[$(1)$]
\item The maximal reflection subgroups of $\sB_4$ are
\begin{compactenum}[$(a)$]
\item $\sB_3\times\sA_1$, with simple roots $(\alpha_2,\alpha_3,\alpha_4)\times(\alpha_1+\alpha_2+\alpha_3+\alpha_4)$.
\item $\sB_2\times\sB_2$, with simple roots $(\alpha_3,\alpha_4)\times(\alpha_1,\alpha_2+\alpha_3+\alpha_4)$.
\item $\sD_4$, with simple roots $(\alpha_1,\alpha_2,\alpha_3,\alpha_3+2\alpha_4)$.
\end{compactenum}
\item The maximal reflection subgroups of $\sF_4$ are
\begin{compactenum}[$(a)$]
\item $\sB_4$, with simple roots $(\varphi_{\sC_3},\alpha_1,\alpha_2,\alpha_3)$ where $\varphi_{\sC_3}=\alpha_2+2\alpha_3+2\alpha_4$. 
\item $\sC_3\times\sA_1$, with simple roots $(\alpha_4,\alpha_3,\alpha_2)\times(\varphi)$.
\item $\sA_2\times\sA_2$, with simple roots $(\alpha_3,\alpha_4)\times(\alpha_1,\varphi-\alpha_1)$.
\end{compactenum}
\item The maximal reflection subgroups of $\sD_5$ are
\begin{compactenum}[$(a)$]
\item $\sD_3\times\sD_2$, with simple roots $(\alpha_3,\alpha_4,\alpha_5)\times(\alpha_1,\varphi)$. 
\item $\sD_4$, with simple roots $(\alpha_2,\alpha_3,\alpha_4,\alpha_5)$. 
\item $\sA_4$, with simple roots $(\alpha_1,\alpha_2,\alpha_3,\alpha_4)$. 
\end{compactenum}
\item The maximal reflection subgroups of $\sD_6$ are
\begin{compactenum}[$(a)$]
\item $\sD_4\times\sD_2$, with simple roots $(\alpha_3,\alpha_4,\alpha_5,\alpha_6)\times(\alpha_1,\varphi)$. 
\item $\sD_3\times \sD_3$, with simple roots $(\alpha_4,\alpha_5,\alpha_6)\times(\alpha_1,\alpha_2,\varphi-\alpha_1-\alpha_2)$.
\item $\sD_5$, with simple roots $(\alpha_2,\alpha_3,\alpha_4,\alpha_5,\alpha_6)$. 
\item $\sA_5$, with simple roots $(\alpha_1,\alpha_2,\alpha_3,\alpha_4,\alpha_5)$. 
\end{compactenum}
\item The maximal reflection subgroups of $\sE_6$ are
\begin{compactenum}[$(a)$]
\item $\sA_5\times\sA_1$, with simple roots $(\alpha_1,\alpha_3,\alpha_4,\alpha_5,\alpha_6)\times(\varphi)$.
\item $\sA_2\times\sA_2\times\sA_2$, with simple roots $(\alpha_1,\alpha_3)\times(\alpha_5,\alpha_6)\times(\alpha_2,\varphi-\alpha_2)$. 
\item $\sD_5$, with simple roots $(\alpha_1,\alpha_3,\alpha_4,\alpha_2,\alpha_5)$. 
\end{compactenum}
\item The maximal reflection subgroups of $\sE_7$ are
\begin{compactenum}[$(a)$]
\item $\sA_7$, with simple roots $(\alpha_1,\alpha_3,\alpha_4,\alpha_5,\alpha_6,\alpha_7,\varphi_{\sE_6})$. 
\item $\sD_6\times\sA_1$, with simple roots $(\alpha_7,\alpha_6,\alpha_5,\alpha_4,\alpha_2,\alpha_3)\times(\varphi)$. 
\item $\sA_5\times\sA_2$, with simple roots $(\alpha_2,\alpha_4,\alpha_5,\alpha_6,\alpha_7)\times(\alpha_1,\varphi-\alpha_1)$.
\item $\sE_6$, with simple roots $(\alpha_1,\alpha_2,\alpha_3,\alpha_4,\alpha_5,\alpha_6)$. 
\end{compactenum}
\item The maximal reflection subgroups of $\sE_8$ are
\begin{compactenum}[$(a)$]
\item $\sD_8$, with simple roots $(\varphi_{\sE_7},\alpha_8,\alpha_7,\alpha_6,\alpha_5,\alpha_4,\alpha_2,\alpha_3)$. 
\item $\sA_8$, with simple roots $(\varphi-\varphi_{\sA_7},\alpha_1,\alpha_3,\alpha_4,\alpha_5,\alpha_6,\alpha_7,\alpha_8)$. 
\item $\sA_4\times\sA_4$, with simple roots $(\alpha_1,\alpha_3,\alpha_4,\alpha_2)\times(\varphi-\alpha_6-\alpha_7-\alpha_8,\alpha_6,\alpha_7,\alpha_8)$.
\item $\sE_6\times\sA_2$, with simple roots $(\alpha_1,\alpha_2,\alpha_3,\alpha_4,\alpha_5,\alpha_6)\times(\alpha_8,\varphi-\alpha_8)$. 
\item $\sE_7\times\sA_1$, with simple roots $(\alpha_1,\alpha_2,\alpha_3,\alpha_4,\alpha_5,\alpha_6,\alpha_7)\times(\varphi)$. 
\end{compactenum}
\end{compactenum}
In cases (2)--(7) the root system of the reflection subgroup is simply the intersection of the $\ZZ$-span of the simple roots with the ambient root system. In case (1) the intersection of the $\ZZ$-span of the simple roots with the ambient root system is larger than the stated type, and thus one must restrict coefficients in the linear combinations appropriately to obtain the correct type.

Our strategy is as follows.   
\begin{compactenum}
\item[(A)] We first identify a list of reflection subgroups~$W'$ of $W$ such that every homology with $W_{\theta}=W'$ is domestic (we will then call $W'$ domestic).
\item[(B)] We then show that if $W'$ is a reflection subgroup of $W$ not in our list from~(A), then every homology with $W_{\theta}=W'$ is not domestic (we will then call $W'$ non-domestic). 
\end{compactenum}
We prove (A) using the standard technique (Proposition~\ref{prop:standardtechnique}), with some more refined arguments required in the $\sE_7$ and $\sF_4$ cases. Our strategy for proving (B) is via the following lemma.

\begin{lemma}\label{lem:nondomestic}
Let $\theta\in H$ be a homology with root system $\Phi_{\theta}$. Suppose there exist mutually perpendicular roots $\beta_1,\ldots,\beta_k\in\Phi^+\backslash\Phi_{\theta}$. Then
$$
\disp(\theta)\geq M,\quad\text{where $M=\ell(s_{\beta_1}\cdots s_{\beta_k})$}.
$$
Moreover, if $\theta'\in H$ is a homology with root system $\Phi_{\theta'}\subseteq \Phi_{\theta}$ then $\disp(\theta')\geq M$. 
\end{lemma}

\begin{proof}
Consider the chamber $gB=uw_0B$ with $u=x_{\beta_1}(1)\cdots x_{\beta_k}(1)$. Since \mbox{$\beta_1,\ldots,\beta_k\notin\Phi_{\theta}$} we have $\theta u\theta^{-1}=x_{\beta_1}(c_1)\cdots x_{\beta_k}(c_k)$ with $c_1,\ldots,c_k\neq 1$. Moreover, the elements $x_{\beta_1}(a_1),\ldots,x_{\beta_k}(a_k)$, with $a_1,\ldots,a_k\in\FF$, commute with each other (as they are mutually perpendicular), and hence
$$
Bg^{-1}\theta gB=Bw_0^{-1}u^{-1}\theta uw_0B=Bx_{-\beta_1}(c_1-1)\cdots x_{-\beta_k}(c_k-1)B.
$$
Thus by Lemma~\ref{lem:perpendicularroots} we have
$
Bg^{-1}\theta gB=Bs_{\beta_1}\cdots s_{\beta_k}B,
$
and hence $\disp(\theta)\geq \ell(s_{\beta_1}\cdots s_{\beta_k})$. The final statement follows as $\beta_1,\ldots,\beta_k\in \Phi^+\backslash \Phi_{\theta}'$. 
\end{proof}

In particular, note that if Lemma~\ref{lem:nondomestic} is used to prove non-domesticity for a homology~$\theta$ (by finding mutually perpendicular roots $\beta_1,\ldots,\beta_k\in\Phi^+\backslash\Phi_{\theta}$ with $\ell(s_{\beta_1}\cdots s_{\beta_k})$ sufficiently large), then every homology $\theta'$ with $W_{\theta'}$ a reflection subgroup of $W_{\theta}$ is also non-domestic, as $\Phi_{\theta'}\subseteq \Phi_{\theta}$. Thus to prove (B) it suffices to prove non-domesticity for the maximal reflection subgroups in the poset of reflection subgroups of $W$ excluding those in the list from~(A).

%
%
%
We now proceed with our classification of domestic homologies of split exceptional buildings. 
%

\begin{thm}\label{thm:homologyE6}
A nontrivial homology $\theta$ of $\sE_6(\FF)$ is domestic if and only if $\Delta_{\theta}'$ is of type $\sD_5$. Moreover, all such homologies have opposition diagram ${^2}\sE_{6;2}$, and are conjugate to an element of the form $h_{\omega_6}(c)$ with $c\in\FF\backslash\{0,1\}$. 
\end{thm}

\begin{proof}
Let $\theta\in H$. Suppose that $\Delta_{\theta}'$ is of type $\sD_5$. After conjugation, we may assume that $\Phi_{\theta}$ has simple roots $(\alpha_1,\alpha_3,\alpha_4,\alpha_2,\alpha_5)$. The condition $\theta x_{\alpha}(a)\theta^{-1}=x_{\alpha}(a)$ for these simple roots forces $\theta=h_{\omega_6}(c)$ for some $c\neq 1$. We show that $\theta$ is domestic, with diagram ${^2}\sE_{6;2}$. Let $\Phi_6^+=\{\alpha\in\Phi^+\mid \langle \alpha,\omega_6\rangle=1\}$, and let $w_1=w_0w_{\sD_5}$. Let $u\in U^+$, and write $u=u_1u_2$ with $u_1\in U_{\sD_5}^+$ and $u_2\in U_{\Phi_6}^+$. Since $u_1^{-1}\theta u_1=\theta$ and $u_2^{-1}\theta u_2\in U_{\Phi_6}^+\theta$, we have
\begin{align*}
w_1^{-1}w_0^{-1}u^{-1}\theta uw_0w_1&=w_{\sD_5}^{-1}u_2^{-1}\theta u_2w_{\sD_5}\in w_{\sD_5}^{-1}U_{\Phi_6}^+w_{\sD_5}\theta\subseteq B,
\end{align*}
where we have used the fact that $\Phi(w_{\sD_5})=\Phi^+\backslash \Phi_6$. Since $\ell(w_1)=36-20=16$ the standard technique (Proposition~\ref{prop:standardtechnique}) gives $\disp(\theta)\leq 31$. Moreover, since $\theta=h_{\omega_6}(c)$ is not conjugate to a root elation, it must have opposition diagram ${^2}\sE_{6;2}$ (by Theorem~\ref{thm:converse} and the classification of admissible diagrams). 

We now use Lemma~\ref{lem:nondomestic} to show that if $\theta\in H$ is a nontrivial homology with $\Delta_{\theta}'$ not of type $\sD_5$ then $\theta$ is not domestic. As noted above, it is sufficient to consider the maximal elements in the poset of reflection subgroups of $W$ excluding those of type~$\sD_5$. Thus we may assume that $W_{\theta}$ is either a maximal reflection subgroup of $W$ of type $\sA_5\times\sA_1$ or $\sA_2\times\sA_2\times\sA_2$, or a maximal reflection subgroup of the standard $\sD_5$ subgroup of~$W$. Up to conjugation we may suppose that either:
\begin{compactenum}[$(1)$]
\item $\Phi_{\theta}=\sA_5\times\sA_1$ with simple roots $(\alpha_1,\alpha_3,\alpha_4,\alpha_5,\alpha_6)\times(\varphi)$. The $16$ positive roots of $\Phi$ are precisely the roots $\alpha\in\Phi^+$ with $\langle\alpha,\omega_2\rangle\in 2\ZZ$. 
\item $\Phi_{\theta}=\sA_2\times\sA_2\times\sA_2$ with simple roots $(\alpha_1,\alpha_3)\times(\alpha_5,\alpha_6)\times(\alpha_2,\varphi-\alpha_2)$. The $9$ positive roots are precisely the roots $\alpha\in \Phi^+$ with $\langle\alpha,\omega_4\rangle\in 3\ZZ$.
\item $\Phi_{\theta}=\sD_3\times\sD_2$ with simple roots $(\alpha_4,\alpha_2,\alpha_5)\times(\alpha_1,\alpha_1+\alpha_2+2\alpha_3+2\alpha_4+\alpha_5)$. The $8$ positive roots are precisely the roots $\alpha\in\Phi^+$ with both $\langle\alpha,\omega_3\rangle\in2\ZZ$ and $\langle\alpha,\omega_6\rangle=0$. 
\item $\Phi_{\theta}=\sD_4$ with simple roots are $(\alpha_3,\alpha_4,\alpha_2,\alpha_5)$. The $12$ positive roots are precisely the roots $\alpha\in \Phi^+$ with both $\langle\alpha,\omega_1\rangle=0$ and $\langle\alpha,\omega_6\rangle=0$. 
\item $\Phi_{\theta}=\sA_4$ with simple roots are $(\alpha_1,\alpha_3,\alpha_4,\alpha_2)$. The $10$ positive roots are precisely the roots $\alpha\in \Phi^+$ with both $\langle\alpha,\omega_5\rangle=0$ and $\langle\alpha,\omega_6\rangle=0$. 
\end{compactenum}
In cases (1), (2), (3), and (5) we have $\alpha,\beta,\gamma\in\Phi^+\backslash\Phi_{\theta}$, where $\alpha,\beta,\gamma$ are the mutually perpendicular roots $\alpha=(112221)$, $\beta=(111211)$ and $\gamma=(011210)$. Thus by Lemma~\ref{lem:nondomestic} we have $\disp(\theta)\geq \ell(s_{\alpha}s_{\beta}s_{\gamma})$. Moreover, note that the roots $\alpha$, $\beta$, $\gamma$, $\alpha_2$ are mutually perpendicular and invariant under the diagram automorphism of $\sE_6$, and so by Lemma~\ref{lem:longeltE6} we have $s_{\alpha}s_{\beta}s_{\gamma}s_2=w_0$. Thus $\disp(\theta)\geq 35$, and so $\theta$ is not domestic.

In case (4) we have $\alpha',\beta',\gamma'\in\Phi^+\backslash \Phi_{\theta}$, where $\alpha',\beta',\gamma'$ are the mutually perpendicular roots $\alpha'=\varphi$, $\beta'=(001111)$, and $\gamma'=(101110)$. Thus $\disp(\theta)\geq \ell(w)$, where $w=s_{\alpha'}s_{\beta'}s_{\gamma'}$. It is easy to see that $\Phi(w)=\Phi^+\backslash\{\alpha_1,\alpha_4,\alpha_6\}$, and hence $\ell(w)=33$. It follows that $\theta$ is not domestic. 
\end{proof}
%

\begin{lemma}\label{lem:homologyE7}
Let $c\in \FF\backslash\{0,1\}$. The homology $\theta=h_{\omega_1}(c)$ of $\sE_7(\FF)$ is domestic, with opposition diagram $\sE_{7;4}$.
\end{lemma}

\begin{proof} We first prove directly that $\theta$ is $\{7\}$-domestic. Consider the $\sE_{7,7}(\FF)$ geometry, with point set $G/P_7$. By Corollary~\ref{cor:oppsphere} it is sufficient to show that no point opposite the base point $P_7$ in this geometry is mapped onto an opposite point by~$\theta$. The points opposite $P_7$ are of the form $uyP_7$, where $y=w_{\sE_7}w_{\sE_6}$, and $u\in U_{\Phi(y)}^+$. We are required to show that 
$$
P_7y^{-1}u^{-1}\theta uyP_7\neq P_7yP_7\quad\text{for all $u\in U_{\Phi(y)}^+$}.
$$

Let $\Phi_{1}$ be the subsystem with simple roots $(\alpha_7,\alpha_6,\alpha_5,\alpha_4,\alpha_2,\alpha_3)$. Then $\Phi_{1}\subseteq \Phi_{\theta}$ (equality occurs if $c\neq -1$, while if $c=-1$ then $\Phi_{\theta}=\Phi_{1}\cup \{\pm\varphi\}$). In particular, $\theta x_{\alpha}(a)\theta^{-1}=x_{\alpha}(a)$ for all $\alpha\in \Phi_{1}$. Write $u=u_1u_2$ with $u_1\in U_{\Phi_{1}}^+$ and $u_2\in U_{\Phi\backslash \Phi_{1}}^+$. Then 
$$
P_7y^{-1}u^{-1}\theta uyP_7=P_7y^{-1}u_2^{-1}\theta u_2yP_7.
$$
By the commutator relations we have $u_3=u_2^{-1}\theta u_2\theta^{-1}\in U_{\Phi\backslash\Phi_1}^+$, and since $yP_7=w_0P_7$ we have 
$$
P_7y^{-1}u^{-1}\theta uyP_7=P_7w_0^{-1}u_3w_0P_7=P_7\tilde{u}P_7,
$$
where $\tilde{u}\in U_{\Phi\backslash\Phi_1}^-$. Since $\Phi(s_{\varphi})=\Phi^+\backslash\Phi_1^+$ we have
\begin{align}\label{eq:products}
P_7y^{-1}u^{-1}\theta uyP_7=P_7s_{\varphi}(s_{\varphi}^{-1}\tilde{u}s_{\varphi})s_{\varphi}^{-1}P_7\subseteq P_7s_{\varphi}P_7\cdot P_7s_{\varphi}P_7.
\end{align}
It follows from Lemma~\ref{lem:E7symplectic} that $P_7s_{\varphi}P_7\cdot P_7s_{\varphi}P_7\cap P_7yP_7=\emptyset$. 
%
Thus we have shown that $\theta$ is $\{7\}$-domestic, and hence domestic. Since $\theta$ is neither a root elation nor a product of perpendicular root elations it does not have opposition diagram $\sE_{7;1}$ or $\sE_{7;2}$, and since $\theta$ is $\{7\}$-domestic it does not have opposition diagram $\sE_{7;3}$. Thus $\theta$ has opposition diagram $\sE_{7;4}$. 
\end{proof}

\goodbreak

\begin{thm}\label{thm:homologyE7}
A nontrivial homology $\theta$ of $\sE_7(\FF)$ is domestic if and only if $\Delta_{\theta}'$ is of type:
\begin{compactenum}[$(1)$]
\item $\sE_6$, in which case $\theta$ has opposition diagram $\sE_{7;3}$ and is conjugate to an element of the form $h_{\omega_7}(c)$ with $c\in\FF\backslash\{0,1\}$;
\item $\sD_6$, in which case $\theta$ has opposition diagram $\sE_{7;4}$ and is conjugate to an element of the form $h_{\omega_1}(c)$ with $c\in\FF\backslash\{0,1,-1\}$;
\item $\sD_6\times\sA_1$, in which case $\mathrm{char}(\FF)\neq 2$ and $\theta$ has opposition diagram $\sE_{7;4}$ and is conjugate to~$h_{\omega_1}(-1)$. 
\end{compactenum}
\end{thm}

\begin{proof}
We begin by showing that if $\Delta_{\theta}'$ has type $\sE_6$, $\sD_6$, or $\sD_6\times\sA_1$ then $\theta$ is domestic, with the stated diagram and conjugacy class. 

Suppose that $\Phi_{\theta}$ is of type $\sE_6$. The $36$ positive roots of this system are precisely those $\alpha\in\Phi^+$ with $\langle\alpha,\omega_7\rangle=0$, and thus $\theta=h_{\omega_7}(c)$ for some $c\neq 1$. Let $\Phi_7^+=\{\alpha\in\Phi^+\mid \langle \alpha,\omega_7\rangle=1\}$, and note that $\Phi^+=\Phi_{\theta}\sqcup\Phi_7^+$. Arguing as in the $\sE_6$ case (Theorem~\ref{thm:homologyE6}), and using the standard technique with $w_1=w_0w_{\sE_6}$ we see that $\disp(\theta)\leq 53$ (as $\ell(w_1)=27$). Since $\theta$ is neither conjugate to a root elation nor to a product of two perpendicular root elations it does not have opposition diagram $\sE_{7;1}$ or $\sE_{7;2}$, and since $\disp(\theta)<60$ it does not have opposition diagram $\sE_{7;4}$.  Thus it follows from the classification of opposition diagrams that $\theta$ has diagram~$\sE_{7;3}$.

Now suppose that $\Phi_{\theta}$ is of type $\sD_6$ or $\sD_6\times \sA_1$. The positive roots of $\Phi_{\theta}$ are
\begin{align*}
\Phi_{\sD_6}^+=\{\alpha\in\Phi^+\mid \langle\alpha,\omega_1\rangle=0\}\quad\text{and}\quad 
\Phi_{\sD_6\times\sA_1}^+=\{\alpha\in\Phi^+\mid \langle\alpha,\omega_1\rangle\in 2\ZZ\}=\Phi_{\sD_6}^+\cup\{\varphi\}.
\end{align*}
In particular, $\theta=h_{\omega_1}(c)$ where $c\in\FF\backslash\{0,1,-1\}$ in the $\sD_6$ case, and $c=-1$ in the $\sD_6\times\sA_1$ case. Thus by Lemma~\ref{lem:homologyE7} $\theta$ is domestic with opposition diagram~$\sE_{7;4}$ (note that the standard technique does not apply in this case, as $\ell(w_0w_{\sD_6})=33>\ell(w_0)/2$).

We now use Lemma~\ref{lem:nondomestic} to show that if $\theta\in H$ is a nontrivial homology with $\Delta_{\theta}'$ not of type $\sE_6$, $\sD_6$, or $\sD_6\times\sA_1$ then $\theta$ is not domestic. It is sufficient to consider the maximal elements in the poset of reflection subgroups of $W$ excluding those of types~$\sE_6$, $\sD_6$, and $\sD_6\times\sA_1$. Thus we may assume that $W_{\theta}$ is either a maximal reflection subgroup of $W$ of type $\sA_7$ or $\sA_5\times\sA_2$, or a maximal reflection subgroup of the standard $\sD_6\times\sA_1$ of $\sE_6$ subgroups of~$W$ (excluding the domestic $\sD_6$ subgroup). Thus, using the explicit choices of simple roots listed earlier, we may suppose that either:
\begin{compactenum}[$(1)$]
\item $\Phi_{\theta}=\sA_7$, with $\Phi_{\theta}^+=\{\alpha\in \Phi^+\mid \langle\alpha,\omega_2\rangle\in 2\ZZ\}$. 
\item $\Phi_{\theta}=\sA_5\times\sA_2$, with $\Phi_{\theta}^+=\{\alpha\in\Phi^+\mid \langle\alpha,\omega_3\rangle\in 3\ZZ\}$. 
\item $\Phi_{\theta}=(\sD_4\times\sD_2)\times\sA_1$, with $\Phi_{\theta}^+=\{\alpha\in\Phi^+\mid \langle\alpha,\omega_1\rangle=0\text{ and }\langle\alpha,\omega_6\rangle\in 2\ZZ\}\cup\{\varphi\}$. 
\item $\Phi_{\theta}=(\sD_3\times\sD_3)\times \sA_1$, with $\Phi_{\theta}^+=\{\alpha\in \Phi^+\mid\langle\alpha,\omega_1\rangle=0\text{ and }\langle\alpha,\omega_5\rangle\in 2\ZZ\}\cup\{\varphi\}$.
\item $\Phi_{\theta}=\sD_5\times\sA_1$, with $\Phi_{\theta}^+=\{\alpha\in\Phi^+\mid \langle\alpha,\omega_1-2\omega_7\rangle=0\}$. 
\item $\Phi_{\theta}=\sA_5\times \sA_1$ (contained in $\sD_6\times\sA_1$), with $\Phi_{\theta}^+=\{\alpha\in\Phi^+\mid 3\langle\alpha,\omega_1\rangle=2\langle\alpha,\omega_3\rangle\}$. 
\item $\Phi_{\theta}=\sA_5\times \sA_1$ (contained in $\sE_6$), with $\Phi_{\theta}^+=\{\alpha\in \Phi^+\mid  \langle\alpha,\omega_2\rangle\in 2\ZZ\text{ and }\langle\alpha,\omega_7\rangle=0\}$. 
\item $\Phi_{\theta}=\sA_2\times\sA_2\times\sA_2$, with $\Phi_{\theta}^+=\{\alpha\in\Phi^+\mid \langle\alpha,\omega_4\rangle\in 3\ZZ\text{ and }\langle\alpha,\omega_7\rangle=0\}$. 
\item $\Phi_{\theta}=\sD_5$, with $\Phi_{\theta}^+=\{\alpha\in \Phi^+\mid \langle\alpha,\omega_6\rangle=0\text{ and }\langle\alpha,\omega_7\rangle=0\}$. 
\end{compactenum}
In case (1) note that $\beta_1,\ldots,\beta_7\notin\Phi_{\theta}$, where $\beta_1=(1 1 1 1 1 0 0)$, $\beta_2=(0 1 1 2 1 0 0)$, $\beta_3=(0 1 1 1 1 1 0)$, $\beta_4=(0 1 0 1 1 1 1)$, $\beta_5=(1 1 1 2 1 1 0)$, $\beta_6=(1 1 2 2 1 1 1)$, $\beta_7=(1 1 2 3 3 2 1)$. These roots are mutually perpendicular, and so by Lemma~\ref{lem:longelt} we have $s_{\beta_1}\cdots s_{\beta_7}=w_0$. Thus by Lemma~\ref{lem:nondomestic} the homology $\theta$ is not domestic. 

In cases (2)--(6) note that $\gamma_1,\ldots,\gamma_5\notin\Phi_{\theta}$, where $\gamma_1=(1 1 1 1 0 0 0)$, $\gamma_2=(1 0 1 1 1 1 1)$, $\gamma_3=(0 1 1 2 1 1 1)$, $\gamma_4=(1 1 2 3 2 1 0)$, $\gamma_5=(1 2 2 3 3 2 1)$. The roots $\gamma_1,\ldots,\gamma_5,\alpha_3,\alpha_6$ are mutually perpendicular, and so $s_{\gamma_1}\cdots s_{\gamma_5}=w_0s_3s_6$. Thus $\disp(\theta)\geq 61$, and so $\theta$ is not domestic. 

In cases (7)--(9) note that $\delta_1,\ldots,\delta_5\notin\Phi_{\theta}$, where $\delta_1=(0 0 0 0 0 1 1)$, $\delta_2=(0 1 0 1 1 1 0)$, $\delta_3=(1 1 2 2 1 1 0)$, $\delta_4=(1 2 2 3 2 1 1)$, $\delta_5=(1 1 2 3 3 2 1)$. The roots $\delta_1,\ldots,\delta_5,\alpha_1,\alpha_4$ are mutually perpendicular, and so as above $\disp(\theta)\geq 61$ and so $\theta$ is not domestic. 
\end{proof}

\begin{thm}\label{thm:homologyE8}
A nontrivial homology $\theta$ of $\sE_8(\FF)$ is domestic if and only if $\Delta_{\theta}'$ is of type:
\begin{compactenum}[$(1)$]
\item $\sE_7$, in which case $\theta$ is conjugate to an element $h_{\omega_8}(c)$ with $c\in\FF\backslash\{0,1,-1\}$;\item $\sE_7\times\sA_1$, in which case $\mathrm{char}(\FF)\neq 2$ and $\theta$ is conjugate to $h_{\omega_8}(-1)$.
\end{compactenum}
In both cases $\theta$ has opposition diagram $\sE_{8;4}$. 
\end{thm}

\begin{proof} Let $\theta\in H$. Suppose that $\Delta_{\theta}'$ has type $\sE_7$ or $\sE_7\times\sA_1$. After conjugating we may suppose that $\Phi_{\sE_7}\subseteq \Phi_{\theta}$, where $\Phi_{\sE_7}$ is the standard $\sE_7$ subsystem of $\Phi$. Thus $\theta=h_{\omega_8}(c)$ for some $c\in\FF\backslash\{0,1\}$. As in the $\sE_6$ case (Theorem~\ref{thm:homologyE6}), using the standard technique with $w_1=w_0w_{\sE_7}$ we see that $\disp(\theta)\leq 113$ (as $\ell(w_1)=57$). Since $\theta$ is neither conjugate to a root elation nor to a product of two perpendicular root elations it does not have opposition diagram $\sE_{8;1}$ or $\sE_{8;2}$, and hence $\theta$ has opposition diagram~$\sE_{8;4}$. 

We now use Lemma~\ref{lem:nondomestic} to show that if $\theta\in H$ is a nontrivial homology with $\Delta_{\theta}'$ not of type $\sE_7$ or $\sE_7\times\sA_1$ then $\theta$ is not domestic. It is sufficient to consider the maximal elements in the poset of reflection subgroups of $W$ excluding those of types~$\sE_7$ and $\sE_7\times\sA_1$. Thus we may assume that $W_{\theta}$ is either a maximal reflection subgroup of $W$ of type $\sD_8$, $\sA_8$, $\sA_4\times\sA_4$, or $\sE_6\times\sA_2$, or a maximal reflection subgroup of the $\sE_7\times\sA_1$ subgroup of~$W$ (excluding the domestic $\sE_7$ subgroup). Thus, using the explicit choices of simple roots listed earlier, we may suppose that either:
\begin{compactenum}[$(1)$]
\item $\Phi_{\theta}=\sD_8$, with $\Phi_{\theta}^+=\{\alpha\in \Phi^+\mid \langle\alpha,\omega_1\rangle\in 2\ZZ\}$.
\item $\Phi_{\theta}=\sA_8$, with $\Phi_{\theta}^+=\{\alpha\in\Phi^+\mid \langle\alpha,\omega_2\rangle\in 3\ZZ\}$.
\item $\Phi_{\theta}=\sA_4\times \sA_4$, with $\Phi_{\theta}^+=\{\alpha\in \Phi^+\mid \langle\alpha,\omega_5\rangle\in 5\ZZ\}$. 
\item $\Phi_{\theta}=\sE_6\times \sA_2$, with $\Phi_{\theta}^+=\{\alpha\in \Phi^+\mid \langle\alpha,\omega_7\rangle\in 3\ZZ\}$. 
\item $\Phi_{\theta}=\sA_7\times\sA_1$, with $\Phi_{\theta}^+=\{\alpha\in\Phi^+\mid \langle\alpha,\omega_2\rangle=2\langle\alpha,\omega_7\rangle\text{ and }\langle\alpha,\omega_8\rangle=0\}\cup\{\varphi\}$.
\item $\Phi_{\theta}=(\sD_6\times \sA_1)\times\sA_1$, with $\Phi_{\theta}^+=\{\alpha\in\Phi^+\mid\langle\alpha,\omega_1\rangle\in 2\ZZ\text{ and }\langle\alpha,\omega_8\rangle=0\}\cup\{\varphi\}$. 
\item $\Phi_{\theta}=(\sA_5\times\sA_2)\times\sA_1$, with $\Phi_{\theta}^+=\{\alpha\in \Phi^+\mid\langle\alpha,\omega_3\rangle\in 3\ZZ\text{ and }\langle\alpha,\omega_8\rangle=0\}\cup\{\varphi\}$.
\item $\Phi_{\theta}=\sE_6\times \sA_1$, with $\Phi_{\theta}^+=\{\alpha\in \Phi^+\mid \langle\alpha,\omega_7\rangle=\langle\alpha,\omega_8\rangle=0\}\cup\{\varphi\}$.
\end{compactenum}
In case (1) we have $\beta_1,\ldots,\beta_8\notin\Phi_{\theta}$, where $\beta_1=(1 0 1 1 0 0 0 0)$, $\beta_2=(1 1 1 2 2 1 0 0)$, $\beta_3=(1 1 2 2 2 2 1 0)$, $\beta_4=(1 1 2 2 2 1 1 1)$, $\beta_5=(1 2 2 3 2 1 1 0)$, $\beta_6=(1 1 1 2 2 2 2 1)$, $\beta_7=(1 2 2 3 2 2 1 1)$, $\beta_8=(1 2 3 5 4 3 2 1)$. These roots are mutually perpendicular, and so by Lemmas~\ref{lem:longelt}~and~\ref{lem:nondomestic} we see that $\theta$ is not domestic. 

In cases (2)--(4) we have $\gamma_1,\ldots,\gamma_6\notin\Phi_{\theta}$, where $\gamma_1=(1 1 1 2 1 1 1 0)$, $\gamma_2=(1 1 2 2 1 1 1 1)$, $\gamma_3=(0 1 1 2 2 2 2 1)$, $\gamma_4=(1 2 3 4 3 2 1 0)$, $\gamma_5=(1 2 2 4 3 2 1 1)$, $\gamma_6=(2 2 3 4 4 3 2 1)$. The roots $\gamma_1,\ldots,\gamma_6,\alpha_2,\alpha_6$ are mutually perpendicular, and so $\disp(\theta)\geq \ell(w_0)-2=118$. Hence $\theta$ is not domestic.  

Similarly, in cases (5)--(8) we have $\delta_1,\ldots,\delta_6\notin\Phi_{\theta}$, where $\delta_1=(0 0 0 0 0 0 1 1)$, $\delta_2=(1 1 1 1 1 1 1 0)$, $\delta_3=(1 1 2 3 2 1 1 0)$, $\delta_4=(1 1 2 3 3 3 2 1)$, $\delta_5=(2 2 3 4 3 2 1 1)$, $\delta_6=(1 3 3 5 4 3 2 1)$. The roots $\delta_1,\ldots,\delta_6,\alpha_3,\alpha_5$ are mutually perpendicular, and so again $\theta$ is not domestic. 
\end{proof}

\begin{lemma}\label{lem:homologyF4}
Let $\mathrm{char}(\FF)\neq 2$. The homology $\theta=h_{\omega_4}(-1)$ of $\sF_4(\FF)$ is domestic, with opposition diagram $\sF_{4;1}^4$. 
\end{lemma}

\begin{proof}
It is sufficient to prove that $\theta$ is $1$-domestic. Recall that $G/P_1$ is the set of points of the Lie incidence geometry $\sF_{4,1}(\FF)$. By Corollary~\ref{cor:oppsphere} it is sufficient to show that no point opposite $P_1$ is mapped to an opposite point by $\theta$. Since $w_0w_{\sC_3}=s_{\varphi}$ (by comparing inversion sets) we have $P_1w_0P_1=P_1s_{\varphi}P_1$, and hence the points opposite $P_1$ are of the form $us_{\varphi}P_1$ with $u\in U_{\Phi(s_{\varphi})}^+$. Thus we are required to prove that 
$$
P_1s_{\varphi}^{-1}u^{-1}\theta u s_{\varphi}P_1\neq P_1s_{\varphi}P_1\quad\text{for all $u\in U_{\Phi(s_{\varphi})}^+$}.
$$
We have $\Phi_{\theta}=\{\alpha\in\Phi\mid\langle \alpha,\omega_4\rangle\in 2\ZZ\}$. Write $u=u_1u_2$ with $u_1\in  U_{\Phi_{\theta}\cap\Phi(s_{\varphi})}^+$ and $u_2\in U_{\Phi(s_{\varphi})\backslash\Phi_{\theta}}^+$, and so
$
P_1s_{\varphi}^{-1}u^{-1}\theta us_{\varphi}P_1=P_1s_{\varphi}^{-1}u_2^{-1}\theta u_2s_{\varphi}P_1.
$
Now, $\Phi(s_{\varphi})=\{\alpha\in\Phi^+\mid \langle\alpha,\omega_1\rangle\geq 1\}$, and hence by inspection of the root system we have
$$
\Phi(s_{\varphi})\backslash\Phi_{\theta}=\{\alpha,\beta,\gamma,\delta\}\quad\text{where $\alpha=(1111)$, $\beta=(1121)$, $\gamma=(1221)$, $\delta=(1231)$}.
$$
Writing $u_2=x_{\alpha}(a)x_{\beta}(b)x_{\gamma}(c)x_{\delta}(d)$, a commutator relation calculation gives
\begin{align*}
u_2^{-1}\theta u_2&=x_{\delta}(-d)x_{\gamma}(-c)x_{\beta}(-b)x_{\alpha}(-a)x_{\alpha}(-a)x_{\beta}(-b)x_{\gamma}(-c)x_{\delta}(-d)\theta\\
&=x_{\alpha}(-2a)x_{\beta}(-2b)x_{\gamma}(-2c)x_{\delta}(-2d)x_{\varphi}(-4ad+4bc)\theta.
\end{align*}
Thus $P_1s_{\varphi}^{-1}u^{-1}\theta us_{\varphi}P_1=P_1x_{-\delta}(-2a)x_{-\gamma}(2b)x_{-\beta}(-2c)x_{-\alpha}(2d)x_{-\varphi}(4ad-4bc)P_1$. Easy calculations, using the folding relation and commutator relations, show that 
\begin{align*}
x_{-\delta}(-2a)x_{-\gamma}(2b)x_{-\beta}(-2c)x_{-\alpha}(2d)x_{-\varphi}(4ad-4bc)&\in\begin{cases}
Bs_{\delta}B&\text{if $a\neq 0$}\\
Bs_{\gamma}B&\text{if $a=0$ and $b\neq 0$}\\
Bs_{\beta}B&\text{if $a=b=0$ and $c\neq 0$}\\
Bs_{\alpha}B&\text{if $a=b=c=0$ and $d\neq 0$}\\
B&\text{if $a=b=c=d=0$.}
\end{cases}
\end{align*}
Since $\beta=s_3\alpha$, $\gamma=s_2s_3\alpha$, and $\delta=s_3s_2s_3\alpha$, we have $s_{\beta}=s_3s_{\alpha}s_3$, $s_{\gamma}=s_2s_3s_{\alpha}s_3s_2$ and $s_{\delta}=s_3s_2s_3s_{\alpha}s_3s_2s_3$, and hence $P_1s_{\alpha}P_1=P_1s_{\beta}P_1=P_1s_{\gamma}P_1=P_1s_{\delta}P_1$. Thus
$$
P_1s_{\varphi}^{-1}u^{-1}\theta us_{\varphi}P_1=\begin{cases}
P_1s_{\alpha}P_1&\text{if $(a,b,d,c)\neq (0,0,0,0)$}\\
P_1&\text{if $(a,b,c,d)=(0,0,0,0)$}.
\end{cases}
$$
We have $P_1s_{\alpha}P_1\neq P_1s_{\varphi}P_1$. To see this, note that $s_{\alpha}=s_1s_2s_3s_4s_3s_2s_1$ (as $\alpha=s_1s_2s_3\alpha_4$), and so both $s_{\alpha}$ and $s_{\varphi}$ are minimal length in their respective $W_1$-double cosets (we noted above that $s_{\varphi}=w_0w_{\sC_3}$). But of course $s_{\alpha}\neq s_{\varphi}$. Hence $\theta$ is $\{1\}$-domestic. 
\end{proof}

\begin{thm}\label{thm:homologyF4}
Let $\theta$ be a nontrivial homology of $\sF_4(\FF)$. If $\mathrm{char}(\FF)=2$ then $\theta$ is not domestic. If $\mathrm{char}(\FF)\neq 2$ then $\theta$ is domestic if and only if $\Delta_{\theta}'$ has type $\sB_4$. Moreover, all such homologies have opposition diagram $\sF_{4;1}^4$ and are conjugate to $h_{\omega_4}(-1)$. 
\end{thm}

\begin{proof}
Let $\theta\in H$, and suppose that $\Delta_{\theta}'$ has type $\sB_4$. The simple roots of $\Phi_{\theta}$ are $
(\beta,\alpha_1,\alpha_2,\alpha_3)$ where $\beta=\alpha_2+2\alpha_3+2\alpha_4$. Thus $\theta=h_{\omega_4}(-1)$ (in particular $\mathrm{char}(\FF)\neq 2$). By Lemma~\ref{lem:homologyF4} $\theta$ is domestic with diagram~$\sF_{4;1}^4$. 

We now use Lemma~\ref{lem:nondomestic} to show that if $\theta\in H$ is a nontrivial homology with $\Delta_{\theta}'$ not of type $\sB_4$ then $\theta$ is not domestic. It is sufficient to consider the cases:
\begin{compactenum}[$(1)$]
\item $\Phi_{\theta}=\sC_3\times\sA_1$, simple roots $(\alpha_4,\alpha_3,\alpha_2)\times(\varphi)$, and $\Phi_{\theta}^+=\{\alpha\in\Phi^+\mid \langle\alpha,\omega_1\rangle\in 2\ZZ\}$.
\item $\Phi_{\theta}=\sA_2\times\sA_2$, simple roots $(\alpha_3,\alpha_4)\times(\alpha_1,\varphi-\alpha_1)$, and $\Phi_{\theta}^+=\{\alpha\in\Phi^+\mid \langle\alpha,\omega_2\rangle\in 3\ZZ\}$. 
\item $\Phi_{\theta}=\sB_3\times\sA_1$, simple roots $(\alpha_1,\alpha_2,\alpha_3)\times (\varphi')$, and $\Phi_{\theta}^+=\{\alpha\in \Phi^+\mid \langle\alpha,\omega_4\rangle=0\}\cup\{\varphi'\}$.
\item $\Phi_{\theta}=\sB_2\times\sB_2$, simple roots $(\alpha_2,\alpha_3)\times(\beta,\gamma)$, where $\beta=\alpha_2+2\alpha_3+2\alpha_4$ and $\gamma=\alpha_1+\alpha_2+\alpha_3$, and $\Phi_{\theta}^+=\{\alpha_2,\alpha_3,\alpha_2+\alpha_3,\alpha_2+2\alpha_3,\beta,\gamma,\varphi',\varphi\}$.
\item $\Phi_{\theta}=\sD_4$, simple roots $(\alpha_2+2\alpha_3+2\alpha_4,\alpha_1,\alpha_2,\alpha_2+2\alpha_3)$, and 
$\Phi_{\theta}^+=\{\alpha\in\Phi^+\mid \langle\alpha,\omega_4\rangle\in 2\ZZ\text{ and }\langle\alpha,\omega_3\rangle\in 2\ZZ\}$. 
\end{compactenum}
In cases (1) and (2) we have $\beta_1,\beta_2,\beta_3\notin\Phi_{\theta}$, where $\beta_1=(1220)$, $\beta_2=(1222)$, and $\beta_3=(1242)$. The roots $\beta_1,\beta_2,\beta_3,\alpha_1$ are mutually perpendicular, and hence $s_{\beta_1}s_{\beta_2}s_{\beta_3}=w_0s_1$ by Lemma~\ref{lem:longelt}, and so by Lemma~\ref{lem:nondomestic} $\theta$ is not domestic. 

In cases (3), (4) and (5) we have $\gamma_1,\gamma_2,\gamma_3,\gamma_4\notin\Phi_{\theta}$, where $\gamma_1=(0011)$, $\gamma_2=(0111)$, $\gamma_3=(1111)$, and $\gamma_4=(1231)$. Thus again $\theta$ is not domestic. 
\end{proof}

For the $\sG_2$ results see Theorem~\ref{thm:mainG2}.

\section{The polar-copolar type for $\sE_7$, $\sE_8$, and $\sF_4$}\label{sec:polarcopolar}

In this section we prove Theorem~\ref{thm:pocopolarclass}, classifying the automorphisms with polar-copolar diagram in types $\sE_7$, $\sE_8$, and $\sF_4$. The arguments here are of a more geometric flavour, working in long root geometries and the metasymplectic space $\sF_{4,4}(\FF)$. See Subsection~\ref{sec:parapolar} for some relevant terminology, and see~\cite[Chapter~13]{Sch:10} for further details.

Recall, from \cite{PVM:19a}, that if $x$ is a simplex of a spherical building mapped onto an opposite simplex by an automorphism $\theta$, then we write $\theta_{x}$ for the automorphism of the residue $\mathrm{Res}(x)$ given by $\theta_{x}=\mathrm{proj}_{x}\circ\theta$, where $\mathrm{proj}_{x}$ is the projection from $\mathrm{Res}(x^{\theta})$ onto $\mathrm{Res}(x)$. If $\xi$ is a simplex of $\mathrm{Res}(x)$, then $\xi$ and $\xi^{\theta}$ are opposite in $\Delta$ if and only if $\xi$ and $\xi^{\theta_{x}}$ are opposite in the building $\mathrm{Res}(x)$ (see \cite[Proposition~3.29]{Tit:74}). 

%

In this section we prove Theorem~\ref{thm:pocopolarclass}. Our first task is to prove the following theorem (giving the `only if' direction of Theorem~\ref{thm:pocopolarclass}). 

\begin{thm}\label{polar-copolar}
A collineation belonging to a polar-copolar opposition diagram of a split exceptional building $\Delta$ not of type $\sE_6$ if $\Delta$ is large, and not of type $\sE_7$ or $\sE_8$ if $\Delta$ is small,  is the product of two orthogonal root elations. That is, the centres of the elations form a symplectic pair of points in the corresponding long root geometry.
\end{thm}

Large buildings of type $\sE_6$ are true exceptions to the theorem, for there exist homologies with the diagram ${^2}\sE_{6;2}$ (see Theorem~\ref{thm:homologyE6}). For the small buildings of type $\sE_i$, $i=7,8$, we believe that the theorem still holds, however the geometric arguments break down.

The theorem follows from a series of lemmas. First we need some properties of long root elations in the long root geometries $\sE_{7,1}(\FF)$, $\sD_{6,2}(\FF)$, $\sA_{5,\{1,5\}}(\FF)$ and $\sC_{3,1}(\FF)$. In these Lie incidence geometries, a vertex of the corresponding building of type $7$, $6$, $3$ and $3$, respectively, defines a (residual) sub-Lie incidence geometry, which we shall call a \emph{para}, of type $\sE_{6,1}(\FF)$, $\sA_{5,2}(\FF)$, $\sA_{2,1}\times \sA_{2,1}(\FF)$, $\sA_{2,1}(\FF)$, respectively. The paras are also the points of the Lie incidence geometry $\sE_{7,7}(\FF)$, $\sD_{6,6}(\FF)$, $\sA_{5,3}(\FF)$ and $\sC_{3,3}(\FF)$, respectively. We call two paras \emph{adjacent} if the corresponding points in the latter Lie incidence geometries are collinear. A \emph{pencil} of paras corresponds to a line of that geometry. Note also that a long root elation fixes all points collinear and symplectic to a certain (unique) point $c$ of the long root geometry, and $c$ is called the \emph{centre} of the elation. 

\begin{lemma}\label{paraargument}
A long root elation $\theta$ of $\sE_{7,1}(\FF)$, $\sD_{6,2}(\FF)$, $\sA_{5,\{1,5\}}(\FF)$ and $\sC_{3,1}(\FF)$ maps each non-fixed para $P$ to an adjacent one, preserving the pencil defined by $P$\! and $P^\theta$\! and fixing exactly one para of that pencil. 
\end{lemma}

\begin{proof}
A  central elation in one of the Lie incidence geometries $\sE_{7,7}(\FF)$, $\sD_{6,6}(\FF)$, $\sA_{5,3}(\FF)$ or $\sC_{3,3}(\FF)$ fixes all points of a ``central'' symp (which corresponds to the centre of the elation), all points collinear to a maximal singular subspace of that symp, and all lines intersecting that symp in a point. Since every point that is not fixed is on exactly one such line, the lemma follows.
\end{proof}

Let $\theta$ be a collineation with an exceptional polar-copolar opposition diagram of a large building. Let $\Delta$ be the corresponding long root geometry and let $(p,\omega)$ be an incident point-symp pair which is mapped onto an opposite by $\theta$. 

Let $E(p,p^\theta)$ be the \textit{equator geometry} of the pair $(p,p^\theta)$, that is, the geometry induced by the points which are symplectic to both $p$ and $p^\theta$. For $\Delta$ of type $\sE_{8,8}$, $\sE_{7,1}$, $\sE_{6,2}$ and $\sF_{4,1}$, note that $E(p,p^\theta)$ is isomorphic to the long root geometry $\mathscr{G}$ of type $\sE_{7,1}$, $\sD_{6,2}$, $\sA_{5,\{1,5\}}$ and $\sC_{3,1}$, respectively. Moreover, symps, planes and lines of $\Delta$ through $p$ correspond to points, symps and paras, respectively, of $\mathscr{G}$. The case $(\sF_{4,1},\sC_{3,1})$ is special in that the lines of $\sC_{3,1}(\K)$ correspond to points of a symp in $E(p,p^\theta)$ (in this case, $E(p,p^\theta)$ does not contain lines; only symplectic and opposite pairs of points). 

\begin{lemma}\label{longrootinequator}
Let $\theta$ and $p$ be as above. Then $\theta$ preserves $E(p,p^\theta)$ and induces a long root elation in it, say with centre the point $c$. 
\end{lemma}

\begin{proof}
Recall that the symbol $\perp\!\!\!\perp$ means ``symplectic to'', and $\Join$ ``special to''. 

Let $\pi$ be a plane through $p$ fixed by $\theta_p$. We claim that the line $\pi\cap (p^\theta)^{\Join}$ is mapped onto the line $\pi^\theta\cap p^{\Join}$. Indeed, first assume that every line $L$ through $p$ in $\pi$ is fixed under $\theta_p$, the alternative being that exactly one such line is fixed (by Lemma~\ref{paraargument}). Let $M$ be the line in $\pi$ such that $M^\theta=\pi^\theta\cap p^{\Join}$. If $M=\pi\cap (p^\theta)^{\Join}$, then there is nothing to prove, so suppose $M$ and $\pi\cap (p^\theta)^{\Join}$ intersect in a unique point $z$. Then, since $(pz)^{\theta_p}=pz$, we see that $z^\theta\perp z$. Now let $L$ be a line in $\pi$ through $p$, but not through $z$. Since $|\FF|>2$, we can select a point $q\in L\setminus(\{p\}\cup M\cup {p^\theta}^{\Join})$. Let $K$ be a line in $\pi$ through $q$, and set $K\cap M=\{u\}$,  $pu\cap (p^\theta)^{\Join}=\{v\}$ and $K\cap (p^\theta)^{\Join}=\{w\}$.   Since $(pv)^{\theta_p}=pv$, we have $v\perp u^\theta$. Hence $(qu)^{\theta_q}=(qw)^{\theta_q}=qv$. This now yields the equivalence $$(qw)^{\theta_q}=qw \Leftrightarrow v=w \Leftrightarrow u=z\mbox{ or } u\in pq.$$ Consequently the central elation $\theta_q$ fixes $\pi$ and exactly two lines through $q$ in $\pi$, which contradicts Lemma~\ref{paraargument}, recalling that the lines through $q$ correspond to the paras in that lemma. 

Next assume that exactly one line $L$ through $p$ in $\pi$ is fixed under $\theta_p$. Since every such fixed line is contained in a fixed plane all of whose lines through $p$ are fixed, we know by the previous paragraph that $z:=L\cap  (p^\theta)^{\Join}$ is mapped onto $z^\theta=L^\theta\cap p^{\Join}$, and these two points are collinear. Let $M$ again be the line of $\pi$ defined by $M^\theta=\pi^\theta\cap p^{\Join}$. Let, for each $x\in M$, $x'$ be the unique point of $\pi$ collinear to $x^\theta$, then, as a product of a linear collineation and a projection, the correspondence $x\mapsto x'$ is a projectivity from $M$ to $M':=\pi\cap  (p^\theta)^{\Join}$. Since $z=M\cap M'$ is fixed under this correspondence, it is a perspectivity. Let $c$ be the centre of this perspectivity, then $c\notin M\cup M'$ is opposite $c^\theta$, and clearly $\theta_c$ is the identity restricted to $\pi$. By the first case, this implies that $M'=M$.

Now let $\xi$ be an arbitrary symp through $p$. Every point of $\xi$ at distance $2$ from $p^\theta$ is collinear to the unique point $e_\xi$ of $\xi$ symplectic to $p^\theta$, which is also the unique point of $\xi$ belonging to $E(p,p^\theta)$.   Hence the previous claim can be formulated as: $\theta$ maps $p^\perp\cap e_\xi^\perp$ to $(p^\theta)^\perp\cap e_{\xi^\theta}^\perp$. If $\xi$ is hyperbolic, or parabolic in uneven characteristic, then $p^\theta$ and $e_{\xi^\theta}$ are the only two points of $\xi^\theta$ collinear to all points of   $(p^\theta)^\perp\cap e_{\xi^\theta}^\perp$; it follows that $e_\xi^\theta=e_{\xi^\theta}$. 

Now suppose that $\xi$ is parabolic in characteristic 2 (the argument also works for symplectic polar spaces in every characteristic). Then we denote by $H[u,v]$ the hyperbolic (or imaginary in the symplectic case) line of $\xi$ defined by two non-collinear points $u,v$, that is, $$H[u,v]=\{x\in\xi\mid x\perp y, \mbox{ for all } y\in u^\perp\cap v^\perp\}=(\{u,v\}^\perp)^\perp.$$ Then the foregoing implies that some point of $H[p,e_\xi]$ is mapped onto $e_{\xi^\theta}$. Now consider any point $q\perp p$ in $\xi$, with $q\notin e_\xi^\perp$. Then, by the same token, some point of $H[q,e_\xi]$ is mapped onto $e_{\xi^\theta}$. Since $H[p,e_\xi]\cap H[q,e_\xi]=\{e_\xi\}$, it again follows $e_\xi^\theta=e_{\xi^\theta}$. 

Hence we have shown that $E(p,p^\theta)$ is preserved under the action of $\theta$ and the lemma now follows from Thereom~\ref{thm:polarclass} for $\sE_7(\K)$,   $\sD_{6}(\K)$ and $\sC_{3}(\K)$. 
\end{proof}

Let $\theta$, $p$ and $\omega$ be as above, and let $p'\in E(p,p^\theta)$ correspond to $\omega$. Then $p'$ is opposite ${p'}^\theta$ and hence $E(p',{p'}^\theta)$ is preserved by $\theta$. Since $p'\in\omega$, we can apply Lemma~\ref{longrootinequator} and obtain that $\theta$ induces a long root elation in $E(p',{p'}^\theta)$, say with centre $c'$. Now set $\theta^*=\theta_{[c]}\theta_{[c']}$, with $\theta_{[c]}$ the central elation with centre $c$ (and similar for $\theta_{[c']}$). Then $\theta^*\theta^{-1}$ fixes every point of $E(p,p^\theta)\cup E(p',{p'}^\theta)$ (use the fact that, in $E(p,p^\theta)$, $(\theta_{[c]})_{p'}$ fixes ${p'}^{\perp\!\!\!\perp}\cap ({p'}^\theta)^{\perp\!\!\!\perp}$, which precisely coincides with $E(p,p^\theta)\cap E(p',{p'}^\theta)$). 

 Now Theorem~\ref{polar-copolar} follows from the following general proposition.
\begin{prop}\label{fixperpequators}
No nontrivial collineation of the long root geometry of (split) type $\sB_n$ ($n\geq 4$), $\sC_n$ ($n\geq 3$), $\sD_n$ $(n\geq 5$), $\sE_n$ ($8\geq n\geq 6$) or $\sF_4$ fixes two perpendicular equator geometries pointwise.  
\end{prop}

\begin{proof}
For the classical types $\sB/\sC/\sD$ this follows from the easy fact that an equator geometry spans a subspace of dimension $n-4/n-2/n-4$ of the ambient projective space $\PG(n,\FF)$, $n\geq 8/5/9$, respectively; hence if a collineation $\theta$ pointwise fixes two perpendicular equator geometries, then it fixes all points of two subspaces of dimension $n-4/n-2/n-4$ of $\PG(n,\FF)$, spanning $\PG(n,\FF)$, $n\geq 8/5/9$, and hence intersecting in at least one point, and so forcing $\theta$ to be the identity. 

Now consider the long root geometries of type $\sE$. Let $E_1,E_2$ be two perpendicular equator geometries, that is, $E_1$ is the equator geometry $E(p_2,q_2)$ for $p_2,q_2\in E_2$ and $E_2=E(p_1,q_1)$ with $p_1,q_1\in E_1$. Since $E_2$ is fixed pointwise, $\Res(p_1)$ is fixed pointwise. Hence every symp $\xi$ through $p_1$ is fixed, and every line in $\xi$ through $p_1$ is fixed. Moreover, the point $\xi\cap E_2$, which is opposite $p_2$ in $\xi$, is fixed. Since $\xi$ is hyperbolic, $\xi$ is fixed pointwise as soon as some line in $\xi$ through $p_1$ is fixed pointwise. By connectivity and the arbitrariness of $\xi$, it suffices that some line through $p_1$ is fixed pointwise, which is the case as $E_1$ contains a line through $p_1$ and is fixed pointwise. Hence all points collinear or symplectic to $p_1$ are fixed and so we have a central elation with centre $p_1$. But the same thing holds for $q_1$ and hence we have the identity.

At last assume we have the long root geometry of type $\sF_4$.   The same argument as in the previous paragraph shows that it suffices to find one line through $p_1$ that is pointwise fixed. To that aim, let $p_1'\in E_1$ be such that $p_1\perp\!\!\!\perp p_1'$. Then the symp $\xi:=\xi(p_1,p_1')$ corresponds to planes $\pi_p$ and $\pi_q$ through the points $p_2$ and $q_2$, respectively.  Then the lines $L:=\pi_p\cap q_2^{\Join}$ and  $M:=\pi_q\cap p_2^{\Join}$ are fixed and the points of $L^\perp\cap M^\perp$ belong to $E_1$. Hence $(p_1^\perp\cap {p_1'}^\perp)\cup(L^\perp\cap M^\perp)$ is fixed pointwise, which implies that $\xi$ is fixed pointwise (note that $p_1,p_1'\in L^\perp\cap M^\perp$).   

The proposition is proved.
\end{proof}

The proof of Theorem~\ref{polar-copolar} is complete, and the following theorem proves the `if' direction of Theorem~\ref{thm:pocopolarclass}. 

\begin{thm}\label{thm:final}
Every product of two perpendicular long root elations in $\sF_4(\FF)$ (respectively $\sE_{7}(\FF)$, $\sE_8(\FF)$) is domestic with opposition diagram~$\sF_{4;2}$ (respectively $\sE_{7;2}$, $\sE_{8;2}$). 
\end{thm}

\begin{proof}
Consider the $\sF_4$ case. We claim that all pairs $(\alpha,\beta)$ of perpendicular long roots are conjugate under~$W$. Since $W$ is transitive on long roots it suffices to show that the stabiliser $W_{\varphi}$ of $\varphi$ in $W$ is transitive on the set of long roots perpendicular to~$\varphi$. These roots are precisely the long roots of the $\sC_3$ subsystem, and since $W_{\varphi}=\langle s_2,s_3,s_4\rangle=W_{\sC_3}$ the result follows. Hence we may assume that $\theta=x_{\varphi}(a)x_{\varphi_{\sC_3}}(b)$, and then the result follows from Theorem~\ref{thm:unipotentdiags}. The arguments for $\sE_7$ and $\sE_8$ are similar. 
\end{proof}

\section{Domestic automorphisms in split types $\sE_6$, $\sF_4$, and $\sG_2$}\label{sec:classifications}

In this section we give the complete classification of domestic automorphisms of split buildings of types $\sE_6$, $\sF_4$, and $\sG_2$. 

\subsection{Classification of domestic automorphisms of split~$\sF_4$}\label{sec:F4C}

In this section we classify domestic automorphisms of split $\sF_4$ buildings. By \cite[Lemma~4.1]{PVM:19a} no duality of a thick $\sF_4$ building is domestic, and so we may restrict to collineations. The complete list of domestic collineations of the small building $\sF_4(\FF_2)$ is given in \cite[Theorem~4.3]{PVM:19b}, and so we may assume that $|\FF|>2$, and so all automorphisms are capped. Thus the possible opposition diagrams of nontrivial domestic collineations are $\sF_{4;1}^1$, $\sF_{4;1}^4$, and $\sF_{4;2}$. The automorphisms with diagram $\sF_{4;1}^1$ are the long root elations (see Theorem~\ref{thm:polarclass}), and the collineations with diagram $\sF_{4;2}$ are products of perpendicular root elations (see Theorem~\ref{thm:pocopolarclass}). Thus our main task is to consider the diagram $\sF_{4;1}^4$. 

We will prove the following theorem:
  
 \begin{thm}\label{thm:F41hom}
 Let $\theta$ be an automorphism of $\Delta_{\sF_4}(\FF)$ with opposition diagram $\sF_{4,1}^4$. If $\mathrm{char}(\FF)=2$ then $\theta$ is a short root elation, and if $\mathrm{char}(\FF)\neq 2$ then $\theta$ is a homology. 
 \end{thm}

Let $\mathscr{G}=\sF_{4,4}(\FF)$ be the short root geometry of the building $\Delta$ of $\sF_4(\FF)$. Then $\mathscr{G}$ is a metasymplectic space with symps isomorphic to a symplectic polar space of rank~3 (see \cite[Chapter 18]{Sch:10} for details). If $\theta$ is a domestic collineation with opposition diagram $\sF_{4;1}^4$ then the only objects mapped to opposite objects are points. 

Recall from \cite[\S5.2]{ADS-NSNS-HVM} that two opposite points $p,q$ in $\mathscr{G}$ define a geometry $\widehat{E}(p,q)$ of type $\sB_4$ as follows: Let $E(p,q)$ denote the equator geometry of $p,q$ (as in Section~\ref{sec:polarcopolar}), and take $\widehat{E}(p,q)$ to be the union of all equator geometries $E(x,y)$ for $x,y$ opposite points in $E(p,q)$. The lines of $\widehat{E}(p,q)$ are the intersections of $\widehat{E}(p,q)$ with the symps $\xi(u,v)$, with $\{u,v\}$ symplectic pairs in $\widehat{E}(p,q)$. These intersections are imaginary lines of symplectic polar spaces. Note that two lines of $\widehat{E}(p,q)$ are opposite (as polar space lines) if and only if the corresponding symps are opposite in $\mathscr{G}$. 

\begin{lemma}\label{geomhypeq}
Let $p$ be a point of $\mathscr{G}$ mapped onto an opposite point. Then $\theta$ pointwise fixes a geometric hyperplane of $\widehat{E}(p,p^\theta)$.
\end{lemma}

\begin{proof}
Let $p$ be a point of $\mathscr{G}$ mapped onto an opposite. Then $\theta_p$ is the identity. 
Since the identity automatically and trivially satisfies the properties of central elations mentioned in Lemma~\ref{paraargument}, it follows from the proof of Lemma~\ref{longrootinequator} that $\theta$ pointwise fixes $E(p,p^\theta)$. Hence $\theta$ stabilizes $\widehat{E}(p,p^\theta)$. Since (opposite) lines of $\widehat{E}(p,p^\theta)$ correspond to (opposite) symps in $\mathscr{G}$, $\theta$ induces a line-domestic collineation in $\widehat{E}(p,p^\theta)$. The assertion now follows from \cite[Theorem 5.1]{TTM:12}.
\end{proof}

In the case that $\mathrm{char}(\FF)=2$ we need a more precise version of the above lemma.

\begin{lemma}\label{perpor3}
Let $p$ be a point mapped onto an opposite. If $\mathrm{char}(\FF)=2$ then $\theta$ pointwise fixes either the perp of a point in $\widehat{E}(p,p^\theta)$, or a subspace of $\widehat{E}(p,p^\theta)$ inducing a non-degenerate polar space of Witt index $3$ in $\widehat{E}(p,p^\theta)$. In both cases, $p^{\theta^2}=p$.  
\end{lemma}

\begin{proof}
In this case, $\widehat{E}(p,p^\theta)$ may be identified with the polar space (of type $\mathsf{B_4}$) obtained from the standard symplectic polar space in $\PG(7,\FF)$ with alternating form $x_0y_1+x_1y_0+\cdots x_6y_7+x_7y_6$ by only considering the points whose coordinates $(x_0,\ldots,x_7)$ satisfy $x_0x_1+x_2x_3+x_4x_5+x_6x_7\in\FF^2$, where the latter is the subfield of squares of $\FF$. Without loss we may assume that $p=(0,0,\ldots,1,0)$ is mapped onto $(0,\ldots,0,1)$. By Lemma~\ref{geomhypeq} and its proof, $\theta$ fixes pointwise a hyperplane of $\PG(7,\FF)$ containing the subspace with equations $X_6=X_7=0$. Hence some point $t=(0,\ldots,0,k,1)$ is fixed, with $k\in\FF$, along with all points having coordinates $(x_0,x_1,\ldots,x_5,k,1)$. If $k\in\FF^2$, then $t\in\widehat{E}(p,p^\theta)$ and $\theta$ fixes the perp of $t$. Then $\theta$ is an elation in $\PG(7,\FF)$ and hence an involution on $\widehat{E}(p,p^\theta)$. 

Now assume $k\notin\FF^2$. Let $a,b,c\in\FF$ be defined by $$\theta:(x_0,\ldots,x_5,x_6,x_7)\mapsto (x_0,\ldots,x_5,ax_7,bx_6+cx_7).$$ Expressing that $(0,\ldots,0,1,k,k,1)$ is fixed, we obtain $a=k$ and $c=kb+1$. Expressing that the image of an arbitrary point of $\widehat{E}(p,p^\theta)$ stays inside of $\widehat{E}(p,p^\theta)$, we get $bk=1$ and so $\theta$ is involutive on $\widehat{E}(p,p^\theta)$. Moreover, we can now project the fixed point set from $(0,\ldots,0,k,1)$ onto the subspace $X_6=X_7=0$ and obtain the polar space whose points have coordinates $(x_0,x_1,\ldots,x_5)$ with $x_0x_1+x_2x_3+x_4x_5\in\FF^2+k\FF^2$. That is a polar space of Witt index~$3$. 
\end{proof}

\begin{lemma}\label{perpcase}
Assume that $\mathrm{char}(\FF)=2$, and suppose that whenever a point $p$ is mapped onto an opposite point, $\theta$ pointwise fixes the perp of a point in $\widehat{E}(p,p^\theta)$. Then $\theta$ is a central elation. 
\end{lemma}

\begin{proof}
Let $p$ be a point mapped onto an opposite, and suppose $\theta$ fixes all points of $\widehat{E}(p,p^\theta)$ symplectic to $x\in\widehat{E}(p,p^\theta)$. Arguments similar to the ones in the proof of Proposition~\ref{fixperpequators} show that all vertices of $\Delta$ incident with $x$ are fixed by $\theta$. Let $y$ be opposite $x$ in $\widehat{E}(p,p^\theta)$. Then $E(x,y)$ is fixed pointwise. Let $y'\perp y$ be also opposite $x$. Then we have a unique path $y,yy',b,ab,a,ax,x$ of consecutively incident points and lines connecting $y$ with $x$ using the line $yy'$. Since $x,y,y^\theta$ are on the same imaginary line, and since $ax$ is fixed, we see that $(ab)^\theta=ab$ (this also follows from the fact that the line $ab$ defines a unique plane of ${E}(p,p^\theta)$,which is fixed; see \cite[Proposition~5.3.9]{ADS-NSNS-HVM}). Hence the unique point $x'$ on $ax$ and on the imaginary line determined by $y'$ and ${y'}^\theta$ is fixed. If $x\neq x'$, then, since $\theta$ acts linearly (as it poinwise fixes $E(p,p^\theta)$), it fixes all points of $ax$ and the arguments in the proof of Lemma~\ref{fixperpequators} then imply that all points collinear to $x$ are fixed. Otherwise, $x=x'$ and $E(x,y')$ is fixed pointwise, again implying that at least one line through  $x$ is fixed pointwise, and hence, as before,  all of them are.  Also as before in the proof of Lemma~\ref{fixperpequators}, it follows now that all points collinear or symplectic to $x$ are fixed and so $\theta$ is a central elation with centre $x$. 
\end{proof}


\begin{lemma}\label{collpointlinefixed}
If a point $p$ is mapped onto a collinear point $p^\theta\neq p$, then the line $pp^\theta$ is fixed under $\theta$. 
\end{lemma}

\begin{proof}
In $\Res_\Delta(p^\theta)$ we find a line $L$ opposite both $pp^\theta$ and $(pp^\theta)^\theta$. Then any point $x$ on $L^{\theta^{-1}}$ distinct from $p$ is mapped onto an opposite. Then $E(x,x^\theta)$ is fixed pointwise.  If $\mathrm{char}(\FF)\neq 2$, the points $x$ and $x^\theta$ are the only points of $\widehat{E}(x,x^\theta)$ symplectic to all points of $E(x,x^\theta)$. It follows that $(x^\theta)^\theta=x$. If $\mathrm{char}(\FF)=2$, the latter follows directly from Lemma~\ref{perpor3}. Also, each line through $x$ is mapped onto its projection on $x^\theta$, so $xp$ is mapped onto $x^\theta p^\theta$ and vice versa, and so $(p^\theta)^\theta=p$ and the line $pp^\theta$ is fixed. 
\end{proof}

\begin{lemma}\label{lem:new2}
Suppose $\mathrm{char}(\FF)\neq 2$, or that $\mathrm{char}(\FF)=2$ and that for at least one point $x$ mapped onto an opposite the fixed point structure induced by $\theta$ in $\widehat{E}(x,x^\theta)$ is a non-degenerate polar space. Then necessarily $\mathrm{char}(\FF)\neq 2$ and there exists an apartment $\Sigma$ of $\Delta$ fixed pointwise by $\theta$. Also, $\theta$ fixes some panel pointwise and hence is a homology.
\end{lemma}

\begin{proof}
Let $p,q$ be two fixed (by $\theta$) opposite points such that $\theta$ fixes pointwise a geometric hyperplane $H$ of $\widehat{E}(p,q)$, see Lemma~\ref{geomhypeq}. If $\mathrm{char}(\FF)=2$, we may assume that $H$ is a nondegenerate polar space of Witt index 3. If $\mathrm{char}(\FF)\neq 2$, the hyperplane is not singular (if a collineation fixes all points collinear to a given point of a parabolic polar space, then it is the identity). Hence $H$ is a subquadric either of Witt index 4 (of type $\sD_4$),  or of Witt index 3. Now assume for a contradiction that $H$ has Witt index 3 (in either characteristic).

Let $\pi$ be a plane of $\widehat{E}(p,q)$. Then by \cite[Corollary 5.3.7]{ADS-NSNS-HVM} there is a unique point $p_U$ collinear to each $3$-space $U$ of  $\widehat{E}(p,q)$ containing $\pi$. Moreover, Proposition~5.3.9 of \cite{ADS-NSNS-HVM} implies that the set of all such points $p_U$, for $U$ running through all $3$-spaces of  $\widehat{E}(p,q)$ containing $\pi$, forms a line $L_\pi$ of $\Delta$. Since $\pi$ is fixed, the line $L_\pi$ is also fixed by $\theta$. Note that, since no $3$-space of  $\widehat{E}(p,q)$ is fixed, no point on $L_\pi$ is fixed. Select any point $x\in\pi$. Then also the plane spanned by $L_\pi$ and $x$ is fixed by $\theta$. Hence an arbitrary point $z\in\pi\setminus(\{x\}\cup L_\pi)$ is either fixed or mapped to a collinear point. In the former case the point $L_\pi\cap xz$ is fixed, in the latter case the point $L_\pi\cap zz^\theta$ is fixed (use Lemma~\ref{collpointlinefixed}), twice the same contradiction. 

Hence $H$ has Witt index 4 (and $\mathrm{char}(\FF)\neq 2$). Select two opposite points $x_0,x_1$ in  $\widehat{E}(p,q)$, fixed by $\theta$. Select two opposite points $y_0,y_1\in\{x_0,x_1\}^{\perp\!\!\!\perp}$, again fixed by $\theta$. Set $\xi_i:=\xi(x_i,y_i)$, $i=1,2$. Then $\xi_0$ is opposite $x_1$. Select lines $L_0,L_1$ and planes $\pi_0,\pi_1$ in  ${E}(x_0,x_1)$ such that $y_i\in L_i\subseteq \pi_i$, $i=0,1$, $L_0$ is opposite $L_1$, whereas $\pi_0$ is opposite $\pi_1$ in  the polar space $\widehat{E}(p,q)$, and $L_0,L_1,\pi_0,\pi_1$ are fixed by $\theta$. To $L_i$ and $\pi_i$ correspond a plane $\alpha_i$ and a line $K_i$ through $x_i$, with $K_i\subseteq\alpha_i\subseteq\xi_i$, $i=0,1$, and $K_0$ is opposite $K_1$, whereas $\alpha_0$ is opposite $\alpha_1$. Hence the chambers $\{x_0,K_0,\alpha_0,\xi_0\}$ and $\{x_1,K_1,\alpha_1,\xi_1\}$ are fixed and opposite and so they uniquely determine a pointwise fixed apartment $\Sigma$. 

With the notation of the previous paragraph, it is clear that $\theta$ fixes all lines in $\pi_0$ through $y_0$. Hence the panel determined by $\{x_0,K_0,\xi_0\}$ is fixed pointwise by $\theta$. 
\end{proof}

The proof of Theorem~\ref{thm:F41hom} is now complete. We note the following corollary (completing a postponed part of the proof of Theorem~\ref{thm:short}). Let $\varphi'$ be the highest short root of the $\sF_4$ root system.

\begin{cor}\label{cor:F4longandshort}
In $\sF_4(\FF)$, the element $x_{\varphi'}(a)x_{\varphi}(b)$ with $a,b\neq 0$ has opposition diagram $\sF_{4;2}$. If $\mathrm{char}(\FF)\neq 2$ then $x_{\varphi'}(a)$ has this diagram too, and if $\mathrm{char}(\FF)=2$ then it has diagram $\sF_{4;1}^4$. 
\end{cor}

\begin{proof}
By the \textit{angle} between root elations $x_{\alpha}(a)$ and $x_{\beta}(b)$ we shall mean the angle between the roots $\alpha$ and $\beta$. We claim that:
\begin{compactenum}[$(1)$]
\item Every product of a short root elation and a long root elation at angle $\pi/4$ is conjugate to a product of two perpendicular long root elations. 
\item If $\mathrm{char}(\FF)\neq 2$ then every short root elation is conjugate to a product of two perpendicular long root elations. 
\end{compactenum}
To begin with, note that the Weyl group is transitive on pairs $(\alpha,\beta)$ with $\alpha$ short, $\beta$ long, and $\mathrm{angle}(\alpha,\beta)=\pi/4$. An example of such a pair is $(\varphi',\varphi)$, and it is sufficient to show that the stabiliser $W_{\sC_3}$ of $\varphi$ is transitive on the set of short roots with angle $\pi/4$ with $\varphi$. These $6$ roots are $(1110),(1111),(1121),(1221),(1231),(1232)$, and an easy check shows that $W_{\sC_3}$ is indeed transitive on these roots. 

Thus to prove (1) we may assume that $\theta=x_{\alpha}(a)x_{\beta}(b)$ with $\alpha=(1110)$, $\beta=(1000)$, and $a,b\neq 0$. By commutator relations
\begin{align*}
x_{\alpha}(a)x_{\beta}(b)&=x_{\gamma}(ab^{-1})^{-1}x_{\delta}(a^2b^{-1})x_{\beta}(b)x_{\gamma}(ab^{-1}),
\end{align*}
where $\gamma=(0110)$ and $\delta=(1220)$. This proves (1), as $(\delta,\beta)$ is a pair of perpendicular long roots. 

Now suppose that $\theta$ is a short root elation. After conjugating we may assume that $\theta=x_{\varphi'}(a)$ for some $a\neq 0$. Let $\epsilon=(1110)$. By commutator relations we have
$$
x_{\epsilon}(a^{-1}/2)\theta x_{\epsilon}(a^{-1}/2)^{-1}=x_{\varphi'}(a)x_{\varphi}(1).
$$
Note that $\varphi'$ is short, $\varphi$ is long, and $\mathrm{angle}(\varphi',\varphi)=\pi/4$, and so applying the first statement completes the proof of the claims. 

The statement of the Corollary now follows from Theorem~\ref{thm:pocopolarclass}, and the fact that if $\mathsf{char}(\FF)=2$ then the $\sF_{4,4}(\FF)$ geometry isometrically embeds into the $\sF_{4,1}(\FF)$ geometry, with short root elations becoming long root elations. 
\end{proof}

We now collect the results to prove Theorem~\ref{thm:F4class}.

\begin{proof}[Proof of Theorem~\ref{thm:F4class}]
All that remains is to prove the statements concerning the number of conjugacy classes. If $\theta$ has diagram $\sF_{4;1}^1$ then $\theta$ is conjugate to $x_{\varphi}(a)$ for some $a\in\FF$ (by Theorem~\ref{thm:polarclass}). Then
 $
 h_{\omega_1-\omega_2}(a)x_{\varphi}(a)h_{\omega_1-\omega_2}(a)^{-1}=x_{\varphi}(1),
 $
 and so all long root elations are conjugate. 
 
If $\theta$ has diagram $\sF_{4;1}^4$, and $\mathrm{char}(\FF)\neq 2$, then $\theta$ is conjugate to $h_{\omega_4}(-1)$ (by Theorem~\ref{thm:F41hom} and Theorem~\ref{thm:homologyF4}). If $\mathrm{char}(\FF)=2$ then $\theta$ is conjugate to $x_{\varphi'}(1)$ (by Lemmas~\ref{perpcase} and~\ref{lem:new2}). 
 
Thus suppose that $\theta$ has diagram $\sF_{4;2}$. By the proof of Corollary~\ref{cor:F4longandshort} we may assume that $\theta=x_{\varphi}(a)x_{\varphi_{\sC_3}}(b)$, and conjugating by a diagonal element we may take $a=1$. If $b=c^2$ is a square then
 $$
 h_{3\omega_1-2\omega_2}(c)x_{\varphi}(1)x_{\varphi_{\sC_3}}(c^2)h_{3\omega_1-2\omega_2}(c)^{-1}=x_{\varphi}(1)x_{\varphi_{\sC_3}}(1).
 $$
This shows that there is at most one conjugacy class for each element of the quotient $\FF^{\times}/(\FF^{\times})^2$ (where $(\FF^{\times})^2=\{x^2\mid x\in\FF^{\times}\}$). Thus if $\FF$ is quadratically closed then we conclude that there is a unique conjugacy class of automorphisms with diagram $\sF_{4;2}$.  Similarly if $\FF$ is finite and $\mathrm{char}(\FF)=2$ then there is a unique conjugacy class (as every element is a square). Finally, if $\FF$ is finite with $\mathrm{char}(\FF)\neq 2$ we conclude that there are at most $2$ conjugacy classes, and by the tables in \cite{LS:12} the elements $x_{\varphi}(1)x_{\varphi_{\sC_3}}(1)$ and $x_{\varphi}(1)x_{\varphi_{\sC_3}}(b)$ with $b\notin(\FF^{\times})^2$ are not conjugate, and so there are precisely $2$ classes. 
\end{proof}

\subsection{Classification of domestic automorphisms of thick $\sE_6$ buildings}

The classification of domestic automorphisms of the small $\sE_6$ building (with $\FF=\FF_2$) is given in \cite[Theorems~4.6 and~4.6]{PVM:19b}, and the classification of domestic dualities of large $\sE_6$ buildings is given in~\cite{HVM:13}. By Theorem~\ref{thm:polarclass} the collineations with diagram~${^2}\sE_{6;1}$ are root elations, and so all that remains to complete the classification of domestic automorphisms of thick $\sE_6$ buildings is to classify the collineations of large $\sE_6$ buildings with diagram~${^2}\sE_{6;2}$.

We begin with some setup. Let $\varphi$ be the highest root of $\sE_6$, and let $\varphi'=(101111)$ be the highest root of the $\sA_5$ subsystem. Let $\Phi_1$ be the $\sA_3$ subsystem generated by the simple roots $\alpha_3,\alpha_4,\alpha_5$, and let $\varphi''=(001110)$ be the highest root of $\Phi_1$. 

Let $\beta_1=(100000)$, $\beta_2=(010000)$, $\beta_3=(000001)$, $\beta_4=(111100)$, $\beta_5=(010111)$, and $\beta_6=(111111)$, and for $1\leq i\leq 6$ let $C_i=\{\alpha\in \Phi^+\mid \alpha-\beta_i\in \ZZ_{\geq 0}\alpha_3+\ZZ_{\geq 0}\alpha_4+\ZZ_{\geq 0}\alpha_5\}$. Explicitly we have
\begin{align*}
C_1&=\{(100000),(101000),(101100),(101110)\}\\
C_2&=\{(010000),(010100),(011100),(010110),(011110),(011210)\}\\
C_3&=\{(000001),(000011),(000111),(001111)\}\\
C_4&=\{(111100),(111110),(111210),(112210)\}\\
C_5&=\{(010111),(011111),(011211),(011221)\}\\
C_6&=\{(111111),(111211),(112211),(111221),(112221),(112321)\}.
\end{align*}
Note that $\{C_i\mid 1\leq i\leq 6\}$ is a partition of $\Phi^+\backslash(\Phi_1\cup\{\varphi,\varphi'\})$. For each $i=1,\ldots,6$ let
$$
C_i'=(C_i \cup C_{i+1}\cup\cdots\cup C_6)\backslash\{\beta_i,\beta_i+\varphi''\}.
$$
The following technical lemma is required. 

\begin{lemma}\label{lem:claim4}
Let $1\leq i\leq 6$, and suppose that 
$
u=u'x_{\beta_i+\varphi''}(b)x_{\beta_i}(a)
$
with $a\neq 0$, $b\in\FF$, and $u'$ a product of root elations with roots in $C_i'$. Let $z_1,z_2,c,d\in\FF$. If $z_2\neq 0$ and $bz_2\neq -1$, then 
$$
x_{-\beta_i-\varphi''}(z_1)x_{-\varphi}(1)x_{-\varphi'}(1)x_{\varphi}(c)x_{\varphi'}(d)ux_{-\beta_i-\varphi''}(z_2)\in Bw_0s_4B.
$$
\end{lemma}

\begin{proof}
We make the following two claims:
\begin{compactenum}[$(1)$]
\item If $z_1,z_1'\in\FF$ then $Bx_{-\beta_i-\varphi''}(z_1)x_{-\varphi}(1)x_{-\varphi'}(1)x_{\varphi}(c)x_{\varphi'}(d)x_{-\beta_i-\varphi''}(z_1')B= Bs_{\varphi}s_{\varphi'}B$;
\item With $z_3=(b+z_2^{-1})^{-1}$, we have $ux_{-\beta_i-\varphi''}(z_2)B\in x_{-\beta_i-\varphi''}(z_3)Bs_{\varphi''}B$. 
\end{compactenum}
The result follows, because writing $g=x_{-\beta_i-\varphi''}(z_1)x_{-\varphi}(1)x_{-\varphi'}(1)x_{\varphi}(c)x_{\varphi'}(d)ux_{-\beta_i-\varphi''}(z_2)$, and using (2) followed by (1) gives
\begin{align*}
g&\in x_{-\beta_i-\varphi''}(z_1)x_{-\varphi}(1)x_{-\varphi'}(1)x_{\varphi}(c)x_{\varphi'}(d)x_{-\beta_i-\varphi''}(z_3)Bs_{\varphi''}B\\
&\subseteq Bs_{\varphi}s_{\varphi'}B\cdot Bs_{\varphi''}B\\
&=Bs_{\varphi}s_{\varphi'}s_{\varphi''}B=Bw_0s_4B,
\end{align*}
where for the final two equalities we have used Lemma~\ref{lem:long2}.

We first prove (1). Let $X=Bx_{-\beta_i-\varphi''}(z_1)x_{-\varphi}(1)x_{-\varphi'}(1)x_{\varphi}(c)x_{\varphi'}(d)x_{-\beta_i-\varphi''}(z_1')B$. We will write $x_{\alpha}(\cdot)$ as shorthand for an element of $U_{\alpha}$ with the convention that $x_{\alpha}(\cdot)=1$ if $\alpha\notin\Phi$. By commutator relations, and replacing $x_{-\varphi}(1)=x_{\varphi}(1)s_{\varphi}^{-1}x_{\varphi}(1)$ and similarly for $x_{-\varphi'}(1)$, and noting that the roots $\varphi$ and $\varphi'$ are perpendicular, we compute
\begin{align*}
X&=Bx_{-\varphi}(1)x_{-\varphi'}(1)x_{-\varphi'-\beta_i-\varphi''}(\cdot)x_{-\beta_i-\varphi''}(\cdot)x_{\varphi}(c)x_{\varphi'}(d)x_{-\beta_i-\varphi''}(z_1')B\\
&=Bs_{\varphi}s_{\varphi'}x_{\varphi}(1)x_{\varphi'}(1)x_{-\varphi'-\beta_i-\varphi''}(\cdot)x_{-\beta_i-\varphi''}(\cdot)x_{\varphi'-\beta_i-\varphi''}(\cdot)x_{\varphi-\beta_i-\varphi''}(\cdot)x_{\varphi+\varphi'-\beta_i-\varphi''}(\cdot)B.
\end{align*}
Now, by inspection both $\varphi-\beta_i-\varphi''$ and $\varphi+\varphi'-\beta_i-\varphi''$ are either not roots, or are positive roots, and hence these terms can be absorbed into $B$. Thus 
\begin{align*}
X&=Bs_{\varphi}s_{\varphi'}x_{\varphi}(1)x_{\varphi'}(1)x_{-\varphi'-\beta_i-\varphi''}(\cdot)x_{-\beta_i-\varphi''}(\cdot)x_{\varphi'-\beta_i-\varphi''}(\cdot)B.
\end{align*}
Moreover, note that $\varphi'-\beta_i-\varphi''$ is a negative root if and only if $i=6$ (in all other cases it is either not a root, or is a positive root, and hence can be absorbed into $B$). Also note that $-\varphi'-\beta_i-\varphi''$ is only a root for $i=2$. 

We now push the $x_{\varphi}(1)x_{\varphi'}(1)$ term to the right, to be absorbed into the $B$. In the case $i\neq 2,6$ this gives $X=Bs_{\varphi}s_{\varphi'}x_{-\beta_i-\varphi''}(\cdot)B$, 
and then since $\beta_i+\varphi''\in \Phi(s_{\varphi}s_{\varphi'})=\Phi^+\backslash\Phi_1$ the term $x_{-\beta_i-\varphi''}(\cdot)$ is absorbed into the $B$ on the left hand side after moving past $s_{\varphi}s_{\varphi'}$. For $i=2$ we compute $
X=Bs_{\varphi}s_{\varphi'}x_{-\varphi'-\beta_i-\varphi''}(\cdot)x_{-\beta_i-\varphi''}(\cdot)B$, 
and since $\varphi'+\beta_i+\varphi'',\beta_i+\varphi''\in\Phi(s_{\varphi}s_{\varphi'})$ the result again follows. Similarly, for $i=6$ we compute  $X=Bs_{\varphi}s_{\varphi'}x_{-\beta_i-\varphi''}(\cdot)x_{\varphi'-\beta_i-\varphi''}(\cdot)B$. In this case $\varphi'-\beta_i-\varphi''=-(011110)$, and so both roots $\beta_i+\varphi''$ and $-\varphi'+\beta_i+\varphi''$ are in $\Phi(s_{\varphi}s_{\varphi'})$, and the result follows as before. Hence the claim.

We now prove (2). With $u'$ as in the statement of the theorem, we have
\begin{align*}
ux_{-\beta_i-\varphi''}(z_2)B&=u'x_{\beta_i+\varphi''}(b)x_{\beta_i}(a)x_{-\beta_i-\varphi''}(z_2)B\\
&=u'x_{\beta_i+\varphi''}(b)x_{\beta_i}(a)x_{\beta_i+\varphi''}(z_2^{-1})s_{\beta_i+\varphi''}B\\
&=u'x_{\beta_i+\varphi''}(b+z_2^{-1})x_{\beta_i}(a)s_{\beta_i+\varphi''}B\\
&=u'x_{\beta_i+\varphi''}(b+z_2^{-1})s_{\beta_i+\varphi''}x_{-\varphi''}(\pm a)B,
\end{align*} 
where we have used the facts that $2\beta_i+\varphi''\notin\Phi$ and $s_{\beta_i+\varphi''}(\beta_i)=-\varphi''$ for all $1\leq i\leq 6$. Recalling that $z_3^{-1}=b+z_2^{-1}$, using the folding relation we have
\begin{align*}
ux_{-\beta_i-\varphi''}(z_2)B&=u'x_{-\beta_i-\varphi''}(z_3)s_{\beta_i+\varphi''}(z_3^{-1})x_{-\beta_i-\varphi''}(z_3)s_{\beta_i+\varphi''}x_{-\varphi''}(\pm a)B\\
&=u'x_{-\beta_i-\varphi''}(z_3)s_{\beta_i+\varphi''}(z_3^{-1})s_{\beta_i+\varphi''}x_{\beta_i+\varphi''}(\pm z_3)x_{-\varphi''}(\pm a)B\\
&=u'x_{-\beta_i-\varphi''}(z_3)s_{\beta_i+\varphi''}(z_3^{-1})s_{\beta_i+\varphi''}x_{-\varphi''}(\pm a)B\\
&=u'x_{-\beta_i-\varphi''}(z_3)x_{-\varphi''}(z_4)B\\
&=u'x_{-\beta_i-\varphi''}(z_3)x_{\varphi''}(z_4^{-1})s_{\varphi''}B
\end{align*}
for some $z_4\neq 0$. 

By commutator relations, one may write
$$
u'x_{-\beta_i-\varphi''}(z_3)=x_{-\beta_i-\varphi''}(z_3)u''
$$
where $u''$ is a product of elements from root subgroups $U_{\gamma}$ with $\gamma\in Y_i$, where $Y_i=\{\alpha,\alpha-\beta_i-\varphi''\mid \alpha\in C_i'\}\cap \Phi$. By inspection of the elements of $C_i'$, if $\gamma\in Y_i$ is a negative root then necessarily $\gamma\in\Phi_1$. Moreover one sees that 
\begin{align*}
Y_i\cap \Phi_1=
\begin{cases}
\{-\alpha_5,-\alpha_4-\alpha_5\}&\text{if $i=1,5$}\\
\{-\alpha_3,-\alpha_3-\alpha_4\}&\text{if $i=3,4$}\\
\{-\alpha_5,-\alpha_3,\alpha_4\}&\text{if $i=2,6$}.
\end{cases}
\end{align*}
It follows that $u''$ can be written in the form
$
u''=u'''y_i,
$
where $u'''\in U^+$ and $y_i\in \prod_{\alpha\in Y_i\cap \Phi_1} U_{\alpha}$. Simple calculations in $\sA_3$, using the fact that $\Phi(s_{\varphi''})=\Phi_1^+\backslash\{\alpha_4\}$, show that $y_ix_{\varphi''}(z_4^{-1})s_{\varphi''}B\in Bs_{\varphi''}B$ for each $i=1,\ldots,6$. For example, if $i=1,5$ we have (for some $a_1,a_2\in\FF$)
\begin{align*}
y_ix_{\varphi''}(z_4^{-1})s_{\varphi''}B&=x_{-000010}(a_1)x_{-000110}(a_2)x_{001110}(z_4^{-1})s_{\varphi''}B\\
&=x_{-000010}(a_1)x_{001110}(z_4^{-1})x_{001000}(\pm a_2z_4^{-1})x_{-000110}(a_2)s_{\varphi''}B\\
&=x_{001110}(z_4^{-1})x_{001100}(\pm a_1z_4^{-1})x_{-000010}(a_1)x_{001000}(\pm a_2z_4^{-1})s_{\varphi''}B\\
&=x_{001110}(z_4^{-1})x_{001100}(\pm a_1z_4^{-1})x_{001000}(\pm a_2z_4^{-1})s_{\varphi''}B\subseteq Bs_{\varphi''}B.
\end{align*}
Hence (2) is proved.
\end{proof}

We can now complete the classification of domestic automorphisms of thick $\sE_6$ buildings. 

\begin{thm}\label{thm:converseE6}
Let $\theta$ be a collineation of a thick $\sE_6$ building $\Delta$. Then $\theta$ has opposition diagram ${^2}\sE_{6;2}$ if and only if $\theta$ is either a product of perpendicular root elations, or a nontrivial homology fixing a subbuilding with thick frame of type~$\sD_5$. 
\end{thm}

\begin{proof}
It is easy to see that if $\theta$ is a product of perpendicular root elations then $\theta$ is conjugate to an element $x_{\varphi}(a)x_{\varphi'}(b)$ with $a,b\neq 0$, and hence the ``if'' direction of the theorem follows from Theorem~\ref{thm:unipotentdiags} and Theorem~\ref{thm:homologyE6}. Thus it remains to prove the ``only if'' direction.

The result for the small building $\sE_6(2)$ follows from \cite[Theorem~4.6]{PVM:19b}. Thus suppose that $\Delta$ is large, and so the underlying field $\FF$ has at least $3$ elements. Note that $w_{\sA_3}w_{\sE_6}=s_{\varphi}s_{\varphi'}$ (by comparing inversion sets). Thus $\disp(\theta)=\ell(s_{\varphi}s_{\varphi'})$, and since $\theta$ is capped $\ell(\delta(gB,\theta gB))=\ell(s_{\varphi}s_{\varphi'})$ if and only if $\delta(gB,\theta gB)=s_{\varphi}s_{\varphi'}$. Moreover, after replacing $\theta$ by a conjugate, we may assume that the base chamber $B$ is mapped to Weyl distance $s_{\varphi}s_{\varphi'}$ (see \cite[Theorem~2.6]{PVM:19a}). Since the stabiliser of $B$ is transitive on each $w$-sphere centred at~$B$ we may assume that $B$ is mapped to the chamber $x_{\varphi}(1)x_{\varphi'}(1)s_{\varphi}s_{\varphi'}B$. By the folding relation, and the fact that $\varphi$ and $\varphi'$ are perpendicular, we have $x_{\varphi}(1)x_{\varphi'}(1)s_{\varphi}s_{\varphi'}B=x_{-\varphi}(1)x_{-\varphi'}(1)B$. The condition $\theta(B)=x_{-\varphi}(1)x_{-\varphi'}(1)B$ gives
$$
\theta=x_{-\varphi}(1)x_{-\varphi'}(1)uh\sigma\quad\text{for some $u\in U^+$, $h\in H$, and $\sigma\in\mathrm{Aut}(\FF)$}. 
$$
We will now determine $u,h$ and $\sigma$. The primary strategy (as in Theorem~\ref{thm:converse}) is to show that if these elements do not take certain particular forms, then one can find an element $g\in G$ such that $g^{-1}\theta g\in BwB\sigma$ with $\ell(w)>\ell(s_{\varphi}s_{\varphi'})$. Thus the chamber $gB$ is mapped to distance $\ell(w)>\ell(s_{\varphi}s_{\varphi'})$, contradicting the fact that $\theta$ has polar-copolar diagram. A useful observation is that if $w=s_{\varphi}s_{\varphi'}v$ with $v\in W_{\sA_3}$ then $\ell(w)=\ell(s_{\varphi}s_{\varphi'})+\ell(v)$. We now proceed with the analysis. The first three steps are completely analogous to the proof of Theorem~\ref{thm:converse}, and we omit the easy details.
\medskip

\noindent\textit{Claim 1: We have $u\in U_{\Phi\backslash\Phi_1}^+$}. 
\medskip

\noindent\textit{Claim 2: We have $h=h_{\omega_1}(c_1)h_{\omega_2}(c_2)h_{\omega_6}(c_3)$ for some $c_1,c_2,c_3\in\mathbb{F}^{\times}$}. 
\medskip

\noindent\textit{Claim 3: We have $\sigma=\mathrm{id}$}. 
\medskip

\noindent\textit{Claim 4: We have $u\in U_{\varphi}U_{\varphi'}$}. Write $u=x_{\varphi}(c)x_{\varphi'}(d)u_1$ with $c,d\in\FF$ and $u_1\in U_{\Phi\backslash(\Phi_1\cup\{\varphi,\varphi'\})}^+$. We must show that $u_1=1$. Suppose not. Let $i\in\{1,2,3,4,5,6\}$ be minimal such that a root in $C_i$ appears in $u_1$. Using the same idea as in the proof of Claim~4 in Theorem~\ref{thm:converse}, one may then conjugate $u_1$ by an element $g\in U_{\Phi_1}^-$ to obtain
$$
g^{-1}u_1g=u_1'x_{\beta_i+\varphi''}(b)x_{\beta_i}(a)
$$
with $a\neq 0$, $b\in\FF$, and $u_1'$ a product of root elations with roots in $C_i'$. Moreover, since $g^{-1}hg=h$ (by Claim 2) and $g^{-1}x_{-\varphi}(1)x_{-\varphi'}(1)x_{\varphi}(c)x_{\varphi'}(d)g=x_{-\varphi}(1)x_{-\varphi'}(1)x_{\varphi}(c)x_{\varphi'}(d)$ (because $\pm\varphi-\alpha,\pm\varphi'-\alpha\notin\Phi$ for all $\alpha\in\Phi_1^+$) we have
$$
g^{-1}\theta g=x_{-\varphi}(1)x_{-\varphi'}(1)x_{\varphi}(c)x_{\varphi'}(d)u_1'x_{\beta_i+\varphi''}(b)x_{\beta_i}(a)h.
$$
Let $g_1=x_{-\beta_i-\varphi''}(z)$ with $z\in\FF$. Then
\begin{align*}
g_1^{-1}g^{-1}\theta gg_1&=x_{-\beta_i-\varphi''}(-z)x_{-\varphi}(1)x_{-\varphi'}(1)x_{\varphi}(c)x_{\varphi'}(d)u_1'x_{\beta_i+\varphi''}(b)x_{\beta_i}(a)x_{-\beta_i-\varphi''}(\lambda z)h
\end{align*}
for some $\lambda\neq 0$. Since $\Delta$ is large we have $|\FF|\geq 3$, and so we may choose $z\in\FF$ with $z\neq 0$ and $b\lambda z\neq -1$. Applying Lemma~\ref{lem:claim4} then gives
$$
g_1^{-1}g^{-1}\theta gg_1\in Bw_0s_4B,
$$
and hence the chamber $g_1gB$ is mapped to Weyl distance $w_0s_4$, a contradiction. Thus $u_1=1$. 
\medskip

\noindent\textit{Claim 5: We have $\theta=x_{-\varphi}(1)x_{-\varphi'}(1)x_{\varphi}(a)x_{\varphi'}(b)h_{\omega_1}(c_1)h_{\omega_2}(c_2)h_{\omega_6}(c_3)$, where}
\begin{align*}
a=-(c_1c_2-1)(c_2c_3-1)\quad\text{and}\quad b=-(c_1-1)(c_3-1).
\end{align*}
From Claims 1--4 we have $\theta=x_{-\varphi}(1)x_{-\varphi'}(1)x_{\varphi}(a)x_{\varphi'}(b)h_{\omega_1}(c_1)h_{\omega_2}(c_2)h_{\omega_6}(c_3)$ for some $a,b\in\FF$. Let $\alpha=(101000)$ and $\beta=(001111)$, and let $g=x_{-\alpha}(1)x_{-\beta}(1)$. A direct calculation shows that if $b\neq -(c_1-1)(c_3-1)$ then 
$$
g^{-1}\theta g\in Bs_{\varphi}s_{\varphi'}B\cdot Bs_3B=Bs_{\varphi}s_{\varphi'}s_3B,
$$
a contradiction. 

A similar calculation, with $g=x_{-\gamma}(1)x_{-\delta}(1)$ where $\gamma=(010100)$ and $\delta=(112321)$ shows that if $a\neq-(c_1c_2-1)(c_2c_3-1)$ then $g^{-1}\theta g\in Bs_{\varphi}s_{\varphi'}s_4B$, again a contradiction.
\medskip

\noindent\textit{Claim 6: $\theta$ fixes a chamber.} Let $gB=x_{\varphi}(y_1)x_{\varphi'}(y_2)s_{\varphi}s_{\varphi'}B$ with $y_1,y_2\in\FF$. With $\theta$ as in Claim~5 we have
\begin{align*}
\theta gB&=x_{-\varphi}(1)x_{-\varphi'}(1)x_{\varphi}(a)x_{\varphi'}(b)x_{\varphi}(c_1c_2^2c_3y_1)x_{\varphi'}(c_1c_3y_2)s_{\varphi}s_{\varphi'}B\\
&=x_{-\varphi}(1)x_{\varphi}(a+c_1c_2^2c_3y_1)x_{-\varphi'}(1)x_{\varphi'}(b+c_1c_3y_2)s_{\varphi}s_{\varphi'}B.
\end{align*}
We now use the $\mathsf{SL}_2$ relation $x_{-\alpha}(z)x_{\alpha}(z')=x_{\alpha}(z'(1+zz')^{-1})x_{-\alpha}(-z(1+zz'))h_{\alpha^{\vee}}(1+zz')^{-1}$, valid whenever $zz'\neq -1$, to obtain
\begin{align*}
\theta gB&=x_{\varphi}((a+c_1c_2^2c_3y_1)(1+a+c_1c_2^2c_3y_1)^{-1})x_{\varphi'}((b+c_1c_3y_2)(1+b+c_1c_3y_2)^{-1})s_{\varphi}s_{\varphi'}B,
\end{align*}
where we have chosen $y_1,y_2$ so that $1+a+c_1c_2^2c_3y_1\neq 0$ and $1+b+c_1c_3y_2\neq 0$. We have $\theta gB=gB$ if and only if $(a+c_1c_2^2c_3y_1)=y_1(1+a+c_1c_2^2c_3y_1)$ and $(b+c_1c_3y_2)=y_2(1+b+c_1c_3y_2)$. Recalling that $a=-(c_1c_2-1)(c_2c_3-1)$ and $b=-(c_1-1)(c_3-1)$, the discriminants in the quadratics (in $y_1$ and $y_2$) are perfect squares, and we find that $gB$ is fixed if and only if 
\begin{align*}
y_1&\in\{1-c_1^{-1}c_2^{-1},1-c_2^{-1}c_3^{-1}\}&y_2&\in\{1-c_1^{-1},1-c_3^{-1}\}
\end{align*}
(note that these values avoid the excluded values $y_1=1-c_1^{-1}c_2^{-1}-c_2^{-1}c_3^{-1}$ and $y_2=1-c_1^{-1}-c_3^{-1}$ required for the $\mathsf{SL}_2$ relation to hold). 

\medskip

\noindent\textit{Claim 7: $\theta$ is conjugate to $x_{\varphi}(-c_1c_2)x_{\varphi'}(-c_1)h_{\omega_1}(c_1c_3^{-1})$.} Let $g=x_{\varphi}(1-c_1^{-1}c_2^{-1})x_{\varphi'}(1-c_1^{-1})s_{\varphi}s_{\varphi'}$. It follows from Claim 5 that $g^{-1}\theta g\in B$, and a direct calculation yields
\begin{align*}
g^{-1}\theta g&=x_{\varphi}(-c_1c_2)x_{\varphi'}(-c_1)h_{\omega_1}(c_1c_3^{-1})
\end{align*}
as required.

\medskip

\noindent\textit{Claim 8: $\theta$ is either a product of perpendicular root elations, or is a homology fixing a subbuilding with thick frame of type $\sD_5$.} After conjugating, as in Claim 6, we may replace $\theta$ by the conjugate $\theta=x_{\varphi}(-c_1c_2)x_{\varphi'}(-c_1)h_{\omega_1}(c_1c_3^{-1})$. If $c_1=c_3$ then $\theta$ is a product of perpendicular root elations. If $c_1\neq c_3$ then we note that the chamber $gB$, with $g=x_{\varphi}(c_1c_2c_3/(c_1-c_3))x_{\varphi'}(c_1c_3/(c_1-c_3))s_{\varphi}s_{\varphi'}$, is fixed by $\theta$. Thus $g^{-1}\theta g\in B$, and a direct calculation shows that 
$$
g^{-1}\theta g=h_{\omega_6}(c_1^{-1}c_3),
$$
which is a homology fixing a subbuilding with thick frame of type $\sD_5$ (by Theorem~\ref{thm:homologyE6}). 
\end{proof}

\subsection{Classification of domestic automorphisms of split $\sG_2$ buildings}\label{sec:G2}

Let $\Phi$ be a root system of type $\sG_2$, with the convention that $\alpha_1$ is a short root. Note that $\alpha_2$ is the polar node. Let $\varphi=(32)$ be the highest root of $\Phi$. The highest short root is $\varphi'=(21)$. Let $G$ be the Chevalley group $\sG_2(\FF)$. The building $\Delta=G/B$ may be regarded as a generalised hexagon (the \textit{dual split Cayley hexagon}), where the point set is $G/P_1$ and the line set is $G/P_2$, where $P_i=B\cup Bs_iB$ (in particular, note that ``points'' are type $\{2\}=S\backslash\{1\}$ vertices, and ``lines'' are type $\{1\}=S\backslash\{2\}$ vertices). Then $G\rtimes \mathrm{Aut}(\FF)$ is the full collineation group of $\Delta$.

Recall the following definitions. A \textit{distance $3$-ovoid} in a generalised hexagon $\mathscr{G}$ is a set $\cS$ of mutually opposite points such that every element of the generalised hexagon is at distance at most $3$ (in the incidence graph) from a point of~$\cS$. A subhexagon $\mathscr{G}'$ of $\mathscr{G}$ is \textit{full} if every point of $\mathscr{G}$ incident with a line of $\mathscr{G}'$ belongs to $\mathscr{G}'$, and is \textit{large} if every element of $\mathscr{G}$ is at distance at most $3$ from some element of $\mathscr{G}'$. Recall from \cite[Theorem~2.7]{PTM:15} that a nontrivial automorphism $\theta$ of a generalised hexagon is point-domestic if and only if its fixed element structure is either a ball of radius $3$ in the incidence graph centred at a line, a large full subhexagon, or a distance $3$-ovoid. Dual statements hold for line-domestic automorphisms. (Note that the ``if'' direction was omitted in the statement of \cite[Theorem~2.7]{PTM:15}, however it is obvious from fixed element structure in each case).

In this section we prove the following theorem, which along with \cite[Theorem~2.7, 2.9, and 2.14]{PTM:15} completely classifies the domestic automorphisms of $\Delta$. 

\begin{thm}\label{thm:mainG2}
Let $\Delta$ be the building of $\sG_2(\FF)$. There exists a unique conjugacy class $\mathscr{C}_1$ of automorphisms with opposition diagram $\sG_{2;1}^2$, and a unique conjugacy class $\mathscr{C}_2$ of automorphisms with opposition diagram $\sG_{2;1}^1$. The class $\mathscr{C}_1$ consists precisely of the long root elations. The class $\mathscr{C}_2$ has representative $\theta_2$, where:
\begin{compactenum}[$(1)$]
\item if $\mathrm{char}(\FF)=3$ then $\theta_2=x_{\varphi'}(1)$ is a short root elation fixing precisely a ball of radius $3$ in the incidence graph centred at a line,
\item if there is $z\in\FF\backslash\{1\}$ with $z^3=1$ then $\theta_2=h_{\omega_1}(z)$ is a homology fixing a large full subhexagon, and
\item if $\mathrm{char}(\FF)\neq 3$ and there is no $z\in\FF\backslash\{1\}$ with $z^3=1$ then $\theta_2=x_{\alpha_1}(1)s_1$ fixes precisely a distance $3$-ovoid. 
\end{compactenum}
\end{thm}

In particular, note that in the case~(3) there is no automorphism fixing a chamber with opposition diagram $\sG_{2;1}^1$. This is the only case of a split building with this property. The proof of Theorem~\ref{thm:mainG2} will follow from a series of lemmas.

\begin{lemma}\label{lem:rootelations1}
Let $\Delta=\sG_2(\FF)$. A collineation $\theta$ fixes precisely a ball of radius~$3$ in the incidence graph centred at a point (respectively a line) if and only if $\theta$ is conjugate to $x_{\varphi}(1)$ (respectively, $\mathrm{char}(\FF)=3$ and $\theta$ is conjugate to $x_{\varphi'}(1)$). Moreover, the automorphism $x_{\varphi}(1)$ is domestic with opposition diagram~$\sG_{2;1}^2$, and the automorphism $x_{\varphi'}(1)$ is domestic if and only if $\mathrm{char}(\FF)=3$, in which case it has opposition diagram $\sG_{2;1}^1$. 
\end{lemma}

\begin{proof}
Consider the case where the centre is a line~$L$, and let $p$ be a point on $L$. Then $\{p,L\}$ is a chamber fixed by $\theta$, and after conjugation we may suppose that $\{p,L\}$ is the base chamber~$B$. Since $\theta$ fixes all points on the line $L$ it follows that $\theta$ is linear, and hence $\theta\in \sG_2(\FF)$. In the $BN$-pair language, the hypothesis of the lemma gives that $\theta$ fixes each chamber $gB$ with $g\in B\cup Bs_1B\cup Bs_2B\cup Bs_1s_2B\cup Bs_2s_1B\cup Bs_2s_1s_2B$. In particular, $\theta(B)=B$, and so $\theta\in B$. Thus $\theta=hu$ with $h=h_{\omega_1}(c_1)h_{\omega_2}(c_2)\in H$ and $u\in U^+$. Write 
$$
u=x_{10}(a_1)x_{01}(a_2)x_{11}(a_3)x_{21}(a_4)x_{31}(a_5)x_{32}(a_6).
$$
By commutator relations and the fact that $s_1(\alpha)\in\Phi^+$ for all $\alpha\in \Phi^+\backslash\{\alpha_1\}$ we have
$$
\theta x_{\alpha_1}(a)s_1B=x_{\alpha_1}((a+a_1)c_1)s_1B,
$$
and since $\theta$ fixes each chamber $x_{\alpha_1}(a)s_1B$ with $a\in\FF$ we have $a_1=0$ and $c_1=1$. Similarly, since $\theta x_{\alpha_2}(a)s_2B=x_{\alpha_2}((a+a_2)c_2)s_2B$ we have $a_2=0$ and $c_2=1$. Thus $h=1$. Next, since $\theta$ fixes each chamber $x_{10}(a)s_1s_2B$ with $a\in\FF$ we have
\begin{align*}
x_{31}(a)s_1s_2B&=\theta x_{31}(a)s_1s_2B=x_{31}(a+a_5)s_1s_2B,
\end{align*}
and so $a_5=0$. Similarly, since the chamber $x_{11}(a)s_2s_1B$ is fixed for all $a\in\FF$ we have $a_3=0$. Next, the condition that $x_{32}(a)s_2s_1s_2B$ is fixed for all $a\in\FF$ yields $a_6=0$, and so $\theta=x_{\varphi'}(a_4)$. However also the chamber $x_{11}(a)s_2s_1s_2B$ must be fixed, and a calculation gives
$$
\theta x_{11}(a)s_2s_1s_2B=x_{11}(a)x_{32}(3a_4a)s_2s_1s_2B.
$$
Thus $\mathrm{char}(\FF)=3$. We may then conjugate $h_{\omega_2}(a_4^{-1})\theta h_{\omega_2}(a_4)=x_{\varphi'}(1)$, and the result now follows in this case by Theorem~\ref{thm:short}.

In the case that the centre is a point, a very similar analysis to the above gives $\theta=x_{\varphi}(a)$ for some $a\neq 0$, and this element is conjugate to $x_{\varphi}(1)$. Since this element is central in $U^+$ it is clear that $\theta$ does indeed fix a ball of radius~$3$, and moreover by Theorem~\ref{thm:RGD} it has opposition diagram~$\sG_{2,1}^2$. 
\end{proof}

\begin{lemma}\label{lem:G211}
There exists a collineation of $\Delta$ fixing precisely a distance $3$-ovoid if and only if the polynomial $z^2+z+1$ is irreducible over $\FF$. Moreover, if $z^2+z+1$ is irreducible over $\FF$ then there is a unique conjugacy class of collineations fixing an ovoid, with representative $\theta=x_{\alpha_1}(1)s_1$. 
\end{lemma}

\begin{proof}
Suppose that $\theta$ fixes a distance $3$-ovoid (and hence is point-domestic). Since the automorphism group of $\Delta$ is transitive on pairs of opposite vertices we may assume that the points $P_1$ and $w_0P_1$ are both fixed. Write $\theta=\theta_0\cdot\sigma$, with $\theta\in G$ and $\sigma\in\mathrm{Aut}(\FF)$. We have $\theta P_1=\theta_0 P_1$, and so $\theta_0\in P_1=B\cup Bs_1B$. If $\theta_0\in B$ then $\theta$ fixes the chamber $B$, a contradiction. Thus $\theta_0\in Bs_1B$. Each element of $Bs_1B$ can be written as $x_{\alpha_1}(a)s_1b$ with $b\in B$. Replacing $\theta$ by the conjugate $b^{\sigma}\theta b^{-\sigma}$ we may write 
$$
\theta=hu_1x_{\alpha_1}(a)s_1\cdot\sigma,
$$
where $h=h_{\omega_1}(c_1)h_{\omega_2}(c_2)$ with $c_1,c_2\in\FF^{\times}$, $u_1\in \prod_{\alpha\in \Phi^+\backslash\{\alpha_1\}}U_{\alpha}$, and $a\in\FF$. 

\noindent\textit{Claim 1: $u_1=1$.} Since $\theta$ fixes $w_0P_1$, and since $\sigma$ fixes $w_0$, we have
$
w_0^{-1}\theta_0 w_0\in P_1.
$ However
\begin{align*}
w_0^{-1}\theta_0w_0&=(w_0^{-1}hw_0)(w_0^{-1}u_1w_0)(w_0^{-1}x_{\alpha_1}(a)s_1w_0),
\end{align*}
and since the first and third grouped terms are in $P_1$ we have $w_0^{-1}u_1w_0\in P_1=B\cup Bs_1B$. If $w_0^{-1}u_1w_0\in B$ then since $w_0^{-1}u_1w_0$ is also in the opposite Borel we have $u_1\in H$, and so $u_1=1$. If $w_0^{-1}u_1w_0\in Bs_1B$ then $u_1w_0\in w_0Bs_1B=U^-w_0s_1B=U^-s_2s_1s_2s_1s_2B$. However, writing
$$
u_1w_0B=x_{10}(0)x_{31}(d_1)x_{21}(d_2)x_{32}(d_3)x_{11}(d_4)x_{01}(d_5)s_1s_2s_1s_2s_1s_2B,
$$
the folding algorithm described in \cite[\S7]{PRS:09} implies that
\begin{compactenum}[--]
\item if $d_1\neq 0$ then $u_1w_0B\subseteq U^-vB$ for some $v\in \{1,s_1,s_2,s_1s_2,s_2s_1,s_2s_1s_2\}$;
\item if $d_1= 0$ and $d_2\neq 0$ then $u_1w_0B\subseteq U^-s_2B$;
\item if $d_1=d_2=0$ and $d_3\neq 0$ then $u_1w_0B\subseteq U^-s_1B$;
\item if $d_1=d_2=d_3=0$ and $d_4\neq 0$ then $u_1w_0B\subseteq U^-s_1s_2s_1B$;
\item if $d_1=d_2=d_3=d_4=0$ and $d_5\neq 0$ then $u_1w_0B\subseteq U^-s_1s_2s_1s_2s_1B$.
\end{compactenum}
This contradiction shows that indeed $w_0^{-1}u_1w_0\in B$, and hence $u_1=1$. 
\smallskip

\noindent\textit{Claim 2: $c_1^3c_2^2=1$ and $\sigma=1$.} We have $\theta=hx_{\alpha_1}(a)s_1\cdot\sigma$, where $h=h_{\omega_1}(c_1)h_{\omega_2}(c_2)$. Consider the chamber $x_{\varphi}(z)w_0B$, where $z\in\FF$. We have 
\begin{align*}
w_0^{-1}x_{\varphi}(-z)\theta x_{\varphi}(z)w_0&=w_0^{-1}x_{\varphi}(-z)hx_{\alpha_1}(a)s_1x_{\varphi}(z^{\sigma})w_0\\
&=w_0^{-1}hx_{\varphi}(-zc_1^{-3}c_2^{-2})x_{\alpha_1}(a)x_{\varphi}(z^{\sigma})s_1w_0\\
&\in Bw_0^{-1}x_{\varphi}(-zc_1^{-3}c_2^{-2}+z^{\sigma})w_0P_1\\
&=Bx_{-\varphi}(zc_1^{-3}c_2^{-2}-z^{\sigma})P_1.
\end{align*}   
If $zc_1^{-3}c_2^{-2}-z^{\sigma}\neq 0$ then $w_0^{-1}x_{\varphi}(-z)\theta x_{\varphi}(z)w_0\in Bs_{\varphi}P_1$, and since $s_{\varphi}=s_2s_1s_2s_1s_2$ it follows that the point of the chamber $x_{\varphi}(z)w_0B$ is mapped to an opposite point, a contradiction. Thus $zc_1^{-3}c_2^{-2}-z^{\sigma}=0$ for all $z\in\FF$. The case $z=1$ gives $c_1^3c_2^2=1$, and then the equation reads $z-z^{\sigma}=0$, and so $\sigma$ is trivial. 
\smallskip

\noindent\textit{Claim 3: We may assume that either $a=0$ or $a=1$.} Suppose that $a\neq 0$. Then, since $s_1(\omega_1)=-\omega_1+3\omega_2$ we have
\begin{align*}
h_{\omega_1}(a)^{-1}\theta h_{\omega_1}(a)&=h_{\omega_1}(a^{-1}c_1)h_{\omega_2}(c_2)x_{\alpha_1}(a)h_{-\omega_1+3\omega_2}(a)s_1\\
&=h_{\omega_1}(a^{-1}c_1)h_{\omega_2}(c_2)h_{-\omega_1+3\omega_2}(a)x_{\alpha_1}(1)s_1\\
&=h_{\omega_1}(a^{-2}c_1)h_{\omega_2}(a^3c_2)x_{\alpha_1}(1)s_1.
\end{align*}
Renaming $a^{-2}c_1$ and $a^3c_2$ by $c_1$ and $c_2$ gives the result.
\smallskip

\noindent\textit{Claim 4: We have $\mathrm{char}(\FF)\neq 3$.} Suppose that $\mathrm{char}(\FF)=3$. Suppose first that $a=0$. Then a calculation shows that the point
$
x_{01}(1)x_{31}(1)s_2s_1s_2s_1s_2P_1
$
is mapped to the point $x_{01}(-c_2)x_{32}(1)x_{31}(c_2^{-1})s_2s_1s_2s_1s_2P_1$, and that these points are opposite, a contradiction. 

Now suppose that $a=1$. Note that since $c_1^3c_2^2=1$ we have $h_{\omega_1}(c_1)h_{\omega_2}(c_2)=h_{\alpha_1^{\vee}}(c)$ where $c=c_1^{-1}c_2^{-1}$. Thus $\theta=h_{\alpha_1^{\vee}}(c)x_{10}(1)s_1$. Let $u=x_{01}(c^{-3})x_{11}(-c^{-1})x_{21}(1)x_{31}(-1)$. Then if $c\neq 1$ we calculate
$$
P_1w_0^{-1}u^{-1}\theta uw_0P_1=P_1s_2s_1s_2s_1s_2P_1,
$$
showing that the point $uw_0P_1$ is mapped to an opposite point, a contradiction. Thus $c=1$, and hence $c_1=c_2=1$. But then
\begin{align*}
\theta=x_{10}(1)s_1=x_{10}(2)x_{-10}(-1)x_{10}(1)=x_{10}(-1)x_{-10}(-1)x_{10}(1),
\end{align*}  
and hence $\theta$ is conjugate to a root elation, a contradiction. 
\smallskip

\noindent\textit{Claim 5: We have $a\neq 0$ and $c_1=c_2=1$.} Recall that $\theta=hx_{\alpha_1}(a)s_1$, with $h_{\omega_1}(c_1)h_{\omega_2}(c_2)$ where $c_1^3c_2^2=1$, and $a\in\{0,1\}$. Consider the chamber $gB$ where $g=x_{11}(1)w_0$. Then
\begin{align*}
g^{-1}\theta g&=w_0^{-1}x_{11}(-1)hx_{10}(a)s_1x_{11}(1)w_0\\
&\in Bw_0^{-1}x_{11}(-c_1^{-1}c_2^{-1})x_{10}(a)x_{21}(-1)s_1w_0B\\
&=Bw_0^{-1}x_{11}(-c_1^{-1}c_2^{-1})x_{21}(-1)x_{31}(-3a)s_2s_1s_2s_1s_2B.
\end{align*}
Write $u=x_{11}(-c_1^{-1}c_2^{-1})x_{21}(-1)x_{31}(-3a)$. We will show below that if either $a= 0$ or $c_1c_2\neq 1$ then 
$
us_2s_1s_2s_1s_2B\subseteq U^-B=w_0Bw_0B,
$
which in turn gives $g^{-1}\theta g\in Bw_0B$, a contradiction with point-domesticity. To prove the above statement, we compute
\begin{align*}
us_2s_1s_2s_1s_2B&=x_{-11}(-c_1c_2)x_{11}(c_1^{-1}c_2^{-1})n_{11}(-c_1^{-1}c_2^{-1})x_{21}(-1)x_{31}(-3a)s_2s_1s_2s_1s_2B\\
&\subseteq U^-x_{11}(c_1^{-1}c_2^{-1})x_{10}(-c_1c_2)x_{31}(-3a)s_1s_2B.
\end{align*}
Commutator relations give
\begin{align*}
x_{11}(c_1^{-1}c_2^{-1})x_{10}(-c_1c_2)x_{31}(-3a)&=x_{10}(-c_1c_2)x_{31}(3(c_1c_2-a))x_{11}(c_1^{-1}c_2^{-1})x_{32}(3c_1^{-1}c_2^{-1})x_{21}(2).
\end{align*}
Using this in the above equation, and noting that the final three terms can pass past the $s_1s_2$ and remain in~$B$, we have
\begin{align*}
us_2s_1s_2s_1s_2B&\subseteq U^-x_{10}(-c_1c_2)x_{31}(3(c_1c_2-a))s_1s_2B\\
&=U^-x_{-10}(-c_1^{-1}c_2^{-1})x_{10}(c_1c_2)n_{10}(-c_1c_2)x_{31}(3(c_1c_2-a))s_1s_2B\\
&=U^-x_{10}(c_1c_2)x_{01}(3c_1^{-3}c_2^{-3}(c_1c_2-a))s_2B\\
&=U^-x_{01}(3c_2^{-1}(c_1c_2-a))s_2B.
\end{align*}
Recall that $\mathrm{char}(\FF)\neq 3$. Thus if either $a=0$ , or if $a=1$ and $c_1c_2\neq 1$, we have $us_2s_1s_2s_1s_2B\subseteq U^-B$ as required.
\smallskip

\noindent\textit{Claim 6: The polynomial $z^2+z+1$ is irreducible over $\FF$.} Since $\theta$ fixes no lines, in particular no line through $P_1$ is fixed. Equivalently, none of the chambers $B$ and $x_{\alpha_1}(-z)s_1B$ with $z\in\FF$ are fixed by $\theta$. Clearly $B$ is not fixed, and for $z\neq 0$ we have
\begin{align*}
\theta x_{\alpha_1}(-z)s_1B&=x_{\alpha_1}(1)s_1x_{\alpha_1}(-z)s_1B\\
&=x_{\alpha_1}(1)x_{-\alpha_1}(z)B\\
&=x_{\alpha_1}(1)x_{\alpha_1}(z^{-1})s_1B\\
&=x_{\alpha_1}(1+z^{-1})s_1B,
\end{align*}
and so $1+z^{-1}\neq -z$ for all $z\in\FF^{\times}$. Hence the claim.
\smallskip

So far we have proved that if $\theta$ fixes a distance $3$-ovoid, then $\theta$ is conjugate to $x_{\alpha_1}(1)s_1$ and the polynomial $z^2+z+1=0$ is irreducible over $\FF$. It remains to show that if this polynomial is irreducible then the element $\theta=x_{\alpha_1}(1)s_1$ does indeed fix a distance $3$-ovoid. We have shown that no line through the fixed point $P_1$ is fixed. This implies that the fixed element structure of $\theta$ consists of a set of mutually opposite points (for otherwise by projecting onto $P_1$ gives a fixed line through $P_1$). Thus it is  sufficient to prove point-domesticity to deduce that the fixed structure is a distance $3$-ovoid. 

By Lemma~\ref{lem:red2} it is sufficient to show that no point opposite the base point is mapped onto an opposite point. These points are of the form
$$
gs_2s_1s_2s_1s_2P_1=x_{01}(a_1)x_{11}(a_2)x_{32}(a_3)x_{21}(a_4)x_{31}(a_5)s_2s_1s_2s_1s_2P_1.
$$
We compute
$$
g^{-1}\theta g=x_{01}(a')x_{11}(b')x_{21}(b')x_{31}(3b'-a')x_{10}(1)s_1
$$
where $a'=-a_1-a_5$ and $b'=-a_2+a_4-a_5$, and hence
\begin{align*}
P_1s_{\varphi}^{-1}g^{-1}\theta g s_{\varphi}P_1&=P_1x_{-31}(-a')x_{-21}(b')x_{-11}(-b')x_{-01}(3b'-a')P_1.
\end{align*}
The folding algorithm then easily gives
\begin{align*}
P_1s_{\varphi}^{-1}g^{-1}\theta g s_{\varphi}P_1&\in\begin{cases}
P_1&\text{if $a'=b'=0$}\\
P_1s_2s_1s_2P_1&\text{otherwise,}
\end{cases}
\end{align*}
completing the proof.
\end{proof}

\begin{lemma}\label{lem:G211b}
There exists a collineation of $\Delta$ fixing precisely a large full subhexagon if and only if the equation $z^2+z+1$ has a solution $z\neq 1$ in $\FF$. Moreover, if $z^2+z+1=0$ with $z\neq 1$ then there is a unique conjugacy class of collineations fixing a large full subhexagon, with representative $\theta=h_{\omega_1}(z)$. 
\end{lemma}

\begin{proof}
We may, after conjugating, suppose that the base apartment is fixed. Moreover, as the fixed subhexagon is full, all points on a line are fixed, and so $\theta$ is linear. Thus $\theta=h_{\omega_1}(c_1)h_{\omega_2}(c_2)\in H$. But fullness means that all points on the line $L=P_2$ are fixed. These points are $P_1$ and $x_{\alpha_2}(a)s_2P_1$ for $a\in \FF$. Since 
$$
\theta x_{\alpha_2}(a)s_2P_1=x_{\alpha_1}(ac_2)s_2P_1
$$
we have $c_2=1$. Thus $\theta=h_{\omega_1}(c)$ for some $c\in\FF^{\times}$. Consider the chamber $gB=x_{\varphi}(1)w_0B$. Then
$$
Bg^{-1}\theta gB=Bw_0^{-1}x_{\varphi}(c^3-1)w_0^{-1}B.
$$
If $c^3-1\neq 0$ then this double coset is $Bs_{\varphi}B=Bs_2s_1s_2s_1s_2B$, which means that the point of this chamber is mapped to an opposite point, a contradiction. So $c^3-1=0$. If $c=1$ then $\theta$ is the identity, a contradiction (as $\theta$ maps a line to an opposite line). So $c^2+c+1=0$ has a solution $c\neq 1$ (in particular, $\mathrm{char}(\FF)\neq 3$). Finally, we show that in this case the automorphism $\theta=h_{\omega_1}(c)$ does indeed have opposition diagram $\sG_{2;1}^1$. It suffices to show that no point opposite the base point $P_1$ is mapped to an opposite point. Each such point is of the form $us_{\varphi}P_1$, where
$$
u=x_{01}(a_1)x_{11}(a_2)x_{32}(a_3)x_{21}(a_4)x_{31}(a_5).
$$
Commutator relations, and the fact that $c^3=1$, imply that
\begin{align*}
u^{-1}\theta u
&=x_{11}(a_2(c-1))x_{32}(3a_2a_4(1-c))x_{21}(a_4(c^2-1))h_{\omega_1}(c).
\end{align*}
Thus
\begin{align*}
P_1s_{\varphi}^{-1}u^{-1}\theta us_{\varphi}P_1&=P_1x_{-21}(a_2(c-1))x_{-32}(3a_2a_4(c-1))x_{-11}(a_4(1-c^2))P_1.
\end{align*}
Explicit computation then shows that $P_1s_{\varphi}^{-1}u^{-1}\theta us_{\varphi}P_1\in P_1\cup P_1s_2s_1s_2P_1$, completing the proof.
\end{proof}

\begin{proof}[Proof of Theorem~\ref{thm:mainG2}]
An automorphism with opposition diagram $\sG_{2,1}^1$ must fix either a ball of radius~$3$ centred at a line, a large full subhexagon, or a distance~$3$-ovoid, and the result follows from Lemmas~\ref{lem:rootelations1}, \ref{lem:G211} and~\ref{lem:G211b}. 

Consider the opposition diagram~$\sG_{2,1}^2$. The fixed element structure of an automorphism with this diagram is necessarily either a ball of radius~$3$ centred at a point, a large ideal subhexagon (where \textit{ideal} is the dual notion to \textit{full}), or a distance $3$-spread (the dual notion to a distance $3$-ovoid).

Large ideal subhexagons do not exist in the dual split Cayley hexagon, by \cite[Remark~5.9.14]{HVM:98}. 
Suppose that $\theta$ fixes a distance $3$-spread. We argue as in the proof of Lemma~\ref{lem:G211}. As in that proof, we have $\theta=hu_1x_{\alpha_2}(a)s_2\cdot \sigma$, where $h=h_{\omega_1}(c_1)h_{\omega_2}(c_2)$ with $c_1,c_2\in\FF^{\times}$, $u_1\in \prod_{\alpha\in \Phi^+\backslash\{\alpha_2\}}U_{\alpha}$, and $a\in\FF$. The proof of Claim 1 of Lemma~\ref{lem:G211} holds in an analogous fashion, yielding $u_1=1$. Following the argument of Claim~2, considering the chamber $gB=x_{\varphi'}(z)w_0B$, we have
$$
w_0^{-1}x_{\varphi'}(-z)\theta x_{\varphi'}(z)w_0\in Bx_{-\varphi'}(zc_1^{-2}c_1^{-1}-z^{\sigma})P_2,
$$
from which it follows that $zc_1^{-2}c_1^{-1}-z^{\sigma}=0$ for all $z\in\FF$. Setting $z=1$ gives $c_1^2c_2=1$, and then the equation reads $z-z^{\sigma}=0$, and so $\sigma$ is trivial. Thus $\theta=h_{\omega_1}(c)h_{\omega_2}(c^{-2})x_{\alpha_2}(a)s_2$ for some $c\in \FF^{\times}$ and $a\in \FF$. We may assume that either $a=0$, or that $a=1$ (by conjugating by $h_{\omega_2}(c^2a^{-1})$). If $a=0$, with $\mathrm{char}(\FF)\neq 2$, we see that the chamber $x_{10}(1)s_1s_2s_1B$ is mapped to an opposite. If $a=1$ with $\mathrm{char}(\FF)\neq 2$, we see that if $c\neq -1/2$ then $x_{10}(1)s_1s_2s_1B$ is mapped to an opposite, and if $c=-1/2$ then $x_{10}(1)x_{11}(1)s_1s_2s_1B$ is mapped to an opposite. If $\mathrm{char}(\FF)=2$ and $a=1$ then the chamber $x_{10}(1)s_1s_2s_1B$ is mapped onto an opposite chamber. 

Suppose $\mathrm{char}(\FF)=2$ and that $a=0$. Since no chamber is fixed, we compute (for $z\neq 0$)
\begin{align*}
\theta x_{01}(-z)s_2B&=h_{\omega_1}(c)h_{\omega_2}(c^{-2})s_2x_{01}(-z)s_2B\\
&=h_{\omega_1}(c)h_{\omega_2}(c^{-2})x_{-01}(z)B\\
&=x_{01}(z^{-1}c^{-2})s_2B.
\end{align*}
Thus $z^{-1}c^{-2}\neq -z$, and so $z^2\neq -c^{-2}$ for all $z\neq 0$. This is a contradiction (taking $z=c^{-1}$ and recalling that $\mathrm{char}(\FF)=2$). 
\end{proof}

%
%

We note that Corollary~\ref{cor:existence} follows from the above results:

\begin{proof}[Proof of Corollary~\ref{cor:existence}]
The polar closed admissible diagrams are realised as the opposition diagrams of unipotent elements, by Theorem~\ref{thm:unipotent}. The admissible diagram $\sF_{4;1}^4$ is achieved as the opposition diagram of either the homology $h_{\omega_4}(-1)$ (if $\mathrm{char}(\FF)\neq 2$) or the short root elation $x_{\varphi'}(1)$ (in the case $\mathrm{char}(\FF)=2$) by Lemma~\ref{lem:homologyF4} and Corollary~\ref{cor:F4longandshort}. All of these automorphisms lie in $B$, and hence fix the base chamber of $\Delta=G/B$. The remaining diagram is $\sG_{2;1}^1$, and the result follows from Theorem~\ref{thm:mainG2}.
\end{proof}

\subsection{Group theoretic consequences}\label{sec:conjugacyclasses}

We conclude with group theoretic applications of our results, proving Corollaries~\ref{cor:conj1} and~\ref{cor:conj2} from the introduction, and providing some further similar applications. 

\begin{proof}[Proof of Corollary~\ref{cor:conj1}]
Let $\mathscr{C}$ be a conjugacy class in~$G$. Suppose that $\mathscr{C}\cap Bw_0B=\emptyset$. Then no chamber is mapped onto an opposite chamber by any element of $\mathscr{C}$ (for if $\theta\cdot gB$ is opposite $gB$ then $g^{-1}\theta g\in Bw_0B$). Thus every element of $\mathscr{C}$ is domestic. By Theorems~\ref{thm:converse}, \ref{polar-copolar}, and~\ref{thm:F41hom} every domestic automorphism of a large building of type $\sE_6$ or $\sF_4$ necessarily fixes a chamber, and hence $\mathscr{C}\cap B\neq \emptyset$. Similarly, for the small buildings of type $\sE_6$ or $\sF_4$ every domestic automorphism fixes a chamber by \cite[Theorems~4.4 and~4.6]{PVM:19b}, hence the result. 
\end{proof} 

\begin{proof}[Proof of Corollary~\ref{cor:conj2}]
By Corollary~\ref{cor:existsdomestic} every Moufang spherical building, with the exception of a projective plane, admits a type preserving domestic automorphism (a long root elation). The conjugacy class of such an automorphism thus intersects trivially with $Bw_0B$. 
\end{proof} 

It follows, from Remark~\ref{rem:nonchamberfix} below, that the statement of Corollary~\ref{cor:conj1} fails in general for buildings of types $\sE_7$ and $\sE_8$ (as there exist domestic automorphisms fixing no chamber). In these cases we have the following results.

\begin{cor}\label{cor:E7conjugacy}
Let $G=G_{\Phi}(\FF)$ be a Chevalley group of type $\sE_7$ with $|\FF|>2$. Let $\varphi_1$ (respectively $\varphi_2,\varphi_3$) be the highest root of the $\sE_7$ (respectively $\sD_6$, $\sD_4$) subsystem. For each nontrivial conjugacy class $\mathscr{C}$ in $G$ we have both
\begin{align*}
&\mathscr{C}\cap (Bs_{\varphi_1}B\cup Bs_{\varphi_1}s_{\varphi_2}B\cup Bs_{\varphi_1}s_{\varphi_2}s_7B\cup Bs_{\varphi_1}s_{\varphi_2}s_{\varphi_3}s_3B\cup Bw_0B)\neq \emptyset,\text{ and}\\
&\mathscr{C}\cap(B\cup Bs_{\varphi_1}s_{\varphi_2}s_7B\cup Bs_{\varphi_1}s_{\varphi_2}s_{\varphi_3}s_3B\cup Bw_0B)\neq \emptyset.
\end{align*}
\end{cor}

\begin{proof}
If $\mathscr{C}$ is a nontrivial conjugacy class, then every element of $\mathscr{C}$ has non-empty opposition diagram. Let $\theta\in\mathscr{C}$. If $J=\Type(\theta)$ then by \cite[Theorem~2.6]{PVM:19a} there is a chamber $gB$ with $\delta(gB,\theta\cdot gB)=w_{S\backslash J}w_0$, and so $g^{-1}\theta g\in Bw_{S\backslash J}w_0B$. Hence $\mathscr{C}\cap Bw_{S\backslash J}w_0B\neq \emptyset$. The first statement then follows from the classification of admissible diagrams. The second statement follows from the fact that, by Theorem~\ref{polar-copolar}, every automorphism with opposition diagram either $\sE_{7;2}$ or $\sE_{8;2}$ necessarily fixes a chamber, and hence is conjugate to an element of~$B$.
\end{proof}

In the case $\FF=\FF_2$, there exist uncapped automorphisms, and some additional double cosets need to be included in the statement of Corollary~\ref{cor:E7conjugacy}.

\begin{cor}\label{cor:E8conjugacy}
Let $G=G_{\Phi}(\FF)$ be a Chevalley group of type $\sE_8$ with $|\FF|>2$. Let $\varphi_1$ (respectively $\varphi_2,\varphi_3$) be the highest root of the $\sE_8$ (respectively $\sE_7$, $\sD_6$) subsystem. For each nontrivial conjugacy class $\mathscr{C}$ in $G$ we have both
\begin{align*}
&\mathscr{C}\cap (Bs_{\varphi_1}B\cup Bs_{\varphi_1}s_{\varphi_2}B\cup Bs_{\varphi_1}s_{\varphi_2}s_{\varphi_3}s_7B\cup Bw_0B)\neq \emptyset,\text{ and}\\
&\mathscr{C}\cap(B\cup Bs_{\varphi_1}s_{\varphi_2}s_{\varphi_3}s_7B\cup Bw_0B)\neq \emptyset.
\end{align*}
\end{cor}

\begin{proof}
The proof is similar to Corollary~\ref{cor:E7conjugacy}. 
\end{proof}

We conclude with the following remark.

\begin{remark}\label{rem:nonchamberfix}
The results of this paper (along with \cite{HVM:13}) completely classify the automorphisms of split spherical buildings of exceptional type with opposition diagrams other than $\sE_{7;3}$, $\sE_{7;4}$, and $\sE_{8;4}$ (with the possible exception of polar-copolar diagrams for small $\sE_7$ and $\sE_8$ buildings). For these three excluded diagrams we have provided examples (both unipotent elements and homologies) with the respective opposition diagram, and the general classification for these diagrams will be continued in~\cite{PVM:20c}. For now we simply state, without proof, that in both $\sE_7$ and $\sE_8$ the element $\theta=x_{\alpha_1}(1)s_1$ has order~$3$ and opposition diagram $\sE_{7;4}$ or $\sE_{8;4}$ (for all fields). Thus, for example, if $\FF=\FF_2$ then $\theta$ necessarily does not fix a chamber (as the Borel is a Sylow $2$-group in this case), and hence is neither in $U^+$ nor in $H$. Furthermore, for type $\sE_7$ the element $\theta=x_{\alpha_2}(1)s_2x_{\alpha_5}(1)s_5x_{\alpha_7}(1)s_7$ has has order $3$ and opposition diagram $\sE_{7;3}$, and similar statements apply. 
\end{remark}

\begin{appendix}

\section{Root system data}\label{app:data}

The following table lists the number of positive roots (equivalently, the length of the longest element), the highest root~$\varphi$, the highest short root~$\varphi'$, the polar type~$\wp$, the dual polar type~$\wp'$, and the copolar type~$\wp^*$ for each irreducible crystallographic root system. Note that the copolar type is not well defined in the $\sB_n$ and $\sD_n$ cases as $S\backslash \wp$ is not irreducible. 

$$
\begin{array}{|c||c|c|c|c|c|c|}
\hline
&|\Phi^+|&\varphi&\varphi'&\wp&\wp'&\wp^*\\
\hline\hline
\sA_n&n(n+1)/2&(111\cdots111)&-&\{1,n\}&\{1,n\}&\{2,n-1\}\,\,(n\geq 3)\\
\hline
\sB_n&n^2&(122\cdots222)&(111\cdots111)&\{2\}&\{1\}&-\\
\hline
\sC_n&n^2&(222\cdots221)&(122\cdots221)&\{1\}&\{2\}&\{2\}\\
\hline
\sD_n&n(n-1)&(122\cdots211)&-&\{2\}&\{2\}&-\\
\hline
\sE_6&36&(122321)&-&\{2\}&\{2\}&\{1,6\}\\
\hline
\sE_7&63&(2234321)&-&\{1\}&\{1\}&\{6\}\\
\hline
\sE_8&120&(23465432)&-&\{8\}&\{8\}&\{1\}\\
\hline
\sF_4&24&(2342)&(1232)&\{1\}&\{4\}&\{4\}\\
\hline
\sG_2&6&(32)&(21)&\{2\}&\{1\}&\{1\}\\
\hline
\end{array}
$$

\noindent The positive roots of the $\sE_6$ root system are as follows.
\begin{align*}
&1 0 0 0 0 0&&0 1 0 0 0 0&&0 0 1 0 0 0&&0 0 0 1 0 0&&0 0 0 0 1 0&&0 0 0 0 0 1&&1 0 1 0 0 0&&0 1 0 1 0 0&&0 0 1 1 0 0&&0 0 0 1 1 0\\
&0 0 0 0 1 1&&1 0 1 1 0 0&&0 1 1 1 0 0&&0 1 0 1 1 0&&0 0 1 1 1 0&&0 0 0 1 1 1&&1 1 1 1 0 0&&1 0 1 1 1 0&&0 1 1 1 1 0&&0 1 0 1 1 1\\
&0 0 1 1 1 1&&1 1 1 1 1 0&&1 0 1 1 1 1&&0 1 1 2 1 0&&0 1 1 1 1 1&&1 1 1 2 1 0&&1 1 1 1 1 1&&0 1 1 2 1 1&&1 1 2 2 1 0&&1 1 1 2 1 1\\
&0 1 1 2 2 1&&1 1 2 2 1 1&&1 1 1 2 2 1&&1 1 2 2 2 1&&1 1 2 3 2 1&&1 2 2 3 2 1
\end{align*}
\noindent The positive roots of the $\sE_7$ root system are as follows.
\begin{align*}
&1 0 0 0 0 0 0&&
0 1 0 0 0 0 0&&
0 0 1 0 0 0 0&&
0 0 0 1 0 0 0&&
0 0 0 0 1 0 0&&
0 0 0 0 0 1 0&&
0 0 0 0 0 0 1&&
1 0 1 0 0 0 0&&
0 1 0 1 0 0 0\\
&0 0 1 1 0 0 0&&
0 0 0 1 1 0 0&&
0 0 0 0 1 1 0&&
0 0 0 0 0 1 1&&
1 0 1 1 0 0 0&&
0 1 1 1 0 0 0&&
0 1 0 1 1 0 0&&
0 0 1 1 1 0 0&&
0 0 0 1 1 1 0\\
&0 0 0 0 1 1 1&&
1 1 1 1 0 0 0&&
1 0 1 1 1 0 0&&
0 1 1 1 1 0 0&&
0 1 0 1 1 1 0&&
0 0 1 1 1 1 0&&
0 0 0 1 1 1 1&&
1 1 1 1 1 0 0&&
1 0 1 1 1 1 0\\
&0 1 1 2 1 0 0&&
0 1 1 1 1 1 0&&
0 1 0 1 1 1 1&&
0 0 1 1 1 1 1&&
1 1 1 2 1 0 0&&
1 1 1 1 1 1 0&&
1 0 1 1 1 1 1&&
0 1 1 2 1 1 0&&
0 1 1 1 1 1 1\\&
1 1 2 2 1 0 0&&
1 1 1 2 1 1 0&&
1 1 1 1 1 1 1&&
0 1 1 2 2 1 0&&
0 1 1 2 1 1 1&&
1 1 2 2 1 1 0&&
1 1 1 2 2 1 0&&
1 1 1 2 1 1 1&&
0 1 1 2 2 1 1\\&
1 1 2 2 2 1 0&&
1 1 2 2 1 1 1&&
1 1 1 2 2 1 1&&
0 1 1 2 2 2 1&&
1 1 2 3 2 1 0&&
1 1 2 2 2 1 1&&
1 1 1 2 2 2 1&&
1 2 2 3 2 1 0&&
1 1 2 3 2 1 1\\&
1 1 2 2 2 2 1&&
1 2 2 3 2 1 1&&
1 1 2 3 2 2 1&&
1 2 2 3 2 2 1&&
1 1 2 3 3 2 1&&
1 2 2 3 3 2 1&&
1 2 2 4 3 2 1&&
1 2 3 4 3 2 1&&
2 2 3 4 3 2 1
\end{align*}
\noindent The positive roots of the $\sE_8$ root system are as follows.
\begin{align*}
&1 0 0 0 0 0 0 0&&
0 1 0 0 0 0 0 0&&
0 0 1 0 0 0 0 0&&
0 0 0 1 0 0 0 0&&
0 0 0 0 1 0 0 0&&
0 0 0 0 0 1 0 0&&
0 0 0 0 0 0 1 0&&
0 0 0 0 0 0 0 1\\&
1 0 1 0 0 0 0 0&&
0 1 0 1 0 0 0 0&&
0 0 1 1 0 0 0 0&&
0 0 0 1 1 0 0 0&&
0 0 0 0 1 1 0 0&&
0 0 0 0 0 1 1 0&&
0 0 0 0 0 0 1 1&&
1 0 1 1 0 0 0 0\\&
0 1 1 1 0 0 0 0&&
0 1 0 1 1 0 0 0&&
0 0 1 1 1 0 0 0&&
0 0 0 1 1 1 0 0&&
0 0 0 0 1 1 1 0&&
0 0 0 0 0 1 1 1&&
1 1 1 1 0 0 0 0&&
1 0 1 1 1 0 0 0\\&
0 1 1 1 1 0 0 0&&
0 1 0 1 1 1 0 0&&
0 0 1 1 1 1 0 0&&
0 0 0 1 1 1 1 0&&
0 0 0 0 1 1 1 1&&
1 1 1 1 1 0 0 0&&
1 0 1 1 1 1 0 0&&
0 1 1 2 1 0 0 0\\&
0 1 1 1 1 1 0 0&&
0 1 0 1 1 1 1 0&&
0 0 1 1 1 1 1 0&&
0 0 0 1 1 1 1 1&&
1 1 1 2 1 0 0 0&&
1 1 1 1 1 1 0 0&&
1 0 1 1 1 1 1 0&&
0 1 1 2 1 1 0 0\\&
0 1 1 1 1 1 1 0&&
0 1 0 1 1 1 1 1&&
0 0 1 1 1 1 1 1&&
1 1 2 2 1 0 0 0&&
1 1 1 2 1 1 0 0&&
1 1 1 1 1 1 1 0&&
1 0 1 1 1 1 1 1&&
0 1 1 2 2 1 0 0\\&
0 1 1 2 1 1 1 0&&
0 1 1 1 1 1 1 1&&
1 1 2 2 1 1 0 0&&
1 1 1 2 2 1 0 0&&
1 1 1 2 1 1 1 0&&
1 1 1 1 1 1 1 1&&
0 1 1 2 2 1 1 0&&
0 1 1 2 1 1 1 1\\&
1 1 2 2 2 1 0 0&&
1 1 2 2 1 1 1 0&&
1 1 1 2 2 1 1 0&&
1 1 1 2 1 1 1 1&&
0 1 1 2 2 2 1 0&&
0 1 1 2 2 1 1 1&&
1 1 2 3 2 1 0 0&&
1 1 2 2 2 1 1 0\\&
1 1 2 2 1 1 1 1&&
1 1 1 2 2 2 1 0&&
1 1 1 2 2 1 1 1&&
0 1 1 2 2 2 1 1&&
1 2 2 3 2 1 0 0&&
1 1 2 3 2 1 1 0&&
1 1 2 2 2 2 1 0&&
1 1 2 2 2 1 1 1\\&
1 1 1 2 2 2 1 1&&
0 1 1 2 2 2 2 1&&
1 2 2 3 2 1 1 0&&
1 1 2 3 2 2 1 0&&
1 1 2 3 2 1 1 1&&
1 1 2 2 2 2 1 1&&
1 1 1 2 2 2 2 1&&
1 2 2 3 2 2 1 0\\&
1 2 2 3 2 1 1 1&&
1 1 2 3 3 2 1 0&&
1 1 2 3 2 2 1 1&&
1 1 2 2 2 2 2 1&&
1 2 2 3 3 2 1 0&&
1 2 2 3 2 2 1 1&&
1 1 2 3 3 2 1 1&&
1 1 2 3 2 2 2 1\\&
1 2 2 4 3 2 1 0&&
1 2 2 3 3 2 1 1&&
1 2 2 3 2 2 2 1&&
1 1 2 3 3 2 2 1&&
1 2 3 4 3 2 1 0&&
1 2 2 4 3 2 1 1&&
1 2 2 3 3 2 2 1&&
1 1 2 3 3 3 2 1\\&
2 2 3 4 3 2 1 0&&
1 2 3 4 3 2 1 1&&
1 2 2 4 3 2 2 1&&
1 2 2 3 3 3 2 1&&
2 2 3 4 3 2 1 1&&
1 2 3 4 3 2 2 1&&
1 2 2 4 3 3 2 1&&
2 2 3 4 3 2 2 1\\&
1 2 3 4 3 3 2 1&&
1 2 2 4 4 3 2 1&&
2 2 3 4 3 3 2 1&&
1 2 3 4 4 3 2 1&&
2 2 3 4 4 3 2 1&&
1 2 3 5 4 3 2 1&&
2 2 3 5 4 3 2 1&&
1 3 3 5 4 3 2 1\\&
2 3 3 5 4 3 2 1&&
2 2 4 5 4 3 2 1&&
2 3 4 5 4 3 2 1&&
2 3 4 6 4 3 2 1&&
2 3 4 6 5 3 2 1&&
2 3 4 6 5 4 2 1&&
2 3 4 6 5 4 3 1&&
2 3 4 6 5 4 3 2
\end{align*}
\noindent The positive roots of the $\sF_4$ root system are as follows.
\begin{align*}
&1 0 0 0&&
0 1 0 0&&
0 0 1 0&&
0 0 0 1&&
1 1 0 0&&
0 1 1 0&&
0 0 1 1&&
1 1 1 0&&
0 1 2 0&&
0 1 1 1&&
1 1 2 0&&
1 1 1 1\\&
0 1 2 1&&
1 2 2 0&&
1 1 2 1&&
0 1 2 2&&
1 2 2 1&&
1 1 2 2&&
1 2 3 1&&
1 2 2 2&&
1 2 3 2&&
1 2 4 2&&
1 3 4 2&&
2 3 4 2
\end{align*}
\noindent The positive roots of the $\sG_2$ root system are $1 0$, $0 1$, $1 1$, $2 1$, $3 1$, $3 2$.

\end{appendix}

\bibliographystyle{plain}

\end{document}